\documentclass{amsart}
\usepackage[top=1in,bottom=1in,left=1in,right=1in]{geometry}
\usepackage{euscript}
\usepackage{xcolor}
\usepackage{ytableau}
\usepackage{amssymb}
\usepackage{tikz}
\usepackage{tikz-cd}
\usepackage{array} 
\usepackage[colorinlistoftodos]{todonotes}
\usepackage{combelow}
\usepackage{hyperref}

\renewcommand{\triangle}{\vartriangle}

\newcommand{\Z}{\mathbb{Z}}
\newcommand{\Q}{\mathbb{Q}}
\newcommand{\C}{\mathbb{C}}
\newcommand{\K}{\mathbb{K}}

\newcommand{\Y}{\mathbb{Y}}
\newcommand{\F}{\mathbb{F}}

\newcommand{\qqq}{\mathfrak{q}}
\newcommand{\ddd}{\mathfrak{d}}

\newcommand{\cC}{\mathcal{C}}
\newcommand{\cP}{\mathcal{P}}

\newcommand{\bd}{b^\bullet}
\newcommand{\cd}{c^\bullet}

\newcommand{\kb}{\kappa}
\newcommand{\La}{\Lambda}
\newcommand{\la}{\lambda}
\newcommand{\lad}{{\lambda^\bullet}}
\newcommand{\mud}{{\mu^\bullet}}

\newcommand{\Nd}{{N_\bullet}}

\newcommand{\Xd}{X^\bullet}

\newcommand{\charge}{\mathrm{charge}}

\newcommand{\core}{\mathrm{core}}

\newcommand{\ES}{\mathrm{ES}}
\newcommand{\id}{\mathrm{id}}
\newcommand{\quot}{\mathrm{quot}}

\newcommand{\shape}{\mathrm{shape}}
\newcommand{\Sym}{\mathfrak{S}}

\newcommand{\bigp}{\mathrm{big}}

\newcommand{\Ju}{{\underline{J}}}

\newcommand{\fS}{\mathfrak{S}}

\newcommand{\dt}{\cdot}

\providecommand{\blacktriangle}{\mdblktriangle}

\newcommand{\comment}[1]{}

\newtheorem*{ithm}{Theorem}

\newtheorem{thm}{Theorem}

\newtheorem{lem}[thm]{Lemma}
\newtheorem{prop}[thm]{Proposition}
\newtheorem{cor}[thm]{Corollary} 
\theoremstyle{remark}
\newtheorem{rem}[thm]{Remark}

\newtheorem{ex}[thm]{Example}
\numberwithin{thm}{section}

\numberwithin{equation}{section}

\begin{document}

\title{Wreath Macdonald operators}
\author{Daniel Orr}\address{Orr:
Department of Mathematics (MC 0123),
	460 McBryde Hall, Virginia Tech,
	225 Stanger St.,
	Blacksburg, VA 24061 USA /
Max Planck Institut f\"{u}r Mathematik, Vivatsgasse 7,
53111 Bonn,
Germany}\email{dorr@vt.edu}
\author{Mark Shimozono}\address{Shimozono: Department of Mathematics (MC 0123),
	460 McBryde Hall, Virginia Tech,
	225 Stanger St.,
	Blacksburg, VA 24061 USA}\email{mshimo@math.vt.edu}
\author{Joshua Jeishing Wen}\address{Wen: Department of Mathematics, Northeastern University, 463 Lake Hall, 43 Leon St, Boston, MA 02115}\email{j.wen@northeastern.edu}
\date{\today}
\maketitle

\begin{abstract}
We construct a novel family of difference-permutation operators and prove that they are diagonalized by the wreath Macdonald $P$-polynomials; the eigenvalues are written in terms of elementary symmetric polynomials of arbitrary degree. Our operators arise from integral formulas for the action of the horizontal Heisenberg subalgebra in the vertex representation of the corresponding quantum toroidal algebra.
\end{abstract}

\section{Introduction}
Let $X_{N}=\left\{ x_{1},\ldots, x_N \right\}$ be a set of variables.
The \textit{Macdonald polynomials} $\left\{ P_\lambda[X_N;q,t] \right\}$ are a basis of the ring of $(q,t)$-deformed symmetric polynomials $\Q(q,t)[X_N]^{\Sym_N}$ that have appeared across a remarkably broad collection of mathematical fields.
They can be characterized as eigenfunctions of a commuting family of difference operators, the \textit{Macdonald operators}: for $1\le n\le N$,
\begin{align}
\label{OldMacOp}
D_n(X_N ;q,t)&:=t^{\frac{n(n-1)}{2}}\sum_{\substack{I\subset\{1,\ldots, N\}\\|I|=n}}\left(\prod_{\substack{i\in I\\j\not\in I}}\frac{tx_i-x_j}{x_i-x_j}\right)\prod_{i\in I} T_{q,x_i}\\
\label{OldEigen}
D_n(X_N ;q,t) P_\lambda[X_N;q,t] &= e_n(q^{\lambda_1}t^{N-1}, q^{\lambda_2}t^{N-2},\ldots, q^{\lambda_N})P_\lambda[X_N ;q,t]
\end{align}
Here, $T_{q,x_i}$ is the $q$-shift operator 
\[
T_{q,x_i} x_j=q^{\delta_{i,j}}x_j
\]
and $e_n$ is the $n$th elementary symmetric polynomial.
The Macdonald operators are themselves distinguished as Hamiltonians of the quantum trigonometric \textit{Ruijsenaars-Schneider} integrable system.

This paper is concerned with the \textit{wreath} Macdonald polynomials, a generalization of the Macdonald polynomials proposed by Haiman \cite{H2}.
Fix an integer $r>0$ and partition the variables $x_1,\ldots, x_N$ into $r$ subsets:
\[
X_{N_\bullet}:=\bigsqcup_{i=0}^{r-1}\left\{x^{(i)}_{l}\right\}_{l=1,\ldots, N_i}=\left\{ x_1,\ldots,x_N \right\}
\]
where $\sum_{i=0}^{r-1}N_i=N$.
We call the index $i$ the \textit{color} of $x^{(i)}_{l}$, and it will be helpful to view it as an element of $I:=\Z/r\Z$.
The number of variables is recorded by the vector $N_\bullet :=(N_0,\ldots, N_{r-1})$ and we set $|N_\bullet|:=N$.
Consider the action of the product of symmetric groups
\[
\Sym_{N_\bullet}:=\prod_{i\in I}\Sym_{N_i}
\]
on the polynomial ring $\Q(q,t)\left[ X_{N_\bullet} \right]$ whereby $\Sym_{N_i}$ only permutes the variables of color $i$.
The wreath Macdonald polynomials can be viewed as a set of \textit{color-symmetric} polynomials that are again indexed by a single partition:
\[
P_\lambda[X_{N_\bullet};q,t]\in\Q(q,t)\left[ X_{N_\bullet} \right]^{\Sym_{N_\bullet}}.
\]
The combinatorics of $r$-cores and $r$-quotients play a key role in this subject, which we review in Section \ref{Wreath} below.
When we restrict $\lambda$ to range over partitions with a fixed $r$-core and $\ell(\lambda)\le |N_\bullet|$, we obtain a basis of color-symmetric polynomials.
For reasons that seem technical at first, the $r$-core and $N_\bullet$ must satisfy a compatibility condition (see \ref{Finite}). 
The original Macdonald polynomials are the case $r=1$.

Haiman's proposed definition characterizes $P_\lambda[X_{N_\bullet};q,t]$ using a pair of triangularity conditions.
In contrast with the usual Macdonald theory, we \textit{a priori} do not have an analogous characterization as the joint eigenfunction of an explicit family of difference operators.
The present work remedies this situation: we produce a novel family of difference-permutation operators that are diagonalized by the wreath Macdonald polynomials and whose eigenvalues are written in terms of the elementary symmetric polynomials.
In addition to the degree $n$, they also carry a color parameter $p\in I$:
\begin{equation}
\begin{aligned}
D_{p,n}(X_{N_\bullet};q,t)&:=
\frac{(-1)^{\frac{n(n-1)}{2}}}{\prod_{k=1}^n(1-q^{k}t^{-k})}\\
&\times\sum_{\mathbf{\Ju}\in Sh_p^{[n]}(X_{N_\bullet})}
\overset{\curvearrowright}{\prod_{a=1}^n}
\left\{(1-qt^{-1})^{|\Ju_a|}\left(\frac{x^{(r-1)}_{\Ju_a^\triangledown}}{x^{(p)}_{\Ju_a}}\right)
\frac{\displaystyle
\prod_{\substack{l=1\\x^{(p)}_{l}\not\in|\mathbf{\Ju}|_{\ge a}}}^{N_p}\left( tx^{(p-1)}_{\Ju_a^\triangledown}-x^{(p)}_{l} \right)}
{\displaystyle
\prod_{\substack{l=1\\x^{(p)}_{l}\not\in|\mathbf{\Ju}|_{\le a}}}^{N_p}\left( x^{(p)}_{\Ju_a}-x^{(p)}_{l} \right)}\right.\\
&\times
\left.
\left( \prod_{\substack{i\in I\backslash\{p\}}}\, 
\prod_{\substack{l=1\\x^{(i)}_{l}\not=x^{(i)}_{\Ju_a^\triangledown}}}^{N_i}\frac{\left(tx^{(i-1)}_{\Ju_a^\triangledown}-x^{(i)}_{l}\right)}
{\left( x^{(i)}_{\Ju_a^\triangledown}-x^{(i)}_{l} \right)}
\right)
\left( \prod_{i\in J_a\backslash \left\{ p \right\}} 
\frac{q^{-1}t T_{\Ju_a}x^{(i)}_{\Ju_a}}
{\left( x^{(i)}_{\Ju_a}-T_{\Ju_a}x^{(i)}_{\Ju_a} \right)}
\right) T_{\Ju_a}\right\}.
\end{aligned}
\label{WreathMacOp}
\end{equation}
The notation used in this formula is outlined in \ref{Cyclic}. Our main result is the following:
\begin{ithm}[see Theorem~\ref{ThmEigenHigher}]
For $\lambda$ having $r$-core compatible with $N_\bullet$ and $\ell(\lambda)\le |N_\bullet|$, the polynomial $P_\lambda[X_{N_\bullet};q,t]$ satisfies the eigenfunction equation
\begin{equation}
D_{p,n}(X_{N_\bullet};q,t)P_\lambda[X_{N_\bullet};q,t]
=e_n\left[ \sum_{\substack{b=1\\ b-\lambda_b\equiv p+1\: \mathrm{mod} \: r}}^{|N_\bullet|}q^{\lambda_b}t^{|N_\bullet|-b} \right]
P_\lambda[X_{N_\bullet};q,t].
\label{WreathEigen}
\end{equation}
\end{ithm}
\noindent
For the eigenvalues, we have used plethystic notation---we merely mean the elementary symmetric function $e_n$ evaluated at the characters appearing in the summation.
In earlier work \cite{OS}, the first two authors constructed the first order dual operators $D_{p,1}^*$ and their eigenfunction equation in Theorem~\ref{ThmEigenHigher}.

Our operators (\ref{WreathMacOp}) are much more complicated than the original Macdonald operators (\ref{OldMacOp}).
In the case $r=1$, we do indeed obtain (\ref{OldMacOp}) after some simplification (see Remark \ref{MacCompare}).
When $r>1$, the vanilla $q$-shift operator $T_{q,x_i}$ is replaced with what we call a \textit{cyclic-shift operator} $T_{\Ju_a}$, which cyclically permutes variables of different colors in addition to multiplying by a power of $q$.
Because of this extra permutation, the cyclic-shift operators might not commute.
Note now the \textit{ordered} product in (\ref{WreathMacOp})---we expect the formula to simplify meaningfully after taking into account the (non)commutativity of the constituent cyclic-shift operators.
Moving beyond the intricacies of our formula, let us now highlight some nice conceptual aspects of our operators.

\subsection{Integral formulas}
Our strategy for deriving (\ref{WreathMacOp}) and establishing the eigenfunction equation uses work of the third author \cite{Wen}.
Namely, we study the wreath Macdonald polynomials using the \textit{quantum toroidal algebra} $U_{\qqq,\ddd}(\ddot{\mathfrak{sl}}_r)$ and its \textit{vertex representation} $W$.
The aforementioned work proves that infinite-variable wreath Macdonald polynomials can be naturally embedded inside $W$ such that they diagonalize a large commutative subalgebra of $U_{\qqq,\ddd}(\ddot{\mathfrak{sl}}_r)$, the \textit{horizontal Heisenberg subalgebra}.
This alone is insufficient for obtaining explicit formulas---we also need work of Negu\cb{t} \cite{Neg} realizing $U_{\qqq,\ddd}(\ddot{\mathfrak{sl}}_r)$ in terms of a \textit{shuffle algebra}.
The shuffle algebra is a space of rational functions endowed with an exotic product structure, and it is isomorphic to a part of $U_{\qqq,\ddd}(\ddot{\mathfrak{sl}}_r)$ via a map that is morally (but not precisely) an integration map.
Writing its action on $W$ and then specializing from infinite to finite variables, we obtain actual integral formulas.
Finally, to pin down the eigenvalues, we use the (twisted) isomorphism established by Tsymbaliuk \cite{T} between the vertex representation and the \textit{Fock representation}.

We apply this process to the shuffle realizations of well-chosen elements of the horizontal Heisenberg subalgebra which were found in \cite{Wen}.
Our operators are the highest degree parts (see Proposition \ref{ResidueCalc}), and we can write their action as follows: for a factored element
\[
f=\prod_{i\in I} f_i\left(x^{(i)}_{\bullet}\right)\in\C(q,t)\left[ X_{N_\bullet} \right]^{\Sym_{N_\bullet}}
\]
\begin{align*}
D_{p,n}(X_{N_\bullet};q,t)f&= \oint_{C}
\frac{(-1)^{\frac{n(n-1)}{2}}t^{-\frac{n(n+1)}{2}}(1-qt^{-1})^{nr}}{\prod_{a=1}^n(1-q^at^{-a})} 
\prod_{i\in I}\prod_{a=1}^n\prod_{l=1}^{N_i}\left(\frac{  tw_{i,a}-x^{(i)}_{l} }{ w_{i+1,a}-x^{(i)}_{l} }\right)\\
&\times\prod_{1\le a<b\le n}\left\{
\frac{\left(w_{p,a}-w_{p,b}\right)\left(w_{p,a}-qt^{-1}w_{p,b}\right)}{\left(w_{p,b}-t^{-1}w_{p+1,a}\right)\left(w_{p-1,a}-t^{-1}w_{p,b}\right)}\right.\\
&\times\left.
\prod_{i\in I\backslash\{p\}}
\frac{\left(w_{i,a}-w_{i,b}\right)\left(w_{i,a}-qt^{-1}w_{i,b}\right)}{\left(w_{i+1,a}-qw_{i,b}\right)\left(w_{i-1,a}-t^{-1}w_{i,b}\right)}\right\}\\
&\times
\prod_{a=1}^n\left\{
\left(\frac{w_{0,a}}{w_{p+1,a}}\right)\left(\frac{w_{p+1,a}}{w_{p,a}-t^{-1}w_{p+1,a}}\right)\right.\\
&\times\left.
\prod_{i\in I\setminus\{p\}}
\left(\frac{w_{i,a}}{w_{i,a}-t^{-1}w_{i+1,a}}\right)
\left( \frac{w_{i+1,a}}{w_{i+1,a}-qw_{i,a}} \right)
\right\}\\ 
&\times 
\prod_{i\in I}f_i\left[  \sum_{l=1}^{N_i}x^{(i)}_{l}+\sum_{a=1}^nqw_{i,a}-\sum_{a=1}^nw_{i+1,a} \right]
\prod_{a=1}^n \frac{dw_{i,a}}{2\pi\sqrt{-1}w_{i,a}}
\end{align*}
where for each variable $w_{i,a}$, the cycle $C$ only encloses poles of the form $(w_{i,a}-qw_{i-1,b})$ and $(w_{i,a}-x_{i-1,l})$.
Explicit evaluation of this integral leads to \eqref{WreathEigen}.
We also carry this out for its dual counterpart in Theorem~\ref{ThmEigenHigher}.

Using other shuffle elements from \cite{Wen}, we obtain similar integral formulas for wreath analogues of the \textit{Noumi-Sano operators} \cite{NSa}, although we are only able to evaluate the integral and obtain formulas for the operators in degree $n=1$.
We note that our approach is similar to \cite{FHHSY} in the $r=1$ case, although our \textit{a priori} knowledge and endgoals are different.
In \cite{FHHSY}, the authors use the well-known Macdonald operators to study the action of certain shuffle elements, whereas we use $r>1$ analogues of their shuffle elements to discover new operators.
In \cite{T2}, Tsymbaliuk has also produced difference operators out of $U_{\qqq,\ddd}(\ddot{\mathfrak{sl}}_r)$ through very different means.
The relation between Tsymbaliuk's operators to wreath Macdonald theory does not seem straightforward but could be interesting.

\subsection{Towards bispectral duality}\label{Duality}
In the case $r=1$, the eigenfunction equation (\ref{OldEigen}) is particularly interesting when juxtaposed with the \textit{Pieri rules} \cite{Mac}.
To make this apparent, introduce a continuous extension of the discrete parameters $\lambda=(\lambda_1,\ldots,\lambda_{N})$:
\[
s_i:=q^{\lambda_i}t^{N-i},\, S_N:=\{s_1,\ldots, s_N\}.
\]
We call the variables $X_N$ the \textit{position variables} and $S_N$ the \textit{spectral variables}.
It is natural to interpret the spectral $q$-shift $T_{q,s_i}P_\lambda[X_N;q,t]$ as adding a box to row $i$ of the partition $\lambda$.
For a certain renormalization $\tilde{P}_\lambda[X_N;q,t]$ of $P_{\lambda}[X_{N};,q,t]$, we can write the Pieri rules as
\begin{equation}
e_n(x_1,\ldots, x_N)\tilde{P}_\lambda[X_N ;q,t]
=t^{\frac{n(n-1)}{2}}\sum_{\substack{I\subset\{1,\ldots, N\}\\|I|=n}}\left(\prod_{\substack{i\in I\\j\not\in I}}\frac{ts_i-s_j}{s_i-s_j}\right)\prod_{i\in I} T_{q,s_i}\tilde{P}_\lambda[X_{N} ;q,t].
\label{Pieri}
\end{equation}
The fact that no shift operator $T_{q,s_i}$ appears more than once enforces the well known support condition of the Pieri rules: the $\tilde{P}_\mu[X_N;q,t]$ that appear on the right hand side of (\ref{Pieri}) are such that $\mu\backslash\lambda$ contains no horizontally adjacent boxes.
On the other hand, we can view the eigenfunction equation (\ref{OldEigen}) as describing multiplication by $e_n(s_1,\ldots, s_N)$.
The similarity between (\ref{OldEigen}) and (\ref{Pieri}) is reflective of a symmetry $X_N\leftrightarrow S_N$.

This symmetry is the subject of many beautiful works in Macdonald theory.
A totalizing perspective on this was given by Noumi and Shiraishi \cite{NS}, who produced an explicit function $f_N(s_1,\ldots, s_N|x_1,\ldots, x_N)$
satisfying
\begin{align*}
f_N(q^{\lambda_1}t^{N-1},q^{\lambda_2}t^{N-2},\ldots, q^{\lambda_N}|x_1,\ldots, x_N)&= \tilde{P}_\lambda[X_N;q,t]\\
f_N(s_1,\ldots, s_N|x_1,\ldots, x_N)&= f_N(x_1,\ldots, x_N|s_1,\ldots, s_N).
\end{align*}
Discretizing the $x$-variables as well, we obtain the well-known \textit{evaluation duality} \cite{Mac}:
\[
\tilde{P}_\lambda(q^{\mu_1}t^{N-1},q^{\mu_2}t^{N-2},\ldots, q^{\mu_N})=\tilde{P}_\mu(q^{\lambda_1}t^{N-1},q^{\lambda_2}t^{N-2},\ldots, q^{\lambda_N}).
\]
The evaluation duality is also a consequence of the \textit{Cherednik-Macdonald-Mehta formula} \cite{C}, which can be regarded as a remarkable statement about the quantum toroidal algebra $U_{q,t}(\ddot{\mathfrak{gl}}_1)$ and its \textit{Miki automorphism}.
The $X_N\leftrightarrow S_N$ symmetry has also been extended by Etingof and Varchenko \cite{EV} to the much broader context of traces of intertwiners for quantum groups, although we note that in their setting, finding explicit formulas is difficult.
Finally, the symmetry is also a case of \textit{3d mirror symmetry} as proposed by Okounkov \cite{AO}.

For the wreath case $r>1$, the spectral variables should also have color. 
We assign $s^{(i)}_{l}$ to some $b$ such that $b-\lambda_b\equiv i+1$ mod $r$:
\[
s^{(i)}_{l}:=q^{\lambda_b}t^{|N_\bullet|-b}.
\]
Here, we point out a natural motivation for imposing our compatibility condition between $\core_r(\lambda)$ and $N_\bullet$---it forces there to also be $N_i$ spectral variables of color $i$. 
The eigenfunction equation (\ref{WreathEigen}) then describes multiplication by $e_n(s^{(p)}_{1},\ldots, s^{(p)}_{N_p})$.
Note that adding a box to a row will not only contribute a $q$-shift but also change the color, and that is precisely what the cyclic-shift operators $T_{\Ju_a}$ do.
Work of the third author \cite{Wen} provides one constraint on the support of the wreath Pieri rules.
Namely, for a box $(a,b)$, if we call the class of $b-a$ mod $r$ its color, then $P_\mu[X_{N_\bullet};q,t]$ appears as a summand of
\[
e_n(x_{p,1},\ldots, x_{p,N_p})P_\lambda[X_{N_\bullet};q,t]
\]
only if $\mu\backslash\lambda$ consists of $n$ boxes of each color such that no boxes of color $p$ and $p+1$ are horizontally adjacent. 
One can check that the combinations of $T_{\Ju_a}$ appearing in (\ref{WreathMacOp}) enforce this condition after swapping $x^{(i)}_{l}\leftrightarrow s^{(i)}_{l}$.
Computer calculations done by the second author also confirm a wreath analogue of evaluation duality.
While we are still a long way from establishing a wreath analogue of the $X_N\leftrightarrow S_N$ symmetry, our strange operators seem to go out of their way to say it must be true.
Generalizing any of the aforementioned perspectives for understanding this symmetry must surely lead to interesting mathematics.

\subsection{Outline}
Section \ref{Wreath} introduces the wreath Macdonald polynomials.
It includes a review of the combinatorics of $r$-cores and $r$-quotients.
Section \ref{Toroidal} focuses on the quantum toroidal algebra and its representations.
We derive eigenvalues for the infinite-variable analogues of our operators.
Section \ref{Shuffle} moves onto the shuffle algebra.
We write the action of a shuffle element on the vertex representation as the constant term of a series.
Section \ref{Difference} is the technical heart of the paper.
We derive integral formulas for our operators and compute the integral.
Some additional efforts are needed to go from the infinite-variable eigenvalues to their finite-variable versions.
Finally, in the Appendix, we derive integral formulas for wreath analogues of Noumi-Sano operators.
Unfortunately, for these operators, we were only able to evaluate the integrals for degree $n=1$.
Throughout, we present examples following the derivation of each of our operators.

\subsection{Acknowledgments}
We thank Mark Haiman, Andrei Negu\cb{t}, and Alexander Tsymbaliuk for helpful conversations. 
D.O. gratefully acknowledges support from the Simons Foundation (638577, MPS-TSM-00008136) and the Max Planck Institute for Mathematics (MPIM Bonn). 
J.J.W. was supported by NSF-RTG grant ``Algebraic Geometry and Representation Theory at Northeastern University'' (DMS-1645877) and ERC consolidator grant No. 101001159 ``Refined invariants in combinatorics, low-dimensional topology and geometry of moduli spaces''.

\section{Wreath Macdonald functions}\label{Wreath}

Fix a positive integer $r$ and let $I=\Z/r\Z$.
	
\subsection{Partitions}
Let $\Y$ be the set of all integer partitions. We define the diagram of a partition $\mu=(\mu_1,\mu_2,\dotsc)\in \Y$ to be $D(\mu)=\{(a,b)\in\left(\Z_{\ge0}\right)^2 : 0 \le a < \mu_{b+1} \}.$ The \emph{residue} of $(a,b)\in\Z^2$ is the element $b-a\in \Z/r\Z$. 






\subsection{Edge sequences and partitions}
A function $b:\Z\to\{0,1\}$ can be viewed as an infinite indexed binary word $\dotsm b(1) b(0) b(-1) \dotsm$; notice that in writing such a word we index the positions in reverse order. An \emph{inversion} of $b$ is a pair of integers $i>j$ such that $b(i)>b(j)$, a $1$ to the left of a $0$. An \emph{edge sequence} is a function $b:\Z\to\{0,1\}$ such that $b(i)=0$ for $i\gg0$ and $b(i)=1$ for $i\ll0$, that is, $b$ has finitely many inversions. Let $\ES$ denote the set of edge sequences. 
The \emph{shape} of $b\in\ES$ is the partition whose French partition diagram has boundary traced out by the 
values of $b$ from northwest to southeast where $0$ (resp. $1$) indicates a vertical (resp. horizontal) unit segment; see Figure \ref{F:edge sequence}.
Its parts are given by the number of $1$'s to the left of each $0$ in the edge sequence.
The \emph{charge} of $b$ is the index of the segment that touches the main diagonal from the northwest, or equivalently the index of the last $0$ in the edge sequence of the form $\dotsm0011\dotsm$ obtained from $b$ by repeatedly swapping adjacent pairs $10$ to $01$ until none remain.
There is a bijection 
\begin{equation}
\label{E:edge sequence of partition}
\begin{split}
\ES &\to \Z \times \Y \\
b &\mapsto (\charge(b),\shape(b)).
\end{split}
\end{equation}

\begin{ex} An edge sequence $b$ and its charge and shape are pictured in Figure \ref{F:edge sequence}.
\ytableausetup{boxsize=7pt,aligntableaux=center}
\begin{figure}
\label{F:edge sequence}
\[
\begin{tikzpicture}
\node at (0,6) {$\bullet$};
\node at (0,5) {$\bullet$};
\node at (0,4) {$\bullet$};
\node at (1,4) {$\bullet$};
\node at (2,4) {$\bullet$};
\node at (2,3) {$\bullet$};
\node at (2,2) {$\bullet$};
\node at (3,2) {$\bullet$};
\node at (3,1) {$\bullet$};
\node at (4,1) {$\bullet$};
\node at (4,0) {$\bullet$};
\node at (5,0) {$\bullet$};
\node at (6,0) {$\bullet$};

\draw (0,6) -- (0,0) -- (6,0);
\coordinate (zf) at (0,6);
\coordinate (ze) at (0,5);
\coordinate (zd) at (0,4); 
\coordinate (ad) at (1,4);
\coordinate (bd) at (2,4);
\coordinate (bc) at (2,3);
\coordinate (bb) at (2,2);
\coordinate (cb) at (3,2);
\coordinate (ca) at (3,1);
\coordinate (da) at (4,1);
\coordinate (dz) at (4,0);
\coordinate (ez) at (5,0);
\coordinate (fz) at (6,0);

\draw[dotted] (0,0) -- (3,3);

\path[every node/.style={font=\sffamily\small}]
     (zf)  edge node[midway,right] {5} node[midway,left] {0} (ze);
\path[every node/.style={font=\sffamily\small}]
     (ze)  edge node[midway,right] {4} node[midway,left] {0} (zd);
\path[every node/.style={font=\sffamily\small}]
     (zd)  edge node[midway,above] {3} node[midway,below] {1} (ad);
\path[every node/.style={font=\sffamily\small}]
     (ad)  edge node[midway,above] {2} node[midway,below] {1} (bd);
\path[every node/.style={font=\sffamily\small}]
     (bd)  edge node[midway,right] {1} node[midway,left] {0} (bc);
 \path[every node/.style={font=\sffamily\small}]
     (bc)  edge node[midway,right] {{\color{red}{0}}} node[midway,left] {0} (bb);
 \path[every node/.style={font=\sffamily\small}]
     (bb)  edge node[midway,above] {-1} node[midway,below] {1} (cb);
 \path[every node/.style={font=\sffamily\small}]
     (cb)  edge node[midway,right] {-2} node[midway,left] {0} (ca);
 \path[every node/.style={font=\sffamily\small}]
     (ca)  edge node[midway,above] {-3} node[midway,below] {1} (da);
 \path[every node/.style={font=\sffamily\small}]
     (da)  edge node[midway,right] {-4} node[midway,left] {0} (dz);
 \path[every node/.style={font=\sffamily\small}]
     (dz)  edge node[midway,above] {-5} node[midway,below] {1} (ez);   
 \path[every node/.style={font=\sffamily\small}]
     (ez)  edge node[midway,above] {-6} node[midway,below] {1} (fz);
\end{tikzpicture}
\]
\begin{align*}
&\begin{array} {|c||c|c|c|c|c|c|c||c|c|c|c|c|c|c|c|c|c|} \hline
i   &\dotsm  &5&4&3&2&1&{\color{red}{0}}&-1&-2&-3&-4&-5&-6&\dotsm \\ \hline
b_i &0       &0 &0 & 1& 1& 0&0&1&0&1&0&1&1&1\\ \hline
\end{array} \\[2mm]
&\charge(b)={\color{red}{0}} \qquad \shape(b) = \ydiagram{2,2,3,4}
\end{align*}
\caption{The shape of an edge sequence}
\end{figure}
\ytableausetup{boxsize=normal,aligntableaux=bottom}
\end{ex}

\subsection{Cores and quotients}
Our goal is to define the bijection 
\begin{equation}
\begin{split}
 \Y&\cong \cC_r \times \Y^r   \\
 \la&\mapsto (\core_r(\la),\quot_r(\la))
\end{split}
\end{equation}
where $\core_r$ is the \emph{$r$-core} and $\quot_r$ is the \emph{$r$-quotient} map.

In the following diagram all horizontal maps are bijections and vertical maps are inclusions.
\[
\begin{tikzcd}[column sep=huge]
\Z\times \Y & \ES \arrow[l,swap,"\charge\times \shape"] \arrow[r] & \ES^r \arrow[r,"\cd\times \quot_r"] & \Z^r \times \Y^r \\
\{0\} \times \Y \arrow[u] & \ES_0 \arrow[u] \arrow[l]\arrow[u] \arrow[r] & (\ES^r)_0 \arrow[r] \arrow[u] & Q\times \Y^r \arrow[u] \arrow[r,"\kb^{-1}\times\id"] & \cC_r \times \Y^r \\ 
\{0\} \times \cC_r \arrow[u] \arrow[rrr,"\kb"] & & & Q \times \varnothing^r\arrow[u]
\end{tikzcd}
\]
Elements $b^\bullet=(b^0,b^1,\dotsc,b^{r-1})\in \ES^r$ are called \emph{abaci}. We may write them as $\{0,1,\dotsc,r-1\}\times \Z$ matrices with entries in $\{0,1\}$ where a $0$ is a bead and a $1$ is a hole (position with no bead) and the $i$-th row represents the edge sequence $b^i$ and is the $i$-th runner in the abacus.

There is a bijection $\ES\to \ES^r$ sending $b$ to $(b^0,b^1,\dotsc,b^{r-1})$ by letting $b^i$ select the bits in $b$ indexed by integers congruent to $i$ mod $r$: $b^i(j) = b(rj+i)$ for $0\le i<r$ and $j\in\Z$. The inverse map is given by interleaving the sequences $b^0,b^1,\dotsc,b^{r-1}$. 
This bijection is charge-additive: $\charge(b) = \sum_{j=0}^{r-1} \charge(b^j)$.
The $r$-fold product of the bijection \eqref{E:edge sequence of partition} yields the 
bijection $\ES^r \cong \Z^r \times \Y^r $. Denote this by $\bd=(b^0,\dotsc,b^{r-1})\mapsto ((c_0,\dotsc,c_{r-1}), \lad)$. We write $\lad = \quot_r(\bd)$; this is the $r$-quotient. Call $(c_0,\dotsc,c_{r-1})=\cd(\bd)$ the charge vector. This indicates the position on each runner where the beads end after pushing all beads to the left.
This defines the bijections going across the top row of the diagram.

We now restrict all these bijections. Let $\ES_0 = \{b\in \ES\mid \charge(b)=0\}$ and 
$(\ES^r)_0 = \{\bd\in\ES^r\mid \sum_{i=0}^{r-1} c_i(\bd) =0 \}$. 
Then $\cd(\bd)$ can be viewed as an element of the $\mathfrak{sl}_r$ root lattice $Q$ (and belongs to the zero lattice $Q=0$ when $r=1$).
The second row of the diagram (save the last map) 
is given by suitable restrictions of the top row of bijections.

An $r$-core is a partition $\gamma$ which does not have $r$ as a hook length.
That is, $h_\gamma(i,j)\ne r$ for all $(i,j)\in \gamma$. We denote by $\cC_r\subset \Y$ the set of $r$-cores.
Let $\gamma$ be a partition and let $b\in\ES$ be such that $\shape(b)=\gamma$.
Then $\gamma$ has a box $(i,j)\in\gamma$ with hook-length $r$, that is, $h_\gamma(i,j)=r$, 
if and only if there is an index $k$ such that $b(k)=1$ and $b(k+r)=0$. This is equivalent to $\mu^{(k)}\ne\varnothing$ where $\mud=\quot_r(\gamma)$ and we take superscripts mod $r$. This proves that $\gamma$ is an $r$-core if and only if the $r$-quotient of $\gamma$ is empty: $\quot_r(\gamma)=(\varnothing^r)$.

Therefore the bijection $\{0\} \times\Y \cong Q\times\Y^r$ restricts to the bijection
$\{0\} \times \cC_r \cong Q \times (\varnothing^r)$, that is, $\cC_r\cong Q$. We call this bijection $\kb$.

\begin{ex}\label{X:core-quot}
Let $b\in \ES_0$ be as in the previous example. We have $\la=\shape(b)=(4,3,2,2)$. Set $r=3$. 
We map $b\mapsto (b^0,b^1,b^2)$ which are pictured in the matrix below.
Reading up the columns of the $\{0,1,2\}\times \Z$ matrix we recover $b$. 
Each runner of the abacus is an edge sequence; the corresponding shapes give the $3$-quotient of $(4,3,2,2)$, which is
$(1,\varnothing,2)$.

To get the $3$-core of $\la$ we move all beads to the left in each runner. This produces the second abacus. 
Reading up columns we obtain the edge sequence $a=\dotsm 0001|1011\dotsm$. Therefore $\core_3(4,3,2,2)=\shape(a)=(2)$.
The charge sequence is $(1,-1,0)\in Q$. 

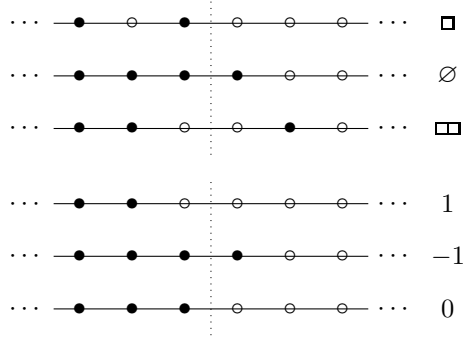
\ytableausetup{boxsize=4pt}

\[
\begin{array} {|c||c|c|c|c|c|c|c||c|c|c|c|c|c|c|c|c|c|} \hline
i   &\dotsm  &5&4&3&2&1&{\color{red}{0}}&-1&-2&-3&-4&-5&-6&\dotsm \\ \hline
b_i &\dotsm       &0 &0 & 1& 1& 0&0&1&0&1&0&1&1&\dotsm\\ \hline
\end{array}
\]

\[
\begin{array}{c||p{10pt}p{10pt}p{10pt}|p{10pt}p{10pt}p{10pt}}
   &2&1&0&-1&-2&-3 \\ \hline \hline
b^0& 0&1& 0&1&1&1 \\
b^1& 0&0& 0&0&1&1 \\
b^2& 0&0& 1&1&0&1 \\ \hline
\end{array}
\]

\begin{align*}
\qquad\qquad&
\begin{tikzpicture}[scale=.70]
\node at (-3,2) (aa) {$\dotsm$};
\node at (-3,1) (ab) {$\dotsm$};
\node at (-3,0) (ac) {$\dotsm$};
\node at (-2,2) (ba) {$\bullet$};
\node at (-2,1) (bb) {$\bullet$};
\node at (-2,0) (bc) {$\bullet$};
\node at (-1,2) (ca) {$\circ$};
\node at (-1,1) (cb) {$\bullet$};
\node at (-1,0) (cc) {$\bullet$};
\node at (0,2) (da) {$\bullet$};
\node at (0,1) (db) {$\bullet$};
\node at (0,0) (dc) {$\circ$};
\node at (1,2) (ea) {$\circ$};
\node at (1,1) (eb) {$\bullet$};
\node at (1,0) (ec) {$\circ$};
\node at (2,2) (fa) {$\circ$};
\node at (2,1) (fb) {$\circ$};
\node at (2,0) (fc) {$\bullet$};
\node at (3,2) (ga) {$\circ$};
\node at (3,1) (gb) {$\circ$};
\node at (3,0) (gc) {$\circ$};
\node at (4,2) (ha) {$\dotsm$};
\node at (4,1) (hb) {$\dotsm$};
\node at (4,0) (hc) {$\dotsm$};
\draw[dotted] (.5,-.5) -- (.5,2.5);
\draw (aa) -- (ha);
\draw (ab) -- (hb);
\draw (ac) -- (hc);
\node at (5,2) (ia) {$\ydiagram{1}$};
\node at (5,1) (ib) {$\varnothing$};
\node at (5,0) (ic) {$\ydiagram{2}$};
\end{tikzpicture}
\\
&
\begin{tikzpicture}[scale=.70]
\node at (-3,2) (aa) {$\dotsm$};
\node at (-3,1) (ab) {$\dotsm$};
\node at (-3,0) (ac) {$\dotsm$};
\node at (-2,2) (ba) {$\bullet$};
\node at (-2,1) (bb) {$\bullet$};
\node at (-2,0) (bc) {$\bullet$};
\node at (-1,2) (ca) {$\bullet$};
\node at (-1,1) (cb) {$\bullet$};
\node at (-1,0) (cc) {$\bullet$};
\node at (0,2) (da) {$\circ$};
\node at (0,1) (db) {$\bullet$};
\node at (0,0) (dc) {$\bullet$};
\node at (1,2) (ea) {$\circ$};
\node at (1,1) (eb) {$\bullet$};
\node at (1,0) (ec) {$\circ$};
\node at (2,2) (fa) {$\circ$};
\node at (2,1) (fb) {$\circ$};
\node at (2,0) (fc) {$\circ$};
\node at (3,2) (ga) {$\circ$};
\node at (3,1) (gb) {$\circ$};
\node at (3,0) (gc) {$\circ$};
\node at (4,2) (ha) {$\dotsm$};
\node at (4,1) (hb) {$\dotsm$};
\node at (4,0) (hc) {$\dotsm$};
\draw[dotted] (.5,-.5) -- (.5,2.5);
\draw (aa) -- (ha);
\draw (ab) -- (hb);
\draw (ac) -- (hc);
\node at (5,2) (ia) {$1$};
\node at (5,1) (ib) {$-1$};
\node at (5,0) (ic) {$0$};
\end{tikzpicture}
\end{align*}
\begin{align*}
\text{core:} \qquad\begin{array} {|c||c|c|c|c|c|c|c||c|c|c|c|c|c|c|} \hline
i &\dotsm  &5&4&3&2&1&0&-1&-2&-3&-4&-5&-6&\dotsm \\ \hline
a_i &0&0&0 & 0&0&0&1&1&0&1&1&1&1&1\\ \hline
\end{array}\qquad \ydiagram{2}
\end{align*}
\end{ex}

\begin{rem}
Our map $\quot_r$ and our definition of charge are the same as in \cite{Wen}, except that we interchange the roles of black and white dots in our Maya diagrams.
\end{rem}

When considering a fixed $r$, we simply write $\core=\core_r$ and $\quot=\quot_r$.

\subsection{Cores and ribbons}

Consider $\mu,\la\in\Y$ such that $\mu\subset \la$. The skew shape $\la/\mu:=D(\la)-D(\mu)$ is a $\mu$-addable and $\la$-removable $r$-ribbon if $|\la|-|\mu|=r$ and the set of boxes $\la/\mu$ is rookwise connected (i.e. any two boxes in $\la/\mu$ can be connected by a chain of horizontally and vertically adjacent boxes in $\la/\mu$) with at most one element on each southwest-northeast diagonal. 
Then an $r$-core is precisely a partition that has no removable $r$-ribbon.
One way to obtain $\core(\mu)$ is to repeatedly remove (removable) $r$-ribbons starting with $\mu$ until an $r$-core is reached; by definition this is $\core(\mu)$. 
This is well-defined: one obtains the same $r$-core independently of the order of removal of $r$-ribbons. It is the same as moving the beads in the abacus to the left.

\subsection{Cores to root lattice}
Recall that $Q$ denotes the $\mathfrak{sl}_r$ root lattice (or $Q=0$ in the case $r=1$), realized as the zero sum elements in the lattice $\Z^I$:
\[
Q:=\left\{(c_0,\ldots, c_{r-1})\in\Z^I\, \middle|\, \sum_{i\in I}c_i=0\right\}.
\]
Let $\epsilon_i\in\Z^I$ be the $i$-th coordinate vector.
Then $Q$ is the spanned by the elements
\[
\alpha_i:=\epsilon_{i-1}-\epsilon_i,\qquad i\in I.
\]
We realize the simple roots of $\mathfrak{sl}_r$ as the $\alpha_i$ for $i\neq 0$.

Another way to compute the bijection $\kb: \cC \to Q$ is as follows.
Define the map $\kb:\Y\to Q$ by
\begin{align*}
  \kb(\mu) = -\sum_{(p,q)\in\mu} \alpha_{q-p}.
\end{align*}
It is not difficult to show that the restriction of $\kb$ to $\cC$ is the same as the bijection $\cC\cong Q$ constructed above.

%

\ytableausetup{boxsize=normal}
\begin{ex} \label{X:roots and cores} Let $r=3$ and consider the $3$-core $(2)$. We put $\alpha_{q-p}$ into the box $(p,q)$:
\[
\begin{ytableau} \alpha_0 & \alpha_2 \end{ytableau}
\]
Thus $\kb((2))=-(\alpha_0+\alpha_2)=\alpha_1$, which agrees with the charge sequence $(1,-1,0)\in Q$ computed above.
\end{ex}



Define the bijection $\bigp:Q\times \Y^I\to \Y$ via the following commutative diagram:
\begin{equation}\label{E:big}
\begin{tikzcd} 
\Y \arrow[rr,"{(\core,\quot)}"]& & \cC\times \Y^I \arrow[d,"\kb\times \id"]\\ && Q\times \Y^I\arrow[ull,"\bigp"] 
\end{tikzcd}
\end{equation}




\begin{ex} \label{X:multipartition order} We list the elements $\mud\in\Y^I$ of total size $2$ and their images under $\mud\mapsto \bigp(-\alpha_1,\mud)$.
\ytableausetup{boxsize=3pt,aligntableaux=bottom}
\begin{equation}\notag
\begin{array}{|c|l|} \hline
\mud & \text{image} \\ \hline \hline
\begin{matrix}
\ydiagram{2} & \dt & \dt
\end{matrix} 
& \ydiagram{0,1,9} \\ \hline
\begin{matrix}
\ydiagram{1,1} & \dt & \dt
\end{matrix} & \ydiagram{0,4,6} \\ \hline
\begin{matrix}
\ydiagram{1} & \dt & \ydiagram{1}
\end{matrix} & \ydiagram{0,2,2,6} \\ \hline
\begin{matrix}
\ydiagram{1} & \ydiagram{1}&\dt
\end{matrix} & \ydiagram{0,1,1,1,1,6} \\ \hline
\begin{matrix}
\dt & \dt & \ydiagram{2} 
\end{matrix} & \ydiagram{0,2,4,4} \\ \hline
\begin{matrix}
\dt & \ydiagram{2} & \dt 
\end{matrix} & \ydiagram{0,1,1,2,3,3} \\ \hline
\begin{matrix}
\dt & \ydiagram{1} & \ydiagram{1} 
\end{matrix} & \ydiagram{0,1,2,2,2,3} \\ \hline
\begin{matrix}
\dt & \dt & \ydiagram{1,1} 
\end{matrix} & \ydiagram{0,1,1,1,2,2,3} \\ \hline
\begin{matrix}
\dt  & \ydiagram{1,1} &\dt
\end{matrix} & \ydiagram{0,1,1,1,1,1,1,1,3} \\ \hline
\end{array} 
\end{equation}
\ytableausetup{boxsize=normal}
\end{ex}


\subsection{Symmetric functions}
Let $\La$ be the algebra of symmetric functions over $\K=\Q(q,t)$ in infinitely many variables \cite[\S I.2]{Mac}.
Denote by $\La^{I}=\La^{\otimes I}$ the $I$-fold tensor power of $\La$ over $\K$, which is a graded $\K$-algebra with grading given by the sum of degrees in each tensor factor. 
For $f\in \La$, we write $f[X^{(i)}]$ to indicate the element of $\La^I$ with $1$ in tensor factors $j\ne i$ and $f$ in factor $i$. 
The power sums $p_k[X^{(i)}]$ for $i\in I$ and $k>0$ generate $\Lambda^I$ as a $\K$-algebra. We write $\Xd$ for the $I$-tuple of alphabets $(X^{(0)},\dotsc,X^{(r-1)})$ and often denote by $f[\Xd]$ a generic element of $\Lambda^I$.
Note that each alphabet $X^{(i)}$ itself contains infinitely many variables.

For an $I$-tuple of partitions $\lad=(\la^{(0)},\la^{(1)},\dotsc,\la^{(r-1)})\in \Y^I$, define the tensor Schur function $s_\lad=\bigotimes_{i\in I} s_{\la^{(i)}}=\prod_{i\in I} s_{\la^{(i)}}[X^{(i)}]$. Let $\langle-,-\rangle$ be the Hall pairing on $\Lambda^I$, which is given by $\langle s_\lad, s_\mud \rangle = \delta_{\lad,\mud}$. For $f\in\Lambda^I$, we denote by $f^\perp$ be the adjoint under the Hall pairing to the operator of multiplication by $f$.
Explicitly,
\[
\left[p_n^\perp[X^{(i)}], p_m[X^{(j)}]\right]=n\delta_{n,m}\delta_{i,j},
\]
where we view $p_m[X^{(j)}]$ as a multiplication operator.

For any $a\in\K$, define the $\K$-algebra automorphism $\cP_{\id-a\chi^{-1}}$ of $\Lambda^I$ by
\begin{align}
\cP_{\id-a\chi^{-1}}(p_k[X^{(i)}]) &= p_k[X^{(i)}]-a^{k}p_k[X^{(i-1)}]
\end{align}
for all $i\in I$ and $k>0$. (The notation $\cP_{\id-a\chi^{-1}}$ arises from more general matrix plethysms $\cP_A$ for $A\in\mathrm{Mat}_{I\times I}(\K)$ defined in \cite{OS2}.)

\subsection{Wreath Macdonald functions}

For a partition $\lambda$, let $H_\lambda[\Xd;q,t]$ be the wreath Macdonald functions \cite[Conjecture 7.2.19]{H2}, as defined in \cite[\S 2.3]{Wen}.\footnote{In the more general framework of \cite{OS2} (due to Haiman), these are the wreath Macdonald functions attached to translation elements in the affine Weyl group of type $A_{r-1}$.} These are characterized by the conditions
\begin{align}
\cP_{\id-q\chi^{-1}}H_\lambda[\Xd;q,t] &\in \K^\times s_{\quot(\lambda)} + \bigoplus_{\substack{\nu>\la\\\kb(\nu)=\kb(\la)}} \K s_{\quot(\nu)}\\
\cP_{\id-t^{-1}\chi^{-1}}H_\lambda[\Xd;q,t] &\in \K^\times s_{\quot(\lambda)} + \bigoplus_{\substack{\nu<\la\\\kb(\nu)=\kb(\la)}} \K s_{\quot(\nu)}\\
\langle s_{(n)}[X^{(0)}], H_\lambda[\Xd;q,t]\rangle &= 1.
\end{align}
where $n=|\quot(\lambda)|$ and $<$ is the (strict) dominance order on partitions \cite[\S I.1]{Mac}.

For any $\lambda\in\Y$, the wreath Macdonald $P$-function $P_\lambda[\Xd;q,t^{-1}]$ is defined to be the scalar multiple of $\cP_{\id-t^{-1}\chi^{-1}}(H_\lambda[\Xd;q,t])$ in which the coefficient of $s_{\quot(\lambda)}$ is $1$. In particular, $P_\lambda[\Xd;q,t^{-1}]$ satisfies the unitriangularity
\begin{align*}
P_\lambda[\Xd;q,t^{-1}]\in s_{\quot(\lambda)} + \bigoplus_{\substack{\nu<\la\\\kb(\nu)=\kb(\la)}} \K s_{\quot(\nu)}
\end{align*}
For any fixed $\alpha\in Q$, the $P_\lambda[\Xd;q,t^{-1}]$ such that $\kb(\lambda)=\alpha$ form a homogeneous basis of $\Lambda^I$, with $P_\lambda[\Xd;q,t^{-1}]$ having degree $|\quot(\lambda)|$. 

Our notation $P_\lambda[\Xd;q,t^{-1}]$ agrees with the usual conventions in the classical $r=1$ case. For technical reasons, it is often convenient to work with $P_\lambda[\Xd;q,t^{-1}]$ rather than $P_\lambda[\Xd;q,t]$, though we will eventually switch to the latter.

\subsection{Symmetric polynomials}
For any $N_\bullet=(N_0,\ldots, N_{r-1})\in\left( \Z_{\ge 0} \right)^I$, we can consider a finite set of variables 
\[
X_{N_\bullet}:=
\{x_{l}^{(i)}\}^{i\in I}_{1\le l\le N_i}
\]
and the corresponding restriction map
\begin{align}\label{E:to pol}
\pi_{N_\bullet} : \Lambda^I &\to \Lambda^I_{N_\bullet}:=\bigotimes_{i\in I} \K\left[x_{1}^{(i)},\dotsc,x_{N_i}^{(i)}\right]^{\fS_{N_i}} \\
\nonumber
p_n[X^{(i)}]&\mapsto \sum_{l=1}^{N_i}\left(x_{l}^{(i)}\right)^n=p_n[x_\bullet^{(i)}]
\end{align}
given by the tensor product $\pi_{N_\bullet} = \otimes_{i\in I} \pi_{N_i}$, where $\pi_N : \Lambda\to\K[x_1,\dotsc,x_N]^{\fS_N}$ is the standard projection to symmetric polynomials. 
We also write $\pi_{N_\bullet}(f)=f[X_{N_\bullet}]$.


\subsection{Finitization}\label{Finite}
Our main result will characterize the images $P_\lambda[X_{N_\bullet};q,t]:=\pi_{N_\bullet}(P_\lambda[\Xd;q,t])$ as eigenfunctions of explicit $q$-difference operators. For reasons which are clarified in Remark \ref{CompRem} below, we will only consider variable number vectors $N_\bullet$ for $P_\lambda$ which are compatible with $\core(\lambda)$ in the following way.
If $\kb(\lambda)=\alpha=(c_0,c_1,\dotsc,c_{r-1})$, then we stipulate that $P_\lambda$ will only be assigned variables $X_{N_\bullet}$ where $N_\bullet$ is equivalent to $-\kb(\lambda)$ modulo $\Z(1,\dotsc,1)$, i.e.,
\begin{align}\label{Compatibility}
N_{i}-N_{i-1}=(\alpha_i^\vee,\kb(\lambda))=(\alpha_i^\vee,\alpha)=c_{i-1}-c_{i},\quad \text{for all $i\in I$,}
\end{align}
where:
\begin{itemize}
    \item $\alpha_i^\vee$ is the coroot for $i\not=0$;
    \item $\alpha_0=-\alpha_1-\cdots-\alpha_{r-1}$;
    \item $(-,-):Q^\vee\times Q\to \Z$ is the standard pairing between $\mathfrak{sl}_r$ root and coroot lattices.
\end{itemize}
 Identifying the lattices $Q^\vee\cong Q$ and realizing $Q$ inside $\Z^I$ as above, $(-,-)$ becomes the dot product on $\Z^I$ and $\alpha_i^\vee=\epsilon_{i-1}-\epsilon_i$ for all $i\in I$.

\begin{ex}
In the setting of Example~\ref{X:core-quot}, the root lattice element is $\kb(\lambda)=(1,-1,0)$. The smallest variable number vector which we allow for $\lambda=(4,3,2,2)$ is therefore $N_\bullet=(0,2,1)$. To this we can add the vector $(1,1,1)$ any number of times.
\end{ex}

\begin{lem}\label{SumR}
Under the compatibility condition \eqref{Compatibility} between $N_\bullet\in(\Z_{\ge 0})^I$ and $\alpha\in Q$, 
we have the following:
\begin{enumerate}
\item The quantity
\[|N_\bullet|:=\sum_{i\in I}N_i\]
is divisible by $r$.
\item For $\lambda\in\Y$ with $\kb(\lambda)=\alpha$ and $\ell(\lambda)\le |N_\bullet|$,
\[
N_i=\#\left\{ 1\le b\le |N_\bullet| :  b-\lambda_b\equiv i+1\hbox{ mod }r \right\}
\]
where we count $\lambda_b=0$ if $\ell(\lambda)<b\le|N_\bullet|$; in particular, $\mathrm{quot}(\lambda)=\lambda^\bullet$ satisfies $\ell(\lambda^{(i)})\leq N_i$ for all $i\in I$.
\item For any $\lambda^\bullet\in\Y^I$ satisfying $\ell(\lambda^{(i)})\leq N_i$ for all $i$, the partition $\lambda=\mathrm{big}(\lambda^\bullet,\alpha)$ satisfies $\ell(\lambda)\le |N_\bullet|$.
\end{enumerate} 
\end{lem}

\begin{proof}
\begin{enumerate}
    \item This follows from the fact that $N_\bullet$ and $-\kb(\lambda)$ are congruent modulo $\Z(1,\dotsc,1)$, and the coordinates of the latter sum to zero.
    \item This follows from \cite[I.1, Ex. 8]{Mac} after taking our labeling conventions into account.
    \item For any edge sequence $b$, the length of $\shape(b)$ is precisely the number of $0$'s positioned to the right of at least one $1$. Given $\alpha\in Q$, our choice of $N_\bullet$ ensures that the number of $0$'s positioned to the right of $1$'s in the interleaved edge sequence defining $\lambda$ will not exceed $|N_\bullet|$.\qedhere
\end{enumerate}
\end{proof}

An immediate consequence of parts (2) and (3) of Lemma~\ref{SumR} is the following:

\begin{prop}\label{Basis}
Under the compatibility condition \eqref{Compatibility} between $N_\bullet\in (\Z_{\ge 0})^I$ and $\alpha\in Q$, the wreath Macdonald polynomials $P_\lambda[X_{N_\bullet};q,t]$ indexed by $\lambda\in \Y$ satisfying $\ell(\lambda)\le |N_\bullet|$ and $\kb(\lambda)=\alpha$ form a basis of $\Lambda^I_{N_\bullet}$.
\end{prop}

\section{Quantum toroidal algebra}\label{Toroidal}

To ensure compatibility with \cite{Wen} and \cite{T}, we assume that $r\ge 3$ from this point on.\footnote{See Remark~\ref{MacCompare} and Remark~\ref{Rmk2Colors} for discussion of the cases $r=1,2$.}

\subsection{The algebra $U_{\qqq,\ddd}(\ddot{\mathfrak{sl}}_r)$}
Let $\qqq$ and $\ddd$ be two indeterminates, and set $\F:=\C(\qqq^{\frac{1}{2}},\ddd^{\frac{1}{2}})$.

\subsubsection{Generators and relations}\label{TorDef}
For $i,j\in I=\Z/r\Z$, we set
\begin{align*}
    a_{i,j}&=\left\{
\begin{array}{ll}
    2 & j=i \\
    -1 & j=i\pm 1\\
    0 & \hbox{otherwise}
\end{array}
\right.\\
m_{i,j}&=\left\{
\begin{array}{ll}
    \mp 1 & j=i\pm 1\\
    0 & \hbox{otherwise}
\end{array}
\right.
\end{align*}
and we define
\[
g_{i,j}(z):=\frac{\qqq^{a_{i,j}}z-1}{z-\qqq^{a_{i,j}}}.
\]
The \textit{quantum toroidal algebra} $U_{s,u}(\ddot{\mathfrak{sl}}_r)$ is a unital associative $\F$-algebra with generators
\[
\{e_{i,k},f_{i,k},\psi_{i,k},\psi_{i,0}^{-1},\gamma^{\pm\frac{1}{2}}, \qqq^{\pm d_1},\qqq^{\pm d_2}\}_{i\in I}^{k\in\Z}.
\]
Its relations are described in terms of \textit{currents}:
\begin{align*}
    e_i(z)&:=\sum_{k\in\Z}e_{i,k}z^{-k}\\
    f_i(z)&:=\sum_{k\in\Z}f_{i,k}z^{-k}\\
    \psi_i^\pm(z)&:=\psi_{i,0}^{\pm 1}+\sum_{k>0}\psi_{i,\pm k}z^{\mp k}.
\end{align*}
The relations are then
\begin{gather*}
[\psi_i^\pm(z),\psi_j^\pm(w)]=0,\,\gamma^{\pm\frac{1}{2}}\hbox{ are central},\\
\psi_{i,0}^{\pm1}\psi_{i,0}^{\mp1}=\gamma^{\pm\frac{1}{2}}\gamma^{\mp\frac{1}{2}}=\qqq^{\pm d_1}\qqq^{\mp d_1}=\qqq^{\pm d_2}\qqq^{\mp d_2}=1,\\
\qqq^{d_1}e_i(z)\qqq^{-d_1}=e_i(\qqq^{-1} z),\, \qqq^{d_1}f_i(z)\qqq^{-d_1}=f_i(\qqq^{-1} z),\, \qqq^{d_1}\psi_i^\pm(z)\qqq^{-d_1}=\psi_i^\pm(\qqq^{-1} z),\\
\qqq^{d_2}e_i(z)\qqq^{-d_2}=\qqq e_i(z),\, \qqq^{d_2}f_i(z)\qqq^{-d_2}=\qqq^{-1} f_i(z),\, \qqq^{d_2}\psi_i^\pm(z)\qqq^{-d_2}=\psi_i^\pm( z),\\
g_{i,j}(\gamma^{-1}\ddd^{m_{i,j}}z/w)\psi_i^{+}(z)\psi_j^{-}(w)=g_{i,j}(\gamma\ddd^{m_{i,j}} z/w)\psi_j^{-}(w)\psi_i^{+}(z),\\
e_i(z)e_j(w)=g_{i,j}(\ddd^{m_{i,j}}z/w)e_j(w)e_i(z),\\
f_i(z)f_j(w)=g_{i,j}(\ddd^{m_{i,j}}z/w)^{-1}f_j(w)f_i(z),\\
(\qqq-\qqq^{-1})[e_i(z),f_j(w)]=\delta_{i,j}\left(\delta(\gamma w/z)\psi_i^+(\gamma^{\frac{1}{2}}w)-\delta(\gamma z/w)\psi_i^-(\gamma^\frac{1}{2}z)\right),\\
\psi_i^\pm(z)e_j(w)=g_{i,j}(\gamma^{\pm\frac{1}{2}}\ddd^{m_{i,j}}z/w)e_j(w)\psi_i^\pm(z),\\
\psi_i^\pm(z)f_j(w)=g_{i,j}(\gamma^{\mp\frac{1}{2}}\ddd^{m_{i,j}}z/w)^{-1}f_j(w)\psi_i^\pm(z),\\
\mathrm{Sym}_{z_1,z_2}[e_i(z_1),[e_i(z_2),e_{i\pm1}(w)]_\qqq]_{\qqq^{-1}}=0,\,[e_i(z),e_j(w)]=0\hbox{ for }j\not=i,i\pm1,\\
\mathrm{Sym}_{z_1,z_2}[f_i(z_1),[f_i(z_2),f_{i\pm1}(w)]_\qqq]_{\qqq^{-1}}=0,\,[f_i(z),f_j(w)]=0\hbox{ for }j\not=i,i\pm1,
\end{gather*}
Here, $\delta(z)$ denotes the delta function
\[\delta(z)=\sum_{k\in\Z}z^k\]
and for $v\in\F$, $[a,b]_v=ab-v ba$ is the $v$-commutator.
We will also work with elements $\{h_{i, k}\}_{i\in I}^{k\not= 0}$ defined by
\begin{equation}
\label{HeisDef}
 \psi_i^\pm(z)=\psi_{i,0}^{\pm 1}\exp\left(\pm(\qqq-\qqq^{-1})\sum_{k>0}h_{i,\pm k}z^{\mp k}\right).   
\end{equation}
Finally, we denote by:
\begin{itemize}
    \item $'\ddot{U}$ the subalgebra obtained by dropping the generator $\qqq^{d_1}$;
    \item $\ddot{U}'$ the subalgebra obtained by dropping the generator $\qqq^{d_2}$;
    \item $'\ddot{U}'$ the subalgebra obtained by dropping both generators $\qqq^{d_1}$ and $\qqq^{d_2}$.
\end{itemize}

\subsubsection{Miki automorphism}\label{MikiAut}
%
We recall that $U_{\qqq,\ddd}(\ddot{\mathfrak{sl}}_r)$ contains two copies of the quantum affine algebra $U_{\qqq}(\dot{\mathfrak{sl}}_r)$.
The first, called the \textit{vertical} copy, is generated by currents where $i\not=0$.
This copy is given in the new Drinfeld presentation.
On the other hand, the second copy, called the \textit{horizontal} copy, is generated by the constant terms of all the currents.
This copy is given in the Drinfeld-Jimbo presentation.
We do not go into detail on these two subalgebras as we will not need them in the sequel.
However, we mention them because they give the `two loops' of the quantum toroidal algebra.
Let $\eta$ denote the $\C(\qqq)$-linear antiautomorphism of $'\ddot{U}'$ defined by
\begin{equation}
\begin{gathered}
    \eta(\ddd)=\ddd^{-1}\\
\eta(e_{i,k})=e_{i,-k},\, \eta(f_{i,k})=f_{i,-k},\, \eta(h_{i,k})=-\gamma^kh_{i,-k},\\
\eta(\psi_{i,0})=\psi_{i,0}^{-1},\,\eta(\gamma^{\frac{1}{2}})=\gamma^{\frac{1}{2}}.
\end{gathered}
\label{EtaDef}
\end{equation}
The following beautiful result of Miki gives the `$S$-transformation' of the torus:
\begin{thm}[\cite{Miki}]\label{MikiAutThm}
There is an algebra automorphism $\varsigma$ of $'\ddot{U}'$ that sends the horizontal copy of $U_{\qqq}(\dot{\mathfrak{sl}}_r)$ to the vertical copy. 
Moreover, $\varsigma$ satisfies $\varsigma^{-1}=\eta\varsigma\eta$.
\end{thm}

\subsubsection{Heisenberg subalgebras}
Recall the generators $\{h_{i,n}\}_{i\in I}^{n\not=0}$ defined by (\ref{HeisDef}).
Together with $\gamma^{\pm\frac{1}{2}}$, these elements generate a rank $r$ Heisenberg algebra.
The relations are
\begin{gather}
\label{HeisRel}     [h_{i, n}, h_{j,n'}]=\delta_{n,-n'}\frac{(\gamma^n-\gamma^{-n})\ddd^{-nm_{i,j}}[na_{i,j}]_{\qqq}}{(\qqq-\qqq^{-1})n}\\
\nonumber    \gamma^{\frac{1}{2}}\hbox{ is central}
\end{gather}
where $[n]_v$ is the usual quantum number:
\[
[n]_v=\frac{v^n-v^{-n}}{v-v^{-1}}
\]

We define \textit{dual elements} $\{h_{i,n}^\perp\}_{i\in I}^{n\not=0}$ by
\begin{equation}
\label{DualForm}
\begin{aligned}
h_{i,n}^\perp&=
\frac{\qqq^{n}(\qqq-\qqq^{-1})n}{(1-\qqq^{nr}\ddd^{nr})(1-\qqq^{nr}\ddd^{-nr})[n]_\qqq}
\sum_{j,k=0}^{r-1} \qqq^{n(j+k)}\ddd^{n(j-k)}h_{i+j-k, n}\\
h_{i,-n}^\perp&=
\frac{\qqq^{n}(\qqq-\qqq^{-1})n}{(1-\qqq^{nr}\ddd^{nr})(1-\qqq^{nr}\ddd^{-nr})[n]_\qqq}
\sum_{j,k=0}^{r-1} \qqq^{n(j+k)}\ddd^{n(j-k)}h_{i-j+k,-n}
\end{aligned}   
\end{equation}

\begin{lem}
   The elements $\{h_{i,n}^\perp\}$ are characterized by
 \begin{equation}
[h_{i,n}^\perp, h_{j,-n'}]=[h_{j,n'}, h_{i,-n}^\perp]=\delta_{i,j}\delta_{n,n'}(\gamma^n-\gamma^{-n})
\label{DualHeis}
\end{equation}  
for $k>0$.
\end{lem}

\begin{proof}
    Equations (\ref{DualHeis}) obviously characterizes these elements.
    For $n>0$, let $M_n$ be the matrix $r\times r$ matrix
    \[
    \left(M_n\right)_{ij}=\ddd^{-nm_{i,j}}[na_{i,j}]_{\qqq}
    \]
    We view the rows and coloumns as indexed by $I$.
    Equation (\ref{HeisRel}) can be rephrased as
    \[
        [h_{i,n}, h_{j,-n}]= \left(M_{n}\right)_{ij}\frac{(\gamma^n-\gamma^{-n})}{(\qqq-\qqq^{-1})n}
    \]
    For any $r\times r$ matrix $A$ (with rows and columns indexed by $I$), set
    \begin{align*}
        Ah_{i,n}&=\sum_{k\in I}A_{ki}h_{k,n},&
        Ah_{i,-n}&=\sum_{k\in I}A_{ki}h_{k,-n}.
    \end{align*}
    We then have for $n>0$,
    \begin{align*}
        \left[Ah_{i,n}, h_{j,-n}\right]&=(A^TM_n)_{ij}\frac{(\gamma^n-\gamma^{-n})}{(\qqq-\qqq^{-1})n},&
        \left[h_{j,n}, Ah_{i,-n}\right]&=\left(M_nA\right)_{ji}\frac{(\gamma^n-\gamma^{-n})}{(\qqq-\qqq^{-1})n}.
    \end{align*}

    Thus, to obtain (\ref{DualHeis}), we need to invert $M_n$.
    To that end, we factorize $M_n$:
    \begin{align*}
        M_n&=[n]_\qqq\qqq^{-n}
        \begin{pmatrix}
            (\qqq^{2n}+1) & -\qqq^n\ddd^n & 0&\cdots & 0 & -\qqq^n\ddd^{-n}\\
            -\qqq^n\ddd^{-n} & (\qqq^{2n}+1) & -\ddd^n & 0 &\cdots & 0\\
            0 & -\qqq^n\ddd^{-n} & (\qqq^{2n}+1) & -\qqq^n\ddd&\ddots &0\\
            \vdots & \ddots& \ddots& \ddots& \ddots & \vdots&\\
            0 & \cdots & 0 & -\qqq^n\ddd^{-n} & (\qqq^{2n}+1) &-\qqq^n\ddd^n\\
            -\qqq^n\ddd^n & 0 &\cdots & 0& -\qqq^n\ddd^{-n} & (\qqq^{2n}+1)
        \end{pmatrix}\\
        &=[n]_\qqq\qqq^{-n}
        \begin{pmatrix}
            1 &  0&\cdots & 0 & -\qqq^n\ddd^{-n}\\
            -\qqq^n\ddd^{-n} & 1 &  0 & & 0\\
            0 & -\qqq^{n}\ddd^{-n} & 1 & \ddots &0\\
            \vdots & \ddots& \ddots& \ddots& 0&\\
            0 & \cdots & 0 & -\qqq^{n}\ddd^{-n} & 1 \\
        \end{pmatrix}
        \begin{pmatrix}
            1 & -\qqq^n\ddd^n & 0&\cdots & 0 \\
            0 & 1 & -\qqq^n\ddd^n & \ddots & \vdots\\
            \vdots & \ddots & 1 & \ddots &0\\
            0 & \cdots & 0 & \ddots &-\qqq^{n}\ddd^n\\
            -\qqq^{n}\ddd^n & 0 &\cdots & 0& 1
        \end{pmatrix}.
    \end{align*}
    Inverting the last two matrices, we obtain (\ref{DualForm}).
\end{proof}

We denote by $\ddot{U}^0$ the subalgebra generated by $\{\gamma^{\pm\frac{1}{2}}\}\cup\{h_{i, k}\}_{i\in I}^{k\not= 0}$ an call it the \textit{vertical Heisenberg subalgebra}.
In analogy with \ref{MikiAut}, we call $\varsigma(\ddot{U}^0)$ the \textit{horizontal Heisenberg subalgebra}.

\begin{rem}\label{PerpRem}
In \cite{Wen}, the author defines elements $\{b_{i,k}^\perp\}$ in terms of a pairing that is not used in this paper.
By comparing the commutator (\ref{HeisRel}) to the pairing in \textit{loc. cit.}, we have that
\[
h_{i,k}^\perp=-b_{i,k}^\perp
\]
\end{rem}

\subsection{Vertex representation}
$U_{\qqq,\ddd}(\ddot{\mathfrak{sl}}_r)$ directly interacts with the wreath Macdonald polynomials via its \textit{vertex representation}, originally constructed by Yoshihisa Saito \cite{Saito}.

\subsubsection{Twisted group algebra}
Recall that $Q$ and $Q^\vee$ denote the $\mathfrak{sl}_r$ root and coroot lattices, respectively, with simple roots $\{\alpha_j\}_{j=1}^{r-1}$, simple coroots $\{\alpha^\vee_j\}_{j=1}^{r-1}$, and canonical pairing $(-,-) : Q^\vee\times Q \to \Z$:
\[
(\alpha_i^\vee,\alpha_j)=a_{i,j}.
\] 
Let $P$ denote the $\mathfrak{sl}_r$ weight lattice and $\{\Lambda_p\}_{j=1}^{r-1}$ the fundamental weights. 
We will also need
\[
\alpha_0=-\sum_{j=1}^{r-1}\alpha_j,\quad
\alpha^\vee_0=-\sum_{j=1}^{r-1}\alpha^\vee_j,\quad
\Lambda_0:=0.
\]
We have that $\{\alpha_2,\ldots,\alpha_{r-1},\Lambda_{r-1}\}$ is a basis of $P$. 

The \textit{twisted group algebra} $\F\{P\}$ is the $\F$-algebra generated by $\{e^{\alpha_j}\}_{j=2}^{r-1}\cup\{e^{\Lambda_{r-1}}\}$ satisfying the relations
\begin{align*}
e^{\alpha_i}e^{\alpha_j}&=(-1)^{\left( \alpha^\vee_i,\alpha_j\right)}e^{\alpha_j}e^{\alpha_i}\\
e^{\alpha_i}e^{\Lambda_{r-1}}&=(-1)^{\delta_{i,r-1}}e^{\Lambda_{r-1}}e^{\alpha_i}.
\end{align*}
Given a general $\alpha\in P$, we write $\alpha=\sum_{j=2}^{r-1}m_j\alpha_j+m_{r}\Lambda_{r-1}$ and then set
\[e^{\alpha}=e^{m_2\alpha_2}\cdots e^{m_{r-1}\alpha_{r-1}}e^{m_r\Lambda_{r-1}}.\]
For example,
\begin{align}
\begin{split}
e^{\alpha_1}&=e^{-2\alpha_2}e^{-3\alpha_3}\cdots e^{-(r-1)\alpha_{r-1}}e^{r\Lambda_{r-1}}\\
e^{\alpha_0}&=e^{\alpha_2}e^{2\alpha_3}\cdots e^{(r-2)\alpha_{r-1}}e^{-r\Lambda_{r-1}}.
\end{split}
\label{TwistedAlpha}
\end{align}
Define $\F\{Q\}$ to be the subalgebra of $\F\{P\}$ generated by $\{e^{\alpha_i}\}_{i=1}^{r-1}$. 

\subsubsection{Vertex operators}
The vertical Heisenberg subalgebra $\ddot{U}^0$ has a Fock representation $F_r$ defined as follows.
Let $\ddot{U}^0_+$ denote the subalgebra generated by $\gamma^{\frac{1}{2}}$ and $\{h_{i,k}\}_{i\in I}^{k>0}$.
$\ddot{U}^0_+$ has a one-dimensional representation $\F_{\qqq}$ where $\gamma^{\frac{1}{2}}$ acts by $\qqq^{\frac{1}{2}}$ while $h_{i,k}$ acts by $0$.
$F_r$ is then the induced representation
\[
F_r:=\mathrm{Ind}_{\ddot{U}^0_+}^{\ddot{U}^0}\F_{\qqq}\cong\K[h_{i,-k}]_{i\in I}^{k>0}.
\]
The vertex representation is defined on the space $W:=F_r\otimes\F\{Q\}$.
For $v\otimes e^\alpha \in W$ where
\begin{align*}
v&=h_{i_1,-k_1}\cdots h_{i_N,-k_N} v_0\\
\alpha&=\sum_{j=1}^{r-1}m_j\alpha_j
\end{align*}
we define the operators $h_{i,k}$, $e^\beta$, $\partial_{\alpha_i}$, $z^{H_{i,0}}$, and $d$ by
\begin{gather}
\nonumber
h_{i,k}(v\otimes e^\alpha ):=(h_{i,k}v)\otimes e^\alpha ,\, e^\beta(v\otimes e^\alpha ):=v\otimes(e^\beta e^\alpha ),\\
\nonumber
\partial_{\alpha_i}(v\otimes e^\alpha ):=\left( \alpha^\vee_i,\alpha\right) v\otimes e^\alpha ,\\
\label{ZHi}
z^{H_{i,0}}(v\otimes e^\alpha ):=z^{\left( \alpha^\vee_i,\alpha\right)} \ddd^{\frac{1}{2}\sum_{j=1}^{r-1}\left( \alpha^\vee_i,m_j\alpha_j\right) m_{i,j}}v\otimes e^\alpha ,\\
\nonumber
d(v\otimes e^\alpha ):=-\left(\dfrac{(\alpha,\alpha)}{2}+\sum_{i=1}^N k_i\right)v\otimes e^\alpha. 
\end{gather}

\begin{thm}[\cite{Saito}]\label{VertexRep}
Let $\vec{c}=(c_0,\ldots, c_{r-1})\in(\F^\times)^r$.
The following formulas endow $W$ with an action of $\ddot{U}'$:
\begin{align*}
\rho_{\vec{c}}(e_i(z))&=c_i\exp\left(\sum_{k>0}\frac{\qqq^{-\frac{k}{2}}}{[k]_{\qqq}}h_{i,-k}z^k\right)\\
&\times\exp\left(-\sum_{k>0}\frac{\qqq^{-\frac{k}{2}}}{[k]_{\qqq}}h_{i,k}z^{-k}\right)e^{\alpha_i} z^{1+H_{i,0}},\\
\rho_{\vec{c}}(f_i(z))&=\frac{(-1)^{r\delta_{i,0}}}{c_i}\exp\left(-\sum_{k>0}\frac{\qqq^{\frac{k}{2}}}{[k]_{\qqq}}h_{i,-k}z^k\right)\\
&\times\exp\left(\sum_{k>0}\frac{\qqq^{\frac{k}{2}}}{[k]_{\qqq}}h_{i,k}z^{-k}\right)e^{-\alpha_i} z^{1-H_{i,0}},\\
\rho_{\vec{c}}(\psi_i^\pm(z))&=\exp\left(\pm(\qqq-\qqq^{-1})\sum_{k>0}h_{i,\pm k}z^{\mp k}\right)\qqq^{\pm\partial_{\alpha_i}},\\
&\rho_{\vec{c}}(\gamma^{\frac{1}{2}})=\qqq^{\frac{1}{2}},\,\rho_{\vec{c}}(\qqq^{d_1})=\qqq^{d}.
\end{align*}
\end{thm}

\subsubsection{Embedding symmetric functions}
We can let $\Lambda^{I}$ act on $F_r$ via multiplication operators given by 
\begin{equation}
p_k[X^{(i)}]\mapsto \frac{k}{[k]_{\qqq}}h_{i,-k}
\label{PowerToBosons}
\end{equation}
for $k>0$.
To obtain an identification $W\cong \Lambda^{I}\otimes\F\left\{ Q \right\}$, we need to embed $\K$ into $\F$:
\begin{equation}
q=\qqq\ddd,\, t=\qqq\ddd^{-1}.
\label{qdtosu}
\end{equation}
Applying $\rho_{\vec{c}}$ to (\ref{DualHeis}) sends $\gamma\mapsto\qqq$.
Thus, as operators on $\Lambda^{I}$, we have the identification
\[
p_k[X^{(i)}]^\perp\mapsto kh_{i,k}^\perp.
\]

Now consider transforming the formulas for $\rho_{\vec{c}}$ using matrix plethysms on $\{p_k[X^{(i)}]\}$.
We can obtain an isomorphic representation as long as we perform a corresponding transformation on $\{h_{i,k}\}$ to maintain the commutation relations, using (\ref{HeisRel}) as a guide.
First, we define $\rho_{\vec{c}}^+$ by performing the plethysm
\[
p_k[X^{(i)}]\mapsto \qqq^{\frac{k}{2}}\left(p_k[X^{(i)}]-t^{-k}p_k[X^{(i-1)}]\right).
\]
For $\rho_{\vec{c}}^+$, we will only be interested in the currents $\{e_i(z)\}$, although we have a representation for the entire algebra:
\begin{align}
\label{ECurrentDef}
\begin{split}
E_i(z):=\rho_{\vec{c}}^+(e_i(z))&=c_i\exp\left[\sum_{k>0}\left(p_{k}[X^{(i)}]-t^{-k}p_{k}[X^{(i-1)}]\right)\frac{z^k}{k}\right]\\
&\times\exp\left[\sum_{k>0}\left(-p_{k}[X^{(i)}]^\perp+q^{-k}p_k[X^{(i-1)}]^\perp\right)\frac{z^{-k}}{k}\right]e^{\alpha_i} z^{1+H_{i,0}}.
\end{split}
\end{align}
Similarly, we define $\rho_{c}^-$ by performing the plethysm
\[
p_k[X^{(i)}]\mapsto \qqq^{-\frac{k}{2}}\left(t^{k}p_k[X^{(i)}]-p_k[X^{(i-1)}]\right).
\]
Here, we will only be interested in the action of the currents $\{ f_i(z)\}$:
\begin{align}
\label{FCurrentDef}
\begin{split}
F_i(z):=\rho_{\vec{c}}^-(f_i(z))&=
\frac{(-1)^{r\delta_{i,0}}}{c_i}\exp\left[\sum_{k>0}\left(-t^{k}p_{k}[X^{(i)}]+p_{k}[X^{(i-1)}]\right)\frac{z^k}{k}\right]\\
&\times\exp\left[\sum_{k>0}\left(q^{k}p_{k}[X^{(i)}]^\perp-p_{k}[X^{(i-1)}]^\perp\right)\frac{z^{-k}}{k}\right]e^{-\alpha_i} z^{1-H_{i,0}}.
\end{split}
\end{align}

The following is a consequence of the main result of \cite{Wen}:
\begin{thm}\label{Eigenstates}
Under both representations $\rho_{\vec{c}}^\pm$, $\varsigma(\ddot{U}^0)$ acts diagonally on $\{P_\lambda[ X^\bullet ;q,t^{-1} ]\otimes e^{\kb(\lambda)}\}$.
\end{thm}

\begin{rem}\label{PlethExp}
The paper \cite{Wen} is concerned with the \textit{transformed} wreath Macdonald functions $\{H_\lambda[X^\bullet ;q,t]\}$.
The plethysms used to define $\rho_{\vec{c}}^\pm$ are both scalar multiples of the plethysm $\cP_{\id-t^{-1}\chi^{-1}}$ which sends $H_\lambda[X^\bullet; q,t]$ to a scalar multiple of $P_{\lambda}[ X^\bullet;q,t^{-1} ]$.
\end{rem}

\subsubsection{Normal ordering}
Later, we will make use of a particular expression for products of the currents $\{E_i(z)\}$ and $\{F_i(z)\}$.
We will need notation for an ordered product or composition of noncommuting operators $a_1,\ldots, a_m$:
\begin{equation}
    \begin{aligned}
    \overset{\curvearrowright}{\prod_{j=1}^m}a_j&:=a_1 a_2\cdots a_m,&
    \overset{\curvearrowleft}{\prod_{j=1}^m}a_j&:=a_m a_{m-1}\cdots a_1
    \end{aligned}
\label{OrdProd}
\end{equation}
\begin{prop}\label{NormalOrder}
For $p\in I$, we have
\begin{equation}
\begin{aligned}
&\overset{\curvearrowright}{\prod_{a=1}^n}\,\overset{\curvearrowright}{\prod_{i=1}^{r}}E_{p+i}(z_{p+i,a})\\
&=
\left( (-1)^{\frac{(r-2)(r-3)}{2}}\ddd^{\frac{r}{2}-1}\prod_{i\in I} c_i\right)^n\\
&\times\prod_{1\le a<b\le n}\prod_{i\in I}
\frac{\displaystyle\left(1-z_{i,b}/z_{i,a}\right)\left(1-q^{-1}t^{-1}z_{i,b}/z_{i,a}\right)}{\displaystyle\left(1-t^{-1}z_{i+1,b}/z_{i,a}\right)\left(1-q^{-1}z_{i-1,b}/z_{i,a}\right)}\\
&\times\prod_{a=1}^n \frac{z_{p,a}/z_{p+1,a}}
{\left(1-q^{-1}z_{p,a}/z_{p+1,a}\right)
\prod_{i\in I\setminus\{p+1\}}
\left(1-t^{-1}z_{i,a}/z_{i-1,a}\right)}\\
&\times\prod_{i\in I}\exp\left(\sum_{a=1}^n\sum_{k>0}\left(p_{k}[X^{(i)}]-t^{-k}p_k[X^{(i-1)}]\right)\frac{z_{i,a}^{k}}{k}\right)\\
&\times\prod_{i\in I}\exp\left(\sum_{a=1}^n\sum_{k>0}\left(-p_{k}[X^{(i)}]^\perp+q^{-k}p_k[X^{(i-1)}]^\perp\right)\frac{z_{i,a}^{-k}}{k}\right)
\prod_{i\in I}\prod_{a=1}^n z_{i,a}^{H_{i,0}}
\end{aligned}
\label{ENormalOrder}
\end{equation}
where all rational functions are Laurent series expanded assuming 
\begin{equation}
|z_{i,a}|=1,\,|q|>1,\,|t|>1.
\label{NormalEExpand}
\end{equation}
For the $F$-currents, we have
\begin{equation}
\begin{aligned}
&\overset{\curvearrowleft}{\prod_{a=1}^n}\,\overset{\curvearrowleft}{\prod_{i=1}^{r}}F_{p+i}(z_{p+i,a})\\
&=\left( \frac{(-1)^{\frac{(r-2)(r-3)}{2}}}{\displaystyle\ddd^{\frac{r}{2}-1} \prod_{i\in I} c_i}\right)^n\\
&\times\prod_{1\le a<b\le n}\prod_{i\in I}
\frac{\left(1-z_{i,a}/z_{i,b}\right)\left(1-qtz_{i,a}/z_{i,b}\right)}{\left(1-tz_{i-1,a}/z_{i,b}\right)\left(1-qz_{i+1,a}/z_{i,b}\right)}\\
&\times \prod_{a=1}^n\frac{z_{p+1,a}/z_{p,a}}{\left(1-qz_{p+1,a}/z_{p,a}\right)
\prod_{i\in I\setminus\{p\}}
\left(1-tz_{i,a}/z_{i+1,a}\right)}\\
&\times\prod_{i\in I}\exp\left(\sum_{a=1}^n\sum_{k>0}\left(-t^{k}p_{k}[X^{(i)}]+p_k[X^{(i-1)}]\right)\frac{z_{i,a}^{k}}{k}\right)\\
&\times\prod_{i\in I}\exp\left(\sum_{a=1}^n\sum_{k>0}\left(q^{k}p_{k}[X^{(i)}]^\perp-p_k[X^{(i-1)}]^\perp\right)\frac{z_{i,a}^{-k}}{k}\right)
\prod_{i\in I}\prod_{a=1}^n z_{i,a}^{-H_{i,0}}
\end{aligned}
\label{FNormalOrder}
\end{equation}
where all rational functions are Laurent series expanded assuming 
\begin{equation}
|z_{i,a}|=1,\, |q|< 1,\, |t|<1.
\label{NormalFExpand}
\end{equation}
\end{prop}

\begin{proof}
The computation is standard.
We will only go over the signs and powers of $\ddd$.
The sign comes from the commutation of $\{e^{\alpha_i}\}$; in both cases, these factors simplify to $\pm e^0$.
For the $E$-currents, if $p=0$, then by (\ref{TwistedAlpha}),
\[e^{\alpha_1}= e^{r\Lambda_{r-1}}e^{-(r-1)\alpha_{r-1}}\cdots e^{-3\alpha_3}e^{-2\alpha_2}.\]
Thus,
\[
e^{\alpha_1}e^{\alpha_2}\cdots e^{\alpha_{r-1}}=(-1)^{\frac{(r-2)(r-3)}{2}}e^{r\Lambda_{r-1}}e^{-(r-2)\alpha_{r-1}}\cdots e^{-2\alpha_3}e^{-\alpha_2}.
\]
On the other hand, if $p\not=0$, we have
\begin{align*}
e^{\alpha_0}e^{\alpha_1}&= (-1)^{\frac{(r-1)(r-2)}{2}-1}e^{-\alpha_2}\cdots e^{-\alpha_{r-1}}\\
&=(-1)^{\frac{(r-1)(r-2)}{2}+r-3}e^{-\alpha_{r-1}}\cdots e^{-\alpha_{2}}\\
&= (-1)^{\frac{(r-2)(r-3)}{2}}e^{-\alpha_{r-1}}\cdots e^{-\alpha_{2}}
\end{align*}
which also leads to a sign of $(-1)^{\frac{(r-2)(r-3)}{2}}$.
For the $F$-currents, first consider the case $p=0$.
\begin{align*}
e^{-\alpha_0}e^{-\alpha_{r-1}}\cdots e^{-\alpha_3}e^{-\alpha_2}&= (-1)^{r+\frac{(r-2)(r-3)}{2}}e^{-2\alpha_2}e^{-3\alpha_3}\cdots e^{-(r-1)\alpha_{-r-1}}e^{r\Lambda_{r-1}}\\
&= (-1)^{r+\frac{(r-2)(r-3)}{2}}e^{r\Lambda_{r-1}}e^{-(r-1)\alpha_{r-1}}\cdots e^{-3\alpha_2}e^{-2\alpha_2}.
\end{align*}
If $p\not=0$, then we use that
\begin{align*}
e^{-\alpha_1}e^{-\alpha_0}&=(-1)^re^{-\alpha_1}e^{r\Lambda_{r-1}}e^{-(r-2)\alpha_{r-1}}\cdots e^{-2\alpha_3}e^{-\alpha_2}\\
&= (-1)^{r+\frac{(r-2)(r-3)}{2}}e^{\alpha_2}e^{\alpha_3}\cdots e^{\alpha_{r-1}}.
\end{align*}
Finally, note that $F_0(z)$ also has a sign of $(-1)^{r}$.
The power of $\ddd$ comes from the interaction between $\{z^{\pm H_{i,0}}\}$ and $\{e^{\pm\alpha_{j}}\}$.
First observe that when considering $E_i(z_{i,a})$ and $E_{j}(z_{j,b})$ for $a\not=b$, the powers of $\ddd$ from $j=i-1$ and $j=i+1$ cancel out.
When $a=b$, there is a total power of $\ddd^{\frac{r}{2}-1}$.
The case for $\{F_i(z)\}$ is similar but inverted.
\end{proof}
\subsection{Fock representation}
While our main focus will be on the vertex representation, we will consider another representation of $U_{\qqq,\ddd}(\ddot{\mathfrak{sl}}_r)$, called the \textit{Fock representation}.
Our goal will be gain some knowledge on the eigenvalues implicit in the statement of Theorem \ref{Eigenstates}.

\subsubsection{Definition}
In order to define the Fock representation, we will need some notation for partitions.
For a partition $\lambda$, let $\square=(a,b)\in D(\lambda)$.
We set:
\begin{enumerate}
\item $\chi_\square=q^{a}t^{b}$, the character of the box;
\item $c_\square= b-a$ modulo $r$ (its color);
\item $d_i(\lambda)$ the number of elements of $D(\lambda)$ with content equivalent to $i$ modulo $r$;
\item $A_i(\lambda)$ and $R_i(\lambda)$ the addable and removable $i$-nodes of $\lambda$, respectively.
\end{enumerate}
Finally, we will abbreviate $a\equiv b\hbox{ mod }r$ by simply $a\equiv b$ and use the Kronecker delta function $\delta_{a= b}:=\delta_{a-b,0}$.

Let $v\in\F^\times$.
The Fock representation $\mathcal{F}(v)$ has a basis $\{|\lambda\rangle\}$ indexed by partitions.
\begin{thm}[\cite{FJMM}, cf. \cite{Wen}]\label{FockRep}
We can define a $'\ddot{U}$-action $\tau_v$ on $\mathcal{F}(v)$ where the only nonzero matrix elements of the generators are
\begin{gather*}
\begin{aligned}
\langle\lambda | e_i(z)|\lambda+\square\rangle&=\delta_{c_\square=i}(-\ddd)^{d_{i+1}(\lambda)}\delta\left( \frac{ z}{\chi_\square v}\right)
\frac{\displaystyle\prod_{\blacksquare\in R_{i}(\lambda)}\left( \chi_\square-\qqq^2\chi_\blacksquare \right)}
{\displaystyle\prod_{\substack{\blacksquare\in A_{i}(\lambda)\\\blacksquare\not=\square}}\left(\chi_\square-\chi_\blacksquare\right)}\\
\langle\lambda+\square |f_i(z)|\lambda\rangle&=\delta_{c_\square=i}(-\ddd)^{-d_{i+1}(\lambda)}\delta\left(\frac{ z}{\chi_\square v}\right)
\frac{\displaystyle\prod_{\substack{\blacksquare\in A_{i}(\lambda)\\\blacksquare\not=\square}}\left( \qqq\chi_\square-\qqq^{-1}\chi_\blacksquare \right)}
{\displaystyle\prod_{\blacksquare\in R_{i}(\lambda)}\qqq\left( \chi_\square-\chi_\blacksquare \right)}\\
\langle\lambda|\psi_i^\pm(z)|\lambda\rangle&=
\prod_{\blacksquare\in A_{i}(\lambda)}\frac{\left(\qqq z-\qqq^{-1}\chi_{\blacksquare}v\right)}{\left( z-\chi_\blacksquare v\right)}
\prod_{\blacksquare\in R_{i}(\lambda)}\frac{\left(\qqq^{-1} z-\qqq\chi_\blacksquare v\right)}{\left( z-\chi_\blacksquare v\right)},
\end{aligned}\\
\langle\lambda|\gamma^{\frac{1}{2}}|\lambda\rangle=1,\,\langle\lambda|\qqq^{d_2}|\lambda\rangle=\qqq^{-|\lambda|}.
\end{gather*}
\end{thm}

%
\subsubsection{Tsymbaliuk isomorphism}
The representation $\tau_v$ on $\mathcal{F}(v)$ has a cyclic vector $|\varnothing\rangle$.
On the other hand, $\rho_{\vec{c}}$ and $\rho_{\vec{c}}^\pm$ also has the natural cyclic vector $1\otimes 1\in F_r\otimes \F\{Q\}$.
The following theorem was proved by Tsymbaliuk:
\begin{thm}[\cite{T}]\label{TsymIso}
Let
\begin{equation}
v=(-1)^{\frac{(\ell-2)(\ell-3)}{2}}\frac{\qqq\ddd^{-\frac{\ell}{2}}}{c_0\cdots c_{\ell-1}}
\label{FockVertexPar}
\end{equation}
The map of cyclic vectors
\[
\mathcal{F}(v)\ni|\varnothing\rangle\mapsto 1\otimes 1\in W
\]
induces an isomorphism between the $'\ddot{U}'$-module $\tau_v$ and the $\varsigma$-twisted modules $\rho_{\vec{c}}\circ\varsigma$, $\rho_{\vec{c}}^\pm\circ\varsigma$.
\end{thm}
\noindent The Tsymbaliuk isomorphism is defined only in terms of cyclic vectors.
In light of Remark \ref{PlethExp}, the following result from \cite{Wen} provides more detail on the Tsymbaliuk isomorphisms:
\begin{thm}
The Tsymbaliuk isomorphisms (Theorem \ref{TsymIso}) between $\tau_v$ and $\rho_{\vec{c}}^\pm$ send 
\[\F|\lambda\rangle\to\F \left(P_\lambda[X^\bullet;q,t^{-1}]\otimes e^{\kb(\lambda)}\right).\]
\end{thm}
\noindent Thus, we can study the eigenvalues of $\varsigma(\ddot{U}^0)$ on $P_\lambda$ by instead studying the eigenvalues of $\ddot{U}^0$ on the basis $\{|\lambda\rangle\}$.

\subsubsection{Infinite-variable eigenvalues}
From the formulas in Theorem \ref{FockRep}, we can see that
\[
\langle\lambda|\psi_{i,0}^{\pm 1}|\lambda\rangle= \qqq^{\pm\left(|A_{i}(\lambda)|-|R_{i}(\lambda)|\right)}.
\]
Therefore,
\begin{align*}
&\left\langle\lambda\left|\exp\left(\pm(\qqq-\qqq^{-1})\sum_{k>0}h_{i,\pm k}z^{\mp k}\right)\right|\lambda\right\rangle\\
&=\prod_{\blacksquare\in A_{i}(\lambda)}\frac{\qqq^{\mp 1}\left(\qqq z-\qqq^{-1}\chi_{\blacksquare}v\right)}{\left( z-\chi_\blacksquare v\right)}
\prod_{\blacksquare\in R_{i}(\lambda)}\frac{\qqq^{\pm 1}\left(\qqq^{-1} z-\qqq\chi_\blacksquare v\right)}{\left( z-\chi_\blacksquare v\right)}\\
&= \exp\left[ \sum_{k>0}\left( \sum_{\blacksquare\in A_{i}(\lambda)}(1-\qqq^{\mp 2k})\chi_{\blacksquare}^{\pm k} +\sum_{\blacksquare\in R_{i}(\lambda)}(1-\qqq^{\pm 2k})\chi_\blacksquare^{\pm k}\right)\frac{v^{\pm k}z^{\mp k}}{k} \right].
\end{align*}
Taking logarithms, we see that for $k>0$,
\begin{equation}
\begin{aligned}
\langle\lambda|h_{i,\pm k}|\lambda\rangle &= \frac{v^{\pm k}[ k]_{\qqq}}{k}\left( \sum_{\blacksquare\in A_{i}(\lambda)}\qqq^{\mp k}\chi_\blacksquare^{\pm k} -\sum_{\blacksquare\in R_{i}(\lambda)}\qqq^{\pm k}\chi_\blacksquare^{\mp k}\right)\\
&= \frac{v^{\pm k}\qqq^{\mp k}[ k]_{\qqq}}{k}\left( \sum_{\blacksquare\in A_{i}(\lambda)}\chi_\blacksquare^{\pm k} - \sum_{\blacksquare\in R_{i}(\lambda)}(qt\chi_\blacksquare)^{\pm k}\right).
\end{aligned}
\label{HEigen}
\end{equation}
Using (\ref{HEigen}), we can try to piece together elements of $\ddot{U}^0$ whose eigenvalues are elementary symmetric functions in $\{q^{\pm\lambda_b}t^{\pm b}\}$.

For $k\in \Z_{>0}$ and $p\in I$, let us define
\begin{align}
    \label{Hhat+}
    \hat{h}_{p,k}&:=\frac{1}{(1-t^{kr})}\sum_{i=0}^{r-1}t^{ k(i+1)}h_{p-i, k}\\
    \label{Hhat-}
    \hat{h}_{p,-k}&:= \frac{1}{(1-t^{- kr})}\sum_{i=0}^{r-1}t^{- k(i+1)}h_{p-i,- k}
\end{align}
\begin{lem}\label{FockEigen}
Assume $|t^{\pm 1}|<1$ (where `$+$' and `$-$' are separate cases).
For $p\in I$, we have
\begin{equation}
\begin{aligned}
&\left\langle\lambda\left|\exp\left[-\sum_{k>0}\hat{h}_{p,\pm k} (-z)^{\mp k}{v^{\pm k}[ k]_{\qqq}} \right]\right|\lambda\right\rangle\\
&= \exp\left[-\sum_{k>0}\left(\sum_{\substack{b>0\\b-\lambda_b\equiv p+1 }}q^{\pm k\lambda_b}t^{\pm kb}\right)\frac{(-z)^{\mp k}}{k}\right]\\
&= \prod_{\substack{b>0\\b-\lambda_b\equiv p+1 }}\left(1+q^{\pm\lambda_b}t^{\pm b}z^{\mp 1}\right)
\end{aligned}
\label{EigenLemma}
\end{equation}
where we set $\lambda_b=0$ for all $b>\ell(\lambda)$.
\end{lem}

\begin{proof}
Comparing (\ref{EigenLemma}) to (\ref{HEigen}), we need to establish the equality
\begin{equation}
\frac{1}{1-t^{\pm kr}}\sum_{i=0}^{r-1}t^{\pm k(i+1)}\left( \sum_{\blacksquare\in A_{p-i}(\lambda)}\chi_\blacksquare^{\pm k} - \sum_{\blacksquare\in R_{p-i}(\lambda)}(qt\chi_\blacksquare)^{\pm k}\right)
=\left(\sum_{\substack{b>0\\b-\lambda_b\equiv p+1 }}q^{\pm k\lambda_b}t^{\pm kb}\right).
\label{AddRemoveEnd}
\end{equation}
We note that here, we consider $(1-t^{\pm kr})^{-1}$ as a geometric series.
The summands on the right hand side of (\ref{AddRemoveEnd}) are $qt$-shifts of the characters of color $p+1$ boxes that are the rightmost boxes in their row.
We can account for these coordinates by starting at each addable box of $D(\lambda)$, going straight up until we reach a box of color $p+1$, then moving upwards by intervals of $r$, and ending the search once we are above the $qt$-shift of the removable box above it.
This is exactly what the left hand side of (\ref{AddRemoveEnd}) does.
We illustrate this with Figure \ref{fig:EigenLem}.
\end{proof}

\begin{figure}
    \centering
    \begin{tikzpicture}[scale=0.5]
        \draw (1, 6) node {$\lambda$};;
        \draw[thick] (0,7)--(2,7)--(2,1)--(4,1)--(4,0);;
        \draw[fill=black] (2,1)--(3,1)--(3,2)--(2,2)--(2,1);;
        \draw[fill=black] (2,7)--(3,7)--(3,8)--(2,8)--(2,7);;
        \draw[fill=lightgray] (2,3)--(3,3)--(3,4)--(2,4)--(2,3);;
        \draw[fill=lightgray] (2,6)--(3,6)--(3,7)--(2,7)--(2,6);;
        \draw[fill=lightgray] (2,9)--(3,9)--(3,10)--(2,10)--(2,9);;
        \draw[thick] (3.5,8.5)--(1.5,10.5);;
        \draw[thick] (3.5,10.5)--(1.5,8.5);;
    \end{tikzpicture}
    \caption{Illustration of the proof of Lemma \ref{EigenLemma}.
    The $t$-shifts on the addable black box at the bottom results in the gray boxes.
    The latter are evenly spaced of interval $r$ and have the desired color.
    The black box at the top is $qt$ times a removable box, and subtracting its $t$-shifts cancels out the extraneous gray boxes.}
    \label{fig:EigenLem}
\end{figure}
%

\section{Shuffle algebra}\label{Shuffle}
We will obtain difference operators by computing the action of $\varsigma(\ddot{U}^0)$ on the vertex representation.
However, computing the images of elements under $\varsigma$ is difficult.
The \textit{shuffle algebra} provides another avatar of the quantum toroidal algebra with which we can access the horizontal Heisenberg subalgebra.

\subsection{Definition and structures}
Let $k_\bullet=(k_0,\ldots, k_{r-1})\in\left( \Z_{\ge 0} \right)^I$ and consider the function spaces: 
\begin{align*}
\mathbb{S}_{k_\bullet}&:=\F(z_{i,a})_{i\in I}^{1\le a\le k_i}\\   
\mathbb{S}&:=\bigoplus_{k_\bullet\in\left(\Z_{\ge0}\right)^I}\mathbb{S}_{\vec{k}}.
\end{align*}
The product of symmetric groups
\[
\Sym_{k_\bullet}:=\Sym_{k_0}\times\cdots\times\Sym_{k_{r-1}}
\]
acts on $\mathbb{S}_{k_\bullet}$ where the factor $\Sym_{k_i}$ only permutes the variables $\{z_{i,a}\}_{a=1}^{k_i}$.
We call $i$ the \textit{color} of $z_{i,r}$, so $\Sym_{k_\bullet}$ acts by \textit{color-preserving permutations}.
Finally, let
\begin{align*}
\mathbf{S}_{k_\bullet}&:=\left(\mathbb{S}_{k_\bullet}\right)^{\Sym_{k_\bullet}}\\
\mathbf{S}&:=\bigoplus_{k_\bullet\in\left( \Z_{\ge 0} \right)^I}\mathbf{S}_{k_\bullet}.
\end{align*}
Unless we say otherwise, an element of $\mathbb{S}$ with $k_i$ variables of color $i$ for all $i$ is assumed to be in $\mathbb{S}_{k_\bullet}$.

\subsubsection{Shuffle product}
We endow $\mathbf{S}$ with the \textit{shuffle product} $\star$, defined as follows.
For $i,j\in I$, we define the \textit{mixing terms}:
\[\omega_{i,j}(z,w):=\left\{\begin{array}{ll}
\left(z-\qqq^{2}w\right)^{-1}\left(z-w\right)^{-1} & \hbox{if }i=j\\
\left(\qqq w-\ddd^{-1}z\right) &\hbox{if }i+1=j\\
\left(z-\qqq\ddd^{-1} w\right) &\hbox{if }i-1=j\\
1 &\hbox{otherwise.}
\end{array}\right.\]
For $F\in\mathbf{S}_{k_\bullet}$ and $G\in\mathbf{S}_{l_\bullet}$, let $F\star G\in\mathbf{S}_{k_\bullet+l_\bullet}$ be defined by
\[
F\star G:=\frac{1}{k_\bullet!l_\bullet!}\mathrm{Sym}_{k_\bullet+l_\bullet}\left[ 
F\left( \left\{ z_{i,a} \right\}_{i\in I}^{1\le a\le k_i} \right)G\left( \left\{ z_{j, b} \right\}_{j\in I}^{k_j<b\le k_j+l_j} \right)
\prod_{i,j\in I}\prod_{\substack{1\le a\le k_i\\k_j<b\le k_j+l_j}}\omega_{i,j}(z_{i,a},z_{j,b})
 \right]
\]
where for $n_\bullet\in\left( \Z_{\ge 0} \right)^I$,
\[
n_\bullet!=\prod_{i\in I}n_i!=|\Sym_{n_\bullet}|
\]
and $\mathrm{Sym}_{n_\bullet}$ denotes the \textit{color symmetrization}, i.e. the symmetrization over $\Sym_{n_\bullet}$.

\subsubsection{The shuffle algebra}
Consider now for each $k_\bullet$ the subspace $\mathcal{S}_{k_\bullet}\subset\mathbf{S}_{k_\bullet}$ consisting of functions $F$ satisfying the following two conditions:
\begin{enumerate}
\item \textit{Pole conditions:} $F$ is of the form
\begin{equation}
F=\frac{f(\{z_{i,r}\})}{\displaystyle \prod_{i\in I}\,\prod_{\substack{1\le r, r'\le k_i\\r\not= r'}}(z_{i,r}-\qqq^2z_{i,r'})}
\label{PoleCond}
\end{equation}
for a color-symmetric \textit{Laurent polynomial} $f$.
\item \textit{Wheel conditions:} $F$ has a well-defined finite limit when
\[\frac{z_{i,r_1}}{z_{i+\epsilon,s}}\rightarrow\qqq\ddd^\epsilon\hbox{ and }\frac{z_{i+\epsilon,s}}{z_{i,r_2}}\rightarrow\qqq\ddd^{-\epsilon}\]
for any choice of $i$, $r_1$, $r_2$, $s$, and $\epsilon$, where $\epsilon\in\{\pm 1\}$. 
This is equivalent to specifying that the Laurent polynomial $f$ in the pole conditions evaluates to zero at
\begin{equation*}
    \begin{aligned}
    z_{i,r_1}&=\qqq\ddd^{\epsilon}z_{i+\epsilon,s},&
    z_{i+\epsilon,s}&=\qqq\ddd^{-\epsilon}z_{i,r_2}.       
    \end{aligned}
\end{equation*}
\end{enumerate}
We set 
\[
\mathcal{S}:=\bigoplus_{k_\bullet\in\left( \Z_{\ge 0} \right)^I}\mathcal{S}_{k_\bullet}.
\]
The following is standard:
\begin{prop}[ \protect{\cite[Proposition 3.3]{Neg}}]
The shuffle product $\star$ defines an associative product on $\mathbf{S}$ and $\mathcal{S}$ is closed under $\star$.
\end{prop}
\noindent We call $(\mathcal{S},\star)$ the \textit{shuffle algebra of type }$\hat{A}_{r-1}$.

\subsubsection{Relation to $U_{\qqq,\ddd}(\ddot{\mathfrak{sl}_r})$}
Let
\begin{itemize}
\item $\ddot{U}^+\subset U_{\qqq,\ddd}(\ddot{\mathfrak{sl}_r})$ be the subalgebra generated by $\left\{ e_{i}(z) \right\}_{i\in I}$ and
\item $\ddot{U}^-\subset U_{\qqq,\ddd}(\ddot{\mathfrak{sl}_r})$ be the subalgebra generated by $\left\{ f_{i}(z) \right\}_{i\in I}$.
\end{itemize}
Correspondingly, we set $\mathcal{S}^+:=\mathcal{S}$ and $\mathcal{S}^-:=\mathcal{S}^{op}$.
The following key structural result was proved by Negu\cb{t}:
\begin{thm}[\cite{Neg}]\label{ShuffTor}
$\mathcal{S}^\pm$ is generated by $\{z_{i,1}^n\}_{i\in I}^{n\in\Z}$ and
\begin{align*}
\Psi_+(z_{i,1}^n)&= e_{i,n}\\
\Psi_-(z_{i,1}^n)&= f_{i,n}.
\end{align*}
induce algebra isomorphisms $\Psi_{\pm}:\mathcal{S}^\pm\rightarrow\ddot{U}^\pm$.
\end{thm}
Finally, note that the subalgebras $\ddot{U}^\pm$ are each closed under $\eta$.
We will need to understand how the antiautomorphism $\eta$ is manifested on the shuffle side:
\begin{prop}\label{EtaShuff}
For $F\in\mathcal{S}^\pm_{k_\bullet}$, define:
\[
\eta_{\mathcal{S}}(F):=
 \left.F(z_{i,r}^{-1})\prod_{i\in I}\prod_{r=1}^{k_i}(-\ddd)^{k_{i+1}k_i} z_{i,r}^{k_{i+1}+k_{i-1}-2(k_i-1)}\right|_{\ddd\mapsto\ddd^{-1}}
\]
We have
\begin{equation}
\Psi_+^{-1}\eta\Psi_+(F)=\Psi_-^{-1}\eta\Psi_-(F)=\eta_{\mathcal{S}}(F) 
.
\label{EtaShuffEq}
\end{equation}
\end{prop}

\begin{proof}
Equation (\ref{EtaShuffEq}) is true when $F=z_{i,1}^n$ is a generator.
To see that it is a $\C(\qqq)$-linear algebra antiautomorphism that inverts $\ddd$, we first observe that
\begin{equation}
 \begin{aligned}
    z^{-2}w^{-2}\omega_{i,i}(z^{-1},w^{-1})\bigg|_{\ddd\mapsto\ddd^{-1}}&=\omega_{i,i}(w,z)\\
    zw(-\ddd)\omega_{i,i+1}(z^{-1},w^{-1})\bigg|_{\ddd\mapsto\ddd^{-1}}&=\omega_{i+1,i}(w,z)\\
    zw(-\ddd)\omega_{i+1,i}(z^{-1}, w^{-1})\bigg|_{\ddd\mapsto\ddd^{-1}}&=\omega_{i,i+1}(w,z)
\end{aligned} 
\label{OmegaInv}
\end{equation}
Now, for $F\in\mathcal{S}_{k_\bullet}^+$ and $G\in\mathcal{S}_{l_\bullet}^+$:
\begin{align}
        \nonumber
        &\eta_{\mathcal{S}}(F\star G)\\
        \nonumber
        &=\frac{1}{k_\bullet!l_\bullet!}\mathrm{Sym}\left[F\left(\{z_{i,a}^{-1}\}_{i\in I}^{1\le a\le k_i}\right)
        G\left(\{z_{j,b}^{-1}\}_{j\in I}^{k_j<b\le k_j+l_j}\right)
        \prod_{i,j\in I}\prod_{\substack{1\le a\le k_i\\k_j<b\le k_j+l_j}}\omega_{i,j}(z_{i,a}^{-1},z_{j,b}^{-1})\right]\\
        \label{Mult}
        &\quad\times
        \prod_{i\in I}\prod_{r=1}^{k_i+l_i}(-\ddd)^{(k_{i+1}+l_{i+1})(k_i+l_i)} z_{i,r}^{k_{i+1}+l_{i+1}+k_{i-1}+l_{i-1}-2(k_i+l_i-1)}\bigg|_{\ddd\mapsto\ddd^{-1}}
\end{align}
The monomial in (\ref{Mult}) is color-symmetric, so we can move it inside the symmetrization.
We can break up the exponents appearing in (\ref{Mult}) as follows:
\begin{align}
\label{DInv}
    (k_{i+1}+l_{i+1})(k_i+l_i)&=k_{i+1}k_{i}+l_{i+1}l_i+\left[k_{i+1}l_i+k_il_{i+1}\right]\\
\label{FInv}
    k_{i+1}+l_{i+1}+k_{i-1}+l_{i-1}-2(k_i+l_i-1)&=k_{i+1}+k_{i-1}-2(k_i-1)+[l_{i+1}+l_{i-1}-2l_i]\\
\label{GInv}
    &=l_{i+1}+l_{i-1}-2(l_i-1)+[k_{i+1}+k_{i-1}-2k_i]
\end{align}
In (\ref{DInv}), we will assign the bracketed summand to the mixing terms, $k_{i+1}k_{i}$ to $F$, and $l_{i+1}l_i$ to $G$.
In a given summand of the symmetrization, if $z_{i,r}$ is assigned to $F$, then in (\ref{FInv}), we assign the bracketed summand to the mixing terms and the rest to $F$.
On the other hand, if $z_{i,r}$ is assigned to $G$, then in (\ref{GInv}), we assign the bracketed summand to the mixing terms and the rest to $G$.
Then, applying (\ref{OmegaInv}), we do indeed obtain
\[
\eta_{\mathcal{S}}(G)\star\eta_{\mathcal{S}}(F).
\]
The case where $F,G\in\mathcal{S}^-$ is similar.
\end{proof}

\subsubsection{Shuffle presentation of horizontal Heisenberg elements}
Recall the vertical Heisenberg elements (\ref{EigenLemma}) whose action on $\mathcal{F}(v)$ are related to infinite-variable Macdonald operators.
Previous work \cite{Wen} gives us a better understanding of the action of $\varsigma^{-1}$ on such elements.
However, we need $\varsigma$ instead, and thus we will apply the identity $\varsigma=\eta\varsigma^{-1}\eta$ (cf. Theorem \ref{MikiAutThm}) and Proposition \ref{EtaShuff}.
To that end, recall the elements $\{\hat{h}_{p,\pm k}\}$ from (\ref{Hhat+}) and (\ref{Hhat-}).
Observe that
\begin{equation}
\begin{aligned}
&\varsigma\exp\left[-\sum_{k>0}\hat{h}_{p,\pm k}\frac{\qqq^{\pm k} (-z)^{\mp k}}{[ k]_{\qqq}} \right] \\
&=\varsigma\exp\left[-\sum_{k>0}\left(\frac{\sum_{i=0}^{r-1}t^{\pm k(i+1)}h_{p-i,\pm k}}{(1-t^{\pm kr})}\right)\frac{\qqq^{\pm k} (-z)^{\mp k}}{[ k]_{\qqq}} \right] \\
&=\eta\exp\left[\sum_{k>0}\left(\frac{\sum_{i=0}^{r-1}q^{\pm k(i+1)}\varsigma^{-1}(h_{p-i,\mp k})}{(1-q^{\pm kr})}\right)\frac{\qqq^{\pm k} (-z)^{\mp k}}{[ k]_{\qqq}} \right]\\
&= \eta\exp\left[(\qqq-\qqq^{-1})^{-1}\sum_{k>0}\left(\varsigma^{-1}(h_{p,\mp k}^\perp)-t^{\pm k}\varsigma^{-1}(h_{p+1,\mp k}^\perp)\right)\frac{q^{\pm k} (-z)^{\mp k}}{k} \right]
\end{aligned}
\label{MikiPerp}
\end{equation}
where in the last line, we use (\ref{HeisRel}).
Let $\delta=(1,\ldots, 1)\in\left( \Z_{\ge 0} \right)^I$ be the diagonal vector and consider the elements $\mathcal{E}_{p,n}^\pm\in\mathcal{S}^\pm$ given by
\begin{equation}
\begin{aligned}
\mathcal{E}_{p,n}^+&:=
\mathrm{Sym}_{n\delta}\left( \prod_{1\le a<b\le n}\left\{\frac{z_{p+1,a}-q^{-1}z_{p,b}}{z_{p+1,a}-tz_{p,b}}\prod_{i,j\in I}\omega_{i,j}\left( z_{i,a},z_{j,b} \right)\right\}\right.\\
&\times \left.\prod_{a=1}^n\left\{\left( \frac{z_{0,a}}{z_{p,a}}-q^{-1}\frac{z_{0,a}}{z_{p+1,a}} \right)\prod_{i\in I}z_{i,a}  \right\} \right)\\
\mathcal{E}_{p,n}^-&:=
\mathrm{Sym}_{n\delta}\left( \prod_{1\le a<b\le n}\left\{\frac{z_{p+1,a}-q^{-1}z_{p,b}}{z_{p+1,a}-tz_{p,b}}\prod_{i,j\in I}\omega_{i,j}\left( z_{i,a},z_{j,b} \right)\right\}\right.\\
&\times \left.\prod_{a=1}^n\left\{\left( q\frac{z_{p+1,a}}{z_{0,a}}-\frac{z_{p,a}}{z_{0,a}} \right)\prod_{i\in I}z_{i,a}  \right\} \right).
\end{aligned}
\label{Epn}
\end{equation}
By \cite[Proposition 4.22]{Wen}, $\mathcal{E}_{p,n}^\pm\in\mathcal{S}^\pm$.
\begin{lem}\label{EpnShuff}
We have
\begin{align*}
\sum_{n=0}^\infty \frac{(-1)^{n}\qqq^{n(r-1)}t^{-n}(1-q^{-1}t^{-1})^{nr}}{v^{-n}\prod_{a=1}^n(1-q^{-a}t^{-a})}\Psi_+\left( \mathcal{E}_{p,n}^+ \right)z^{-n} 
&= \varsigma\exp\left[-\sum_{k>0}\hat{h}_{p,-k}\frac{\qqq^{-k} (-z)^{-k}}{v^{- k}[ k]_{\qqq}} \right] \\
\sum_{n=0}^\infty \frac{(-1)^{nr-n}\ddd^{-n(r-1)}t^n(1-qt)^{nr}}{v^{n}q^n\prod_{a=1}^n(1-q^{-a}t^{-a})}\Psi_-\left( \mathcal{E}_{p,n}^- \right)z^n 
&= \varsigma\exp\left[-\sum_{k>0}\hat{h}_{p,k}\frac{\qqq^{ k} (-z)^{ k}}{v^{ k}[ k]_{\qqq}} \right].
\end{align*}
\end{lem}

\begin{rem}
    Note that prior to taking $\varsigma$, the series on the right-hand-sides are the ones appearing in Lemma \ref{FockEigen}.
\end{rem}

\begin{proof}
In \cite{Wen}, it was shown that
\begin{align*}
&\exp\left[(\qqq-\qqq^{-1})^{-1}\sum_{k>0}\left(\varsigma^{-1}(h_{p, k}^\perp)-t^{-k}\varsigma^{-1}(h_{p+1, k}^\perp)\right)\frac{q^{-k} (-z)^{ k}}{k} \right]\\
&= \sum_{n=0}^\infty\frac{(-1)^{nr}(-q)^{-n}t^{nr}(1-q^{-1}t^{-1})^{nr}}{\qqq^n\prod_{a=1}^n(1-q^{-a}t^{-a})}\Psi_+(\mathcal{H}_{p,n}^+)z^n
\end{align*}
and
\begin{align*}
&\exp\left[(\qqq-\qqq^{-1})^{-1}\sum_{k>0}\left(\varsigma^{-1}(h_{p,- k}^\perp)-t^{ k}\varsigma^{-1}(h_{p+1,- k}^\perp)\right)\frac{q^{ k} (-z)^{- k}}{k} \right]\\
&= \sum_{n=0}^\infty\frac{(-q)^n(1-qt)^{nr}}{\qqq^n\prod_{a=1}^n(1-q^{-a}t^{-a})}\Psi_-(\mathcal{H}_{p,n}^-)z^{-n}
\end{align*}
where
\begin{equation}
\begin{aligned}
\mathcal{H}^+_{p,n}&= \mathrm{Sym}_{n\delta}\left( \prod_{1\le a<b\le n}\left\{\frac{t^{-1}z_{p+1,b}-z_{p,a}}{qz_{p+1,b}-z_{p,a}}\prod_{i,j\in I}\omega_{i,j}\left( z_{i,a},z_{j,b} \right)\right\}\right.\\
&\times \left.\prod_{a=1}^n\left\{\left( \frac{z_{p,a}}{z_{0,a}}-t^{-1}\frac{z_{p+1,a}}{z_{0,a}} \right)\prod_{i\in I}z_{i,a}  \right\} \right)\\
\mathcal{H}^-_{p,n}&= \mathrm{Sym}_{n\delta}\left( \prod_{1\le a<b\le n}\left\{\frac{t^{-1}z_{p+1,b}-z_{p,a}}{qz_{p+1,b}-z_{p,a}}\prod_{i,j\in I}\omega_{i,j}\left( z_{i,a},z_{j,b} \right)\right\}\right.\\
&\times \left.\prod_{a=1}^n\left\{\left( t\frac{z_{0,a}}{z_{p+1,a}}-\frac{z_{0,a}}{z_{p,a}} \right)\prod_{i\in I}z_{i,a}  \right\} \right).
\end{aligned}
\label{Hpn}
\end{equation}
It is helpful to recall Remark \ref{PerpRem} when making comparisons with \cite{Wen}.
The result follows from applying Proposition \ref{EtaShuff} to (\ref{MikiPerp}).
We note that the mixing terms in \ref{Hpn} contribute a power of $\ddd^{-r n(n-1)}$ before inverting $\ddd$.
\end{proof}

\subsection{Action on the vertex representation}
For $F\in\mathcal{S}^+$ and $G\in\mathcal{S}^-$, we will present a way to compute the actions of $\rho^+_{\vec{c}}(\Psi_+(F))$ and $\rho_{\vec{c}}^-(\Psi_-(G))$.
Our approach was inspired by Lemma 3.2 of \cite{FJM} in the case $r=1$.
\subsubsection{Matrix elements}
The following is a consequence of computations similar to those done for Proposition \ref{NormalOrder}:
\begin{prop}
For $v_1,v_2\in W$, we have
\begin{equation}
\begin{aligned}
\left\langle v_1\left| \overset{\curvearrowright}{\prod_{i=0}^{r-1}}\overset{\curvearrowright}{\prod_{a=1}^{k_i}} E_{i}(z_{i,a})\right|v_2\right\rangle
&=\frac{f\left( \left\{ z_{i,a} \right\}_{i\in I}^{1\le a\le k_i} \right)
\displaystyle\prod_{i\in I}\,\prod_{1\le a<b\le k_i}\left( z_{i,a}-z_{i,b} \right)\left( z_{i,a}-q^{-1}t^{-1}z_{i,b} \right)}
{\displaystyle
\prod_{\substack{1\le a\le k_0\\1\le b\le k_{r-1}}}\left( z_{0,a}-t^{-1}z_{r-1,b} \right)
\prod_{i\in I\backslash\{r-1\}}\prod_{\substack{1\le a \le k_i\\1\le b\le k_{i+1}}}\left(z_{i,a}-q^{-1}z_{i+1,b}\right)}
\end{aligned}
\label{EMatrix}
\end{equation}
for some Laurent polynomial $f$, where the rational functions are expanded into Laurent series assuming
\begin{equation}
|z_{i,a}|=1,\, |q|>1,\, |t|>1.
\label{ERegion}
\end{equation}
On the other hand,
\begin{equation}
\begin{aligned}
\left\langle v_1\left| \overset{\curvearrowleft}{\prod_{i=0}^{r-1}}\overset{\curvearrowleft}{\prod_{a=1}^{k_i}} F_{i}(z_{i,a})\right|v_2\right\rangle
&=\frac{g\left( \left\{ z_{i,a} \right\}_{i\in I}^{1\le a\le k_i} \right)
\displaystyle\prod_{i\in I}\,\prod_{1\le a<b\le k_i}\left( z_{i,b}-z_{i,a} \right)\left( z_{i,b}-qtz_{i,a} \right)}
{\displaystyle
\prod_{\substack{1\le a\le k_{r-1}\\1\le b\le k_{0}}}\left( z_{r-1,b}-tz_{0,a} \right)
\prod_{i\in I\backslash\{0\}}\prod_{\substack{1\le b \le k_i\\1\le a\le k_{i-1}}}\left(z_{i,b}-qz_{i-1,a}\right)}
\end{aligned}
\label{FMatrix}
\end{equation}
for some Laurent polynomial $g$, where the rational functions are now expanded into Laurent series assuming
\begin{equation}
|z_{i,a}|=1,\, |q|<1,\, |t|<1.
\label{FRegion}
\end{equation}
\end{prop}
Notice that $\omega_{i,i+1}(z_{i,a},z_{i+1,b})^{-1}$ and $\omega_{i,i-1}(z_{i,a},z_{i-1,b})^{-1}$ are rational functions that we can also expand according to (\ref{ERegion}) and (\ref{FRegion}).
Thus, we can make sense of matrix elements of products of currents multiplied by these inverted mixing terms.
We do not claim that such products yield well-defined series of operators---just that their matrix elements make sense.
The following is a consequence of the toroidal relations:
\begin{prop}\label{DenomRel}
When computing matrix elements, we have the relations
\begin{align}
\label{EMatRel1}
\frac{E_{i}(z)E_i(w)}{\omega_{i,i}(z,w)}&=\frac{E_{i}(w)E_i(z)}{\omega_{i,i}(w,z)}\\
\label{EMatRel2}
\frac{E_{i}(z)E_{i+1}(w)}{\omega_{i,i+1}(z,w)}&= \frac{E_{i+1}(w)E_{i}(z)}{\omega_{i+1,i}(w,z)}\\
\nonumber
\frac{F_{i}(z)F_i(w)}{\omega_{i,i}(w,z)}&=\frac{E_{i}(w)E_i(z)}{\omega_{i,i}(z,w)}\\
\nonumber
\frac{F_{i}(z)F_{i+1}(w)}{\omega_{i+1,i}(w,z)}&= \frac{F_{i+1}(w)F_{i}(z)}{\omega_{i,i+1}(z,w)}.
\end{align}
\end{prop}

\begin{proof}
   We will only prove the statements for $E_i(z)$.
   Applying $\rho_{\vec{c}}^+$ to the relations from \ref{TorDef} yields
   \[
   E_i(z)E_j(w)=g_{i,j}(\ddd^{m_{i,j}}z/w)E_j(w)E_i(z)
   \]
   Strictly speaking, when unpacking this relation, we should clear denominators.
   We then obtain
   \begin{align}
    \label{Eii}
       \left(\frac{z}{w}-\qqq^2\right)E_i(z)E_i(w)&=\left(\qqq^2\frac{z}{w}-1\right) E_i(w)E_i(z)\\
    \label{Eii+1}
       \left(\ddd^{-1}\frac{z}{w}-\qqq^{-1}\right)E_{i}(z)E_{i+1}(w)&=\left(\qqq^{-1}\ddd^{-1}\frac{z}{w}-1\right)E_{i+1}(w)E_i(z)
   \end{align}
   Since $\omega_{i,i}(z,w)^{-1}=(z-\qqq^2w)(z-w)$, (\ref{Eii}) directly yields (\ref{EMatRel1}).
   On the other hand, multiplying both sides of (\ref{Eii+1}) by $-\qqq w$ gives us
   \[
   \omega_{i+1,i}(w,z)E_i(z)E_{i+1}(w)=\omega_{i,i+1}(z,w)E_{i+1}(w)E_i(z)
   \]
   This implies (\ref{EMatRel2}).
\end{proof}

\subsubsection{Constant term formula}
For $F\in\mathcal{S}^+_{k_\bullet}$ and $G\in\mathcal{S}^-_{k_\bullet}$, consider the rational functions:
\begin{align*}
&\frac{\displaystyle F\times\left\langle v_1\left| \overset{\curvearrowright}{\prod_{i=0}^{r-1}}\overset{\curvearrowright}{\prod_{a=1}^{k_i}} E_{i}(z_{i,a})\right|v_2\right\rangle}
{\displaystyle\left(\prod_{i\in I}\prod_{1\le a<a'\le k_i}\omega_{i,i}(z_{i,a},z_{i,a'})\right)
\left(\prod_{0\le i< j\le r-1}\,\prod_{\substack{1\le a\le k_i\\1\le b\le k_j}}\omega_{i,j}(z_{i,a},z_{j,b})\right)}\\
&\frac{\displaystyle G\times\left\langle v_1\left| \overset{\curvearrowleft}{\prod_{i=0}^{r-1}}\overset{\curvearrowleft}{\prod_{a=1}^{k_i}} F_{i}(z_{i,a})\right|v_2\right\rangle}
{\displaystyle\left(\prod_{i\in I}\prod_{1\le a<a'\le k_i}\omega_{i,i}(z_{i,a},z_{i,a'})\right)
\left(\prod_{0\le i< j\le r-1}\,\prod_{\substack{1\le a\le k_i\\1\le b\le k_j}}\omega_{i,j}(z_{i,a},z_{j,b})\right)}.
\end{align*}
We can expand these rational functions into Laurent series according to the assumptions (\ref{ERegion}) and (\ref{FRegion}), respectively.
For any Laurent series, we denote by $\{-\}_0$ this operation of taking constant terms.
\begin{lem}\label{ConstTermLem}
For $F\in\mathcal{S}^+_{k_\bullet}$ and $G\in\mathcal{S}^-_{k_\bullet}$, we have
\begin{equation}
\rho^+_{\vec{c}}(\Psi_+(F))
=\frac{1}{k_\bullet!}\left\{\frac{\displaystyle F\times\overset{\curvearrowright}{\prod_{i=0}^{r-1}}\,\overset{\curvearrowright}{\prod_{a=1}^{k_i}}E_i(z_{i,a})}
{\displaystyle\left(\prod_{i\in I}\prod_{1\le a<a'\le k_i}\omega_{i,i}(z_{i,a},z_{i,a'})\right)
\left(\prod_{0\le i< j\le r-1}\,\prod_{\substack{1\le a\le k_i\\1\le b\le k_j}}\omega_{i,j}(z_{i,a},z_{j,b})\right)}\right\}_0
\label{ConstantTermE}
\end{equation}
where the right-hand side is expanded according to (\ref{ERegion}) and
\begin{equation}
\rho^-_{\vec{c}}(\Psi_-(G))
= \frac{1}{k_\bullet!}\left\{\frac{\displaystyle G\times\overset{\curvearrowleft}{\prod_{i=0}^{r-1}}\,\overset{\curvearrowleft}{\prod_{a=1}^{k_i}}F_i(z_{i,a})}
{\displaystyle\left(\prod_{i\in I}\prod_{1\le a<a'\le k_i}\omega_{i,i}(z_{i,a},z_{i,a'})\right)
\left(\prod_{0\le i< j\le r-1}\,\prod_{\substack{1\le a\le k_i\\1\le b\le k_j}}\omega_{i,j}(z_{i,a},z_{j,b})\right)}\right\}_0
\label{ConstantTermF}
\end{equation}
where the right-hand side is expanded according to (\ref{FRegion}).
In particular, the expressions on the right-hand side are well-defined operators on $W$.
\end{lem}

\begin{proof}
A consequence of Theorem \ref{ShuffTor} and the toroidal relations is that $\mathcal{S}^\pm$ are both spanned by shuffle monomials
\[
z_{0,1}^{n(0,1)}\star z_{0,1}^{n(0,2)}\star\cdots\star z_{0,1}^{n(0,k_0)}\star z_{1,1}^{n(1,1)}\star\cdots\star z_{r-1,1}^{n(r-1,k_{r-1})}
\]
since
\begin{align*}
\Psi_+\left( z_{0,1}^{n(0,1)}\star\cdots\star z_{r-1,1}^{n(r-1,k_{r-1})} \right)&= e_{0,n(0,1)}\cdots e_{r-1,n(r-1,k_{r-1})}\\
\Psi_-\left(z_{0,1}^{n(0,1)}\star\cdots\star z_{r-1,1}^{n(r-1,k_{r-1})} \right)&= f_{0,n(0,1)}\cdots f_{r-1,n(r-1,k_{r-1})}.
\end{align*}
We will check that the matrix elements coincide for these monomials, from which the lemma follows.
For the `$+$' case, the proposed formula gives us
\begin{align*}
&\frac{1}{k_\bullet!}\left\{
\displaystyle\mathrm{Sym}_{k_\bullet}\left(\left\{\prod_{i\in I}\prod_{a=1}^{k_i}z_{i,a}^{n(i,a)}
\,\prod_{1\le a<a'\le k_i}\omega_{i,i}(z_{i,a},z_{i,a'})\right\}
\left\{ \prod_{0\le i<j\le r-1}\,\prod_{\substack{1\le a\le k_i\\1\le b\le k_j}}\omega_{i,j}(z_{i,a},z_{j,b}) \right\}\right)\right.\\
&\times\left.\frac{\displaystyle\left\langle v_1\left| \overset{\curvearrowright}{\prod_{i=0}^{r-1}}\overset{\curvearrowright}{\prod_{a=1}^{k_i}} E_{i}(z_{i,a})\right|v_2\right\rangle}
{\displaystyle\left(\prod_{i\in I}\prod_{1\le a<a'\le k_i}\omega_{i,i}(z_{i,a},z_{i,a'})\right)
\left(\prod_{0\le i< j\le r-1}\,\prod_{\substack{1\le a\le k_i\\1\le b\le k_j}}\omega_{i,j}(z_{i,a},z_{j,b})\right)}\right\}_0.
\end{align*}
Using (\ref{EMatRel1}) to swap variables, we can move both the matrix element and the mixing terms inside the symmetrization, where the mixing terms will all cancel out.
Notice that taking the constant term is insensitive to the labeling of the variables, and thus the constant terms of all the summands of the symmetrization are equal.
The end result is
\[
\left\{\prod_{i\in I}\prod_{a=1}^{k_i}z_{i,a}^{n(i,a)}
\left\langle v_1\left| \overset{\curvearrowright}{\prod_{i=0}^{r-1}}\overset{\curvearrowright}{\prod_{a=1}^{k_i}} E_{i}(z_{i,a})\right|v_2\right\rangle\right\}_0
=\left\langle v_1\left|\rho_{\vec{c}}^{+}\left(e_{0,n(0,1)}\cdots e_{r-1,n(r-1,k_{r-1})}  \right)\right|v_2\right\rangle.
\]
The `$-$' case is similar.
\end{proof}

\section{Difference operators}\label{Difference}

\subsection{Setup} Now, we will fix $\alpha\in Q$, which also fixes a core.
The previous two sections were concerned with symmetric functions in infinitely many variables.
Here, we will shift to working with finitely many variables 
\[
\left\{x_{l}^{(i)}\right\}_{i\in I}^{1\le l\le N_i}=X_{N_\bullet}.
\]
We will impose the compatibility \eqref{Compatibility} between $\alpha$ and the vector $N_\bullet$ recording the number of variables of each color.
Our approach for finding difference operators is straightforward: we use Lemma \ref{ConstTermLem} to compute the action of $\rho^\pm_{\vec{c}}(\Psi_\pm(\mathcal{E}_{p,n}^\pm))$ on a function $f\left[ X_{N_\bullet} \right]$.
We assume that $n\le N_i$ for all $i\in I$.

\subsubsection{Finitized vertex operators}
Recall that $\Lambda^I_{N_\bullet}$ denotes the tensor product over $i\in I$ of rings of symmetric polynomials in $N_i$ variables and $\pi_{N_\bullet}: \Lambda^I\to \Lambda^I_{N_\bullet}$ is the natural projection. 
We will abuse notation and also denote the map $(\pi_{N_\bullet}\otimes 1):\Lambda^I\otimes\K\{Q\}\rightarrow\Lambda^I_{N_\bullet}\otimes \K\{Q\}$ by $\pi_{N_\bullet}$.
Recall Proposition \ref{NormalOrder}.

\begin{rem}
    The action (\ref{ZHi}) of the operator $z^{H_i,0}$ includes a power of $\ddd$.
    In Proposition \ref{NormalOrder} and throughout this paper, we will be working with products of currents that have an equal number of $E_i(z)$ for each $i\in I$ and likewise for $F_i(z)$.
    In this setup, the powers of $\ddd$ will cancel.
    Namely, because $m_{i,i\pm 1}=-m_{i\pm 2,i\pm 1}$, we have that the power of $\ddd$ from the action of $z_{i,a}^{H_{i,0}}$ will be canceled by those from the action of $z_{i\pm 2,a}^{H_{i\pm 2,0}}$.
    Thus, we we will abuse notation and omit the $\ddd$ from the action of $z^{H_{i,0}}$.
    Applying the compatibility condition \eqref{Compatibility}, this leaves
    \[
    z^{H_{i,0}}(e^\alpha)=z^{( \alpha_i^\vee,\alpha)}=z^{N_{i}-N_{i-1}}.
    \]
    \end{rem}
\begin{prop}\label{FiniteVert}
    Let $f\in\Lambda^I$ be factored according to color:
    \[
    f= \prod_{i\in I} f_i[X^{(i)}],
    \]
    where $f_i\in\Lambda$ for all $i\in I$.
    For
    \begin{equation}
    |z|=1,\, |q|>1,\, |t|>1,\, |x_{l}^{(j)}|<1,
    \label{EExpand0}
    \end{equation}
    the vertex operators from (\ref{ENormalOrder}) act on $f$ such that upon finitization, we have:
\begin{equation}
\label{FinEVertex}
    \begin{aligned}
    &\pi_{N_\bullet} \left(\exp\left[ \sum_{k>0}\left( p_k[X^{(i)}]-t^{-k}p_k[X^{(i-1)}] \right) \frac{z^k}{k}\right]\right.\\
    &\times\left.\exp\left[\sum_{k>0}\left( -p_k[X^{(i)}]^\perp+q^{-k}p_k[X^{(i-1)}]^\perp \right)\frac{z^{-k}}{k}  \right]z^{H_{i,0}} (f\otimes e^\alpha)\right) \\
    &=
    \frac{\displaystyle\prod_{l=1}^{N_{i-1}}\left(z^{-1}-t^{-1}x_{l}^{(i-1)}\right)}{\displaystyle\prod_{l=1}^{N_i} \left(z^{-1}-x_{l}^{(i)}\right)}
    \left( f_i\left[x_{\bullet}^{(i)}-z^{-1}\right]f_{i-1}\left[x_{\bullet}^{(i-1)}+q^{-1}z^{-1}\right]
    \prod_{\substack{j\in I\\j\not=i,i-1}} f_j\left[x_{\bullet}^{(j)}\right]\otimes e^\alpha\right).
    \end{aligned}
\end{equation}
    On the other hand, for
    \begin{equation*}
    |z|=1,\, |q|<1,\, |t|<1,\, |x_{l}^{(j)}|<1,
    \label{FExpand0}
    \end{equation*}
    the vertex operators from (\ref{FNormalOrder}) act as:
\begin{equation}
\label{FinFVertex}
    \begin{aligned}
    &\pi_{N_\bullet} \left(\exp\left[ \sum_{k>0}\left( -t^kp_k[X^{(i)}]+p_k[X^{(i-1)}] \right) \frac{z^k}{k}\right]\right.\\
    &\times \left.\exp\left[\sum_{k>0}\left( q^kp_k[X^{(i)}]^\perp-p_k[X^{(i-1)}]^\perp \right)\frac{z^{-k}}{k}  \right]z^{-H_{i,0}} (f\otimes e^\alpha)\right) \\
    &=
    \frac{\displaystyle\prod_{l=1}^{N_{i}}\left(z^{-1}-tx_{l}^{(i)}\right)}
    {\displaystyle\prod_{l=1}^{N_{i-1}} \left(z^{-1}-x_{l}^{(i-1)}\right)}
    \left( f_i\left[x_{\bullet}^{(i)}+qz^{-1}\right]f_{i-1}\left[x_{\bullet}^{(i-1)}-z^{-1}\right]
    \prod_{\substack{j\in I\\j\not=i,i-1}} f_j\left[x_{\bullet}^{(j)}\right]\otimes e^\alpha\right).
    \end{aligned}
\end{equation}
\end{prop}
\begin{proof}
We will only consider (\ref{FinEVertex})---the proof for (\ref{FinFVertex}) is similar.
First consider the `left' half of the vertex operator together with $z^{H_i,0}$. 
We have
\begin{align}
\nonumber
&\pi_{N_\bullet} \left(\exp\left[ \sum_{k>0}\left( p_k[X^{(i)}]-t^{-k}p_k[X^{(i-1)}] \right) \frac{z^k}{k}\right]z^{H_{i,0}} (f\otimes e^\alpha)\right) \\
\label{FinCreation}
&=
\frac{\displaystyle\prod_{l=1}^{N_{i-1}}\left(1-t^{-1}zx_{l}^{(i-1)}\right)}{\displaystyle\prod_{l=1}^{N_i} \left(1-zx_{l}^{(i)}\right)}
z^{N_i-N_{i-1}}
\left( \pi_{N_\bullet}(f)\otimes e^\alpha\right)
=
\frac{\displaystyle\prod_{l=1}^{N_{i-1}}\left(z^{-1}-t^{-1}x_{l}^{(i-1)}\right)}{\displaystyle\prod_{l=1}^{N_i} \left(z^{-1}-x_{l}^{(i)}\right)}
\left( \pi_{N_\bullet}(f)\otimes e^\alpha\right)
\end{align}
For this to hold, we will need to impose conditions on $|x_{l}^{(i)}|$.
Recall that we have the conditions (\ref{NormalEExpand}) when working with $\{E_i(z)\}$.
We extend these to (\ref{EExpand0}) for (\ref{FinCreation}) to hold. 
Let us also point out that the compatibility condition \eqref{Compatibility} is used to obtain the factor $z^{N_i-N_{i-1}}$ after the first equality.

Next, from the `right' half, we have
\begin{align}
\nonumber
&\pi_{N_\bullet}\left( \exp\left[\sum_{k>0}\left( -p_k[X^{(i)}]^\perp+q^{-k}p_k[X^{(i-1)}]^\perp \right)\frac{z^{-k}}{k}  \right]\cdot f \right)\\
\label{FinAnnihilation}
&=
\pi_{N_\bullet}\left(f_i[X^{(i)}-z^{-1}]f_{i-1}[X^{(i-1)}+q^{-1}z^{-1}]\prod_{\substack{j\in I\\j\not=i,i-1}} f_j[X^{(j)}] \right)\\
\nonumber
&= f_i\left[x_{\bullet}^{(i)}-z^{-1}\right]f_{i-1}\left[x_{\bullet}^{(i-1)}+q^{-1}z^{-1}\right]
\prod_{\substack{j\in I\\j\not=i,i-1}} f_j\left[x_{\bullet}^{(j)}\right].\notag
\end{align}
Here, (\ref{FinAnnihilation}) follows from checking on power sums $p_k[X^{(i)}]$ and $p_k[X^{(i-1)}]$.
\end{proof}

\begin{rem}
    As in Proposition \ref{FiniteVert}, when writing formulas involving vertex operators, we will express them in terms of
    functions that are factorizable according to color:
    \[
    f=\prod_{i\in I}f_i[X^{(i)}].
    \]
    Factorizable functions span $\Lambda^I$, so to define an operator, it suffices to consider its action on such functions.
    We can write our operators in terms of general $f$ if we introduce \textit{colored plethystic notation}.
    For instance, the terms at the bottom of (\ref{FinEVertex}) can be written as
    \begin{align*}
        &f_i\left[x_{\bullet}^{(i)}-z^{-1}\right]f_{i-1}\left[x_{\bullet}^{(i-1)}+q^{-1}z^{-1}\right]
    \prod_{\substack{j\in I\\j\not=i,i-1}} f_j\left[x_{\bullet}^{(j)}\right]\\
        &=f\left[\left(x_\bullet^{(i)}-z^{-1}\right)+\left(x_\bullet^{(i-1)}+q^{-1}z^{-1}\right)+\sum_{\substack{j\in I\\j\not=i,i-1}}x_\bullet^{(j)}\right],
    \end{align*}
    where the bottom denotes the image of $f[X_{N_\bullet}]$ under the ring map generated by
    \[
    p_n[x_\bullet^{(j)}]\mapsto
    \begin{cases}
        p_n[x_\bullet^{(i)}]-z^{-n} & j=i\\
        p_n[x_\bullet^{(i-1)}]+q^{-n}z^{-n} & j=i-1\\
        p_n[x_\bullet^{(j)}] & \hbox{otherwise}
    \end{cases}.
    \]
    This notation can then be carried over to general $f$.
    However, the benefits of introducing this notation in our paper seemed marginal at best, so we have elected to making statements in terms of factorizable functions.
\end{rem}

\subsubsection{Applying the constant term formula}\label{ConstEF}

Our next goal is to obtain constant term formulas for the action of the shuffle elements $\mathcal{E}_{p,n}^\pm$ from (\ref{Epn}). In light of Lemma \ref{EpnShuff}, we will also incorporate the constants
\begin{align*}
c_{n}^+&:=\frac{(-1)^{n}\qqq^{n(r-1)}t^{-n}(1-q^{-1}t^{-1})^{nr}}{v^{-n}\prod_{a=1}^n(1-q^{-a}t^{-a})},\quad c_{n}^-:=\frac{(-1)^{nr-n}\ddd^{-n(r-1)}t^n(1-qt)^{nr}}{v^{n}q^n\prod_{a=1}^n(1-q^{-a}t^{-a})}
\end{align*}
where $v=(-1)^{\frac{(r-2)(r-3)}{2}}\qqq\ddd^{-\frac{r}{2}}(c_0\cdots c_{r-1})^{-1}$.

\begin{lem}\label{LemEFConstTerm}
For any factorizable $f=\prod_{i\in I}f_i[X^{(i)}]\in\Lambda^I$, we have
\begin{equation}
\begin{aligned}
&c_n^+
\pi_{N_\bullet}\biggl( (\rho_{\vec{c}}^+\circ\Psi_+)(\mathcal{E}_{p,n}^+)(f\otimes e^\alpha)\biggr)\\
&=\frac{(1-q^{-1}t^{-1})^{nr}}{t^{n}\prod_{a=1}^n(1-q^{-a}t^{-a})}
\left\{\prod_{i\in I}\prod_{a=1}^n\prod_{l=1}^{N_i}\left(\frac{ z_{i+1,a}^{-1}-t^{-1}x_{l}^{(i)} }{ z_{i,a}^{-1}-x_{l}^{(i)} }\right)\right.\\
&\times\prod_{1\le a<b\le n}
\left[
\frac{\left(1-z_{p+1,b}/z_{p+1,a}\right)\left(1-q^{-1}t^{-1}z_{p+1,b}/z_{p+1,a}\right)}{\left(1-tz_{p,b}/z_{p+1,a}\right)\left(1-t^{-1}z_{p+2,b}/z_{p+1,a}\right)}\right.\\
&\times\left.\prod_{i\in I\backslash\{p+1\}}
\frac{\left(1-z_{i,b}/z_{i,a}\right)\left(1-q^{-1}t^{-1}z_{i,b}/z_{i,a}\right)}{\left(1-q^{-1}z_{i-1,b}/z_{i,a}\right)\left(1-t^{-1}z_{i+1,b}/z_{i,a}\right)}\right]\\
&\times\prod_{a=1}^n \left[
\left(\frac{z_{0,a}}{z_{p,a}}\right)\left(\frac{1}{1-t^{-1}z_{p+1,a}/z_{p,a}}\right)\right.\\
&\times\left.
\prod_{i\in I\setminus\{p+1\}}
\left(\frac{1}{1-t^{-1}z_{i,a}/z_{i-1,a}}\right)
\left( \frac{1}{1-q^{-1}z_{i-1,a}/z_{i,a}} \right)\right]\\
&\left.\times
\prod_{i\in I}
f_i\left[ \sum_{l=1}^{N_i} x_{l}^{(i)}-\sum_{a=1}^nz_{i,a}^{-1}+q^{-1}\sum_{a=1}^n z_{i+1,a}^{-1} \right]
\right\}_0\otimes e^\alpha
\end{aligned}
\label{EConstTerm}
\end{equation}
and
\begin{equation}
\label{FConstTerm}
\begin{aligned}
&c_n^-\pi_{N_\bullet}\biggl((\rho_{\vec{c}}^-\circ\Psi_-)(\mathcal{E}_{p,n}^-)(f\otimes e^\alpha)\biggr)\\
&=\frac{t^n(1-qt)^{nr}}{\prod_{k=1}^n(1-q^kt^k)} 
\left\{\prod_{i\in I}\prod_{a=1}^n\prod_{l=1}^{N_i}\left(\frac{  z_{i,a}^{-1}-tx_{l}^{(i)} }{ z_{i+1,a}^{-1}-x_{l}^{(i)} }\right)\right.\\
&\times\prod_{1\le a<b\le n}\left[
\frac{\left(1-z_{p,a}/z_{p,b}\right)\left(1-qtz_{p,a}/z_{p,b}\right)}{\left(1-t^{-1}z_{p+1,a}/z_{p,b}\right)\left(1-tz_{p-1,a}/z_{p,b}\right)}\right.\\
&\times\left.
\prod_{i\in I\backslash\{p\}}
\frac{\left(1-z_{i,a}/z_{i,b}\right)\left(1-qtz_{i,a}/z_{i,b}\right)}{\left(1-qz_{i+1,a}/z_{i,b}\right)\left(1-tz_{i-1,a}/z_{i,b}\right)}\right]\\
&\times
\prod_{a=1}^n\left[
\left(\frac{z_{p,a}}{z_{0,a}}\right)\left(\frac{1}{1-tz_{p,a}/z_{p+1,a}}\right)\right.\\
&\times\left.
\prod_{i\in I\setminus\{p\}}
\left(\frac{1}{1-tz_{i,a}/z_{i+1,a}}\right)
\left( \frac{1}{1-qz_{i+1,a}/z_{i,a}} \right)
\right]\\
&\left.\times
\prod_{i\in I}
f_i\left[ \sum_{l=1}^{N_i}  x_{l}^{(i)}+q\sum_{a=1}^n z_{i,a}^{-1}-\sum_{a=1}^n z_{i+1,a}^{-1}  \right]\right\}_0\otimes e^\alpha.
\end{aligned}
\end{equation}
\end{lem}

\begin{proof}
Plugging in $\mathcal{E}_{p,n}^\pm$ into the formula from Lemma \ref{ConstTermLem}, we can use the toroidal relations and Proposition \ref{DenomRel} to reorder the currents in alignment with Proposition \ref{NormalOrder}.
As in the proof of Lemma \ref{ConstTermLem}, we can use the toroidal relations to remove the symmetrizations in $\mathcal{E}_{p,n}^\pm$.
Taking the result for $\mathcal{E}_{p,n}^+$, acting on $f\otimes e^\alpha$, and then applying $\pi_{N_\bullet}$ gives us:
\begin{align*}
\begin{split}
&\pi_{N_\bullet}\biggl( (\rho_{\vec{c}}^+\circ\Psi_+)(\mathcal{E}_{p,n}^+)(f\otimes e^\alpha)\biggr)\\
&=\left( (-1)^{\frac{(r-2)(r-3)}{2}}\ddd^{\frac{r}{2}-1}\prod_{i\in I} c_i\right)^n
\left\{\prod_{i\in I}\prod_{a=1}^n\prod_{l=1}^{N_i}\left(\frac{ z_{i+1,a}^{-1}-t^{-1}x_{l}^{(i)} }{ z_{i,a}^{-1}-x_{l}^{(i)} }\right)\right.\\
&\times\prod_{1\le a<b\le n}
\left[ 
\frac{1-q^{-1}z_{p,b}/z_{p+1,a}}{1-tz_{p,b}/z_{p+1,a}} \prod_{i\in I}
\frac{\left(1-z_{i,b}/z_{i,a}\right)\left(1-q^{-1}t^{-1}z_{i,b}/z_{i,a}\right)}{\left(1-t^{-1}z_{i+1,b}/z_{i,a}\right)\left(1-q^{-1}z_{i-1,b}/z_{i,a}\right)}
\right]\\
&\times\prod_{a=1}^n
\left[\left(\frac{z_{0,a}}{z_{p+1,a}}\right)\left(\frac{z_{p+1,a}}{\omega_{p+1,p}(z_{p+1,a},z_{p,a})}\right)\right.\\
&\times\left.\prod_{i\in I\setminus\{p+1\}}
\frac{z_{i,a}}{\left(1-t^{-1}z_{i,a}/z_{i-1,a}\right)\omega_{i-1,i}(z_{i-1,a},z_{i,a})}\right]\\
&\times \prod_{i\in I}f_i\left[  \sum_{l=1}^{N_i}x_{l}^{(i)}-\sum_{a=1}^nz_{i,a}^{-1}+q^{-1}\sum_{a=1}^n z_{i+1,a}^{-1} \right]\Bigg\}_0\otimes e^\alpha
\end{split}
\end{align*}
where all rational functions are expanded as Laurent series assuming
\begin{equation}
|z_{i,a}|=1,\,
|x_{l}^{(j)}|<1,\,
|t|>1,\,
|q|>1.   
\label{EExpand}
\end{equation}
For $\mathcal{E}_{p,n}^-$, we instead have
\begin{align*}
\begin{split}
&\pi_{N_\bullet}\biggl((\rho_{\vec{c}}^-\circ\Psi_-)(\mathcal{E}_{p,n}^-) (f\otimes e^\alpha)\biggr)\\
&=
\left( \frac{(-1)^{\frac{(r-2)(r-3)}{2}}}{\displaystyle\ddd^{\frac{r}{2}-1} \prod_{i\in I} c_i}\right)^n
\left\{\prod_{i\in I}\prod_{a=1}^n\prod_{l=1}^{N_i}\left(\frac{  z_{i,a}^{-1}-tx_{l}^{(i)} }{ z_{i+1,a}^{-1}-x_{l}^{(i)} }\right)\right.\\
&\times\prod_{1\le a<b\le n}
\left[
\frac{q^{-1}-z_{p+1,a}/z_{p,b}}{t-z_{p+1,a}/z_{p,b}} 
\prod_{i\in I}
\frac{\left(1-z_{i,a}/z_{i,b}\right)\left(1-qtz_{i,a}/z_{i,b}\right)}{\left(1-tz_{i-1,a}/z_{i,b}\right)\left(1-qz_{i+1,a}/z_{i,b}\right)}
\right]\\
&\times \prod_{a=1}^n
\left[
\left(\frac{z_{p+1,a}}{z_{0,a}}\right)
\left(\frac{-z_{p,a}}
{\omega_{p+1,p}(z_{p+1,a},z_{p,a})}\right)\right.\\
&\left.\times\prod_{i\in I\setminus\{p\}}
 \frac{z_{i,a}}{\left(1-tz_{i,a}/z_{i+1,a}\right)\omega_{i,i+1}(z_{i,a},z_{i+1,a})} \right]\\
&\left.\times 
\prod_{i\in I}f_i\left[ \sum_{l=1}^{N_i}  x_{l}^{(i)}+q\sum_{a=1}^n z_{i,a}^{-1}-\sum_{a=1}^n z_{i+1,a}^{-1}  \right]\right\}_0\otimes e^\alpha
\end{split}
\end{align*}
where all rational functions are expanded into Laurent series assuming.
\begin{equation}
|z_{i,a}|=1,\,
|x_{l}^{(j)}|<1,\,
|q|<1,\,
|t|<1.   
\label{FExpand}
\end{equation}
In both formulas, we are taking constant terms in the $z$-variables.

Finally, to obtain \eqref{EConstTerm} and \eqref{FConstTerm} from these formulas, we multiply through by $c_n^\pm$ and use
$q=\qqq\ddd,\, t=\qqq\ddd^{-1}$ to write
\begin{align*}
\omega_{i,i+1}(z,w)&= \qqq w-\ddd^{-1}z=\ddd^{-1}(qw-z)=\qqq\left(w-q^{-1}z\right)\\
\omega_{i,i-1}(z,w)&= z-\qqq\ddd^{-1}w=z-tw.\qedhere
\end{align*}
\end{proof}

\begin{rem}
    Observe that the formulas in Lemma \ref{LemEFConstTerm} are for symmetric functions in finitely many variables.
    To obtain constant term formulas for operators in infinitely many variables, we can apply Proposition \ref{FiniteVert}.
    For example, starting from (\ref{EConstTerm}), we use (\ref{FinEVertex}) and replace
    \[
    \prod_{i\in I}\prod_{a=1}^n\prod_{l=1}^{N_i}\left(\frac{ z_{i+1,a}^{-1}-t^{-1}x_{l}^{(i)} }{ z_{i,a}^{-1}-x_{l}^{(i)} }\right)
    f_i\left[ \sum_{l=1}^{N_i} x_{l}^{(i)}-\sum_{a=1}^nz_{i,a}^{-1}+q^{-1}\sum_{a=1}^n z_{i+1,a}^{-1} \right]
    \otimes e^\alpha
    \]
    with
    \begin{align*}
    &\exp\left[ \sum_{i\in I}\sum_{a=1}^n\left(\sum_{k>0}\left( p_k[X^{(i)}]-t^{-k}p_k[X^{(i-1)}] \right) \frac{z_{i,a}^k}{k}\right)\right]\\
    &\times\exp\left[\sum_{i\in I}\sum_{a=1}^n\sum_{k>0}\left( -p_k[X^{(i)}]^\perp+q^{-k}p_k[X^{(i-1)}]^\perp \right)\frac{z_{i,a}^{-k}}{k}  \right]z^{H_{i,0}} (f\otimes e^\alpha).   
    \end{align*}
\end{rem}


\subsubsection{Integral formula}
Regardless of $f$, the formulas obtained in \ref{ConstEF} are constant terms of Laurent series expansions of some rational function.
Note that all poles are simple except for the poles at zero possibly coming from the plethystic modifications done to $f$.
Thus, it will be advantageous to invert all the $z$-variables: let
\[
w_{i,a}:=z_{i,a}^{-1}
\]
and define the functions
\begin{align}
\label{En1}
g_{p,n}^+(w_{\bullet,\bullet},X_{N_\bullet})
&:=\frac{(-1)^{\frac{n(n-1)}{2}}(1-q^{-1}t^{-1})^{nr}}{t^{\frac{n(n+1)}{2}}\prod_{a=1}^n(1-q^{-a}t^{-a})}
\prod_{i\in I}\prod_{a=1}^n\prod_{l=1}^{N_i}\left(\frac{ w_{i+1,a}-t^{-1}x_{l}^{(i)} }{ w_{i,a}-x_{l}^{(i)} }\right)\\
\label{En2}
&\times\prod_{1\le a<b\le n}
\left[
\frac{\left(w_{p+1,b}-w_{p+1,a}\right)\left(w_{p+1,b}-q^{-1}t^{-1}w_{p+1,a}\right)}
{\left(w_{p+1,a}-t^{-1}w_{p,b}\right)\left(w_{p+2,b}-t^{-1}w_{p+1,a}\right)}\right.\\
\label{En3}
&\times\left.
\prod_{i\in I\backslash\{p+1\}}
\frac{\left(w_{i,b}-w_{i,a}\right)\left(w_{i,b}-q^{-1}t^{-1}w_{i,a}\right)}{\left(w_{i-1,b}-q^{-1}w_{i,a}\right)\left(w_{i+1,b}-t^{-1}w_{i,a}\right)}\right]\\
\label{En4}
&\times\prod_{a=1}^n \left[
\left(\frac{w_{p,a}}{w_{0,a}}\right)\left(\frac{w_{p+1,a}}{w_{p+1,a}-t^{-1}w_{p,a}}\right)\right.\\
\label{En5}
&\times\left.
\prod_{i\in I\setminus\{p+1\}}
\left(\frac{w_{i,a}}{w_{i,a}-t^{-1}w_{i-1,a}}\right)
\left( \frac{w_{i-1,a}}{w_{i-1,a}-q^{-1}w_{i,a}} \right)\right].
\end{align}
and
\begin{align}
g_{p,n}^-(w_{\bullet,\bullet},X_{N_\bullet})
&:=
\label{Fn1}
\frac{(-1)^{\frac{n(n-1)}{2}}t^{\frac{n(n+1)}{2}}(1-qt)^{nr}}{\prod_{a=1}^n(1-q^at^a)} 
\prod_{i\in I}\prod_{a=1}^n\prod_{l=1}^{N_i}\left(\frac{  w_{i,a}-tx_{l}^{(i)} }{ w_{i+1,a}-x_{l}^{(i)} }\right)\\
\label{Fn2}
&\times\prod_{1\le a<b\le n}\left[
\frac{\left(w_{p,a}-w_{p,b}\right)\left(w_{p,a}-qtw_{p,b}\right)}{\left(w_{p,b}-tw_{p+1,a}\right)\left(w_{p-1,a}-tw_{p,b}\right)}\right.\\
\label{Fn3}
&\times\left.
\prod_{i\in I\backslash\{p\}}
\frac{\left(w_{i,a}-w_{i,b}\right)\left(w_{i,a}-qtw_{i,b}\right)}{\left(w_{i+1,a}-qw_{i,b}\right)\left(w_{i-1,a}-tw_{i,b}\right)}\right]\\
\label{Fn4}
&\times
\prod_{a=1}^n\left[
\left(\frac{w_{0,a}}{w_{p,a}-tw_{p+1,a}}\right)
\prod_{i\in I\setminus\{p\}}
\left(\frac{w_{i,a}}{w_{i,a}-tw_{i+1,a}}\right)
\left( \frac{w_{i+1,a}}{w_{i+1,a}-qw_{i,a}} \right)
\right].
\end{align}

\begin{lem}\label{IntForm}
Let $f=\prod_{i\in I} f_i[X^{(i)}]\in\Lambda^I$ be factorizable according to color.
For the `+' case, assume
\[
|x_{l}^{(i)}|<1,\, |q|>1,\, |t|>1.
\]
We have
\begin{align}
\nonumber
&c_n^+
\pi_{N_\bullet}\biggl( (\rho_{\vec{c}}^+\circ\Psi_+)(\mathcal{E}_{p,n}^+)(f\otimes e^\alpha)\biggr)\\
\label{EIntegral}
&=\left(\underset{|w_{i,a}|=1}{\oint\cdots\oint}g_{p,n}^+(w_{\bullet,\bullet},X_{N_\bullet})
\prod_{i\in I}
f_i\left[ \sum_{l=1}^{N_i} x_{l}^{(i)}-\sum_{a=1}^nw_{i,a}+\sum_{a=1}^{n}q^{-1}w_{i+1,a} \right]
\prod_{a=1}^n \frac{dw_{i,a}}{2\pi\sqrt{-1}w_{i,a}}
\right)\otimes e^\alpha
\end{align}
where we orient the unit circle $|w_{i,a}|=1$ counter-clockwise.
In the `$-$' case, we now assume
\[
|x_{l}^{(i)}|<1,\, |q|<1,\, |t|<1.
\]
We then have
\begin{align}
\nonumber
&c_n^-
\pi_{N_\bullet}\biggl( (\rho_{\vec{c}}^-\circ\Psi_-)(\mathcal{E}_{p,n}^-)(f\otimes e^\alpha)\biggr)\\
\label{FIntegral}
&=\left(\underset{|w_{i,a}|=1}{\oint\cdots\oint}g_{p,n}^-(w_{\bullet,\bullet},X_{N_\bullet})
\prod_{i\in I}
f_i\left[ \sum_{l=1}^{N_i} x_{l}^{(i)}+\sum_{a=1}^nqw_{i,a}-\sum_{a=1}^nw_{i+1,a} \right]
\prod_{a=1}^n \frac{dw_{i,a}}{2\pi\sqrt{-1}w_{i,a}}
\right)\otimes e^\alpha
\end{align}
and also orient the unit circle counter-clockwise.
\end{lem}

\begin{proof}
Upon making the substitution $w_{i,a}:=z_{i,a}^{-1}$, the right hand side of (\ref{EConstTerm}) is equal to
\begin{equation}
\left\{g_{p,n}^+(w_{\bullet,\bullet},X_{N_\bullet})\prod_{i\in I}f_i\left[\sum_{l=1}^{N_i}  x_{l}^{(i)}-\sum_{a=1}^nw_{i,a}+\sum_{a=1}^nq^{-1}w_{i+1,a} \right]\right\}_0\otimes e^\alpha
\label{ConstTermEW}
\end{equation}
Now, all the poles appearing in (\ref{ConstTermEW}) are simple.
Similarly, the right hand side of (\ref{FConstTerm}) becomes
\begin{equation}
\left\{g_{p,n}^-(w_{\bullet,\bullet},X_{N_\bullet})\prod_{i\in I}f_i\left[  \sum_{l=1}^{N_i}x_{l}^{(i)}+\sum_{a=1}^nqw_{i,a}- \sum_{a=1}^nw_{i+1,a} \right]\right\}_0\otimes e^\alpha
\label{ConstTermFW}
\end{equation}

In case `$\pm$', the integrands are given by series in the $x^{(i)}_l$ and $q^{\mp 1}$, $t^{\mp 1}$, with coefficients which are Laurent polynomials in the $w_{i,a}$. Under the given assumptions, these series converge uniformly absolutely on the integration cycle and thus we can exchange the order of summation and integration. This turns the integrals (\ref{EIntegral}) and (\ref{FIntegral}) into the constant term formulas \eqref{ConstTermEW} and \eqref{ConstTermFW}, respectively.
\end{proof}

\begin{rem}\label{CompRem}
Recall that the compatibility condition \eqref{Compatibility} between $N_\bullet$ and $\alpha$ was used to obtain the formulas in Proposition \ref{FiniteVert}.
At this stage, we note that without the compatibility, we would have to contend with an additional Laurent monomial factor in the variables $w_{i,a}$ in \eqref{ConstTermEW} and \eqref{ConstTermFW}. 
This would prevent us from obtaining a manageable formula due to the presence of non-simple poles at zero.
\end{rem}

\comment{
\begin{proof}
First notice that we can factor any variable $w_{i,a}$ out of $g_n^\pm(w_{\bullet,\bullet},X_{N_\bullet})$ without creating a new pole at $w_{i,a}=0$.
The proposed formulas follow from general considerations involving partial fraction decomposition.
Namely, let $F(x)$ be a rational function with simple finite poles $p_1,\ldots,p_m$ that do not lie on the unit circle $|x|=1$.
By partial fraction decomposition, we have
\[
xF(x)=\sum_{i=1}^m\frac{x\mathrm{Res}_{p_i}(F)}{x-p_i}
\]
where $\mathrm{Res}_{p_i}(F)$ is the residue of $F$ at $p_i$.
Now consider the effect of expanding into Laurent series expansion assuming $|x|=1$ and taking constant terms.
If $p_i$ is outside of the unit circle, then its summand becomes
\[
\frac{x\mathrm{Res}_{p_i}(F)}{x-p_i}=-\frac{x p_i^{-1}\mathrm{Res}_{p_i}(F)}{\displaystyle 1-x/p_i}
\]
which will have trivial constant term.
On the other hand, if $p_i$ is inside the unit circle, then its summand is
\[
\frac{x\mathrm{Res}_{p_i}(F)}{x-p_i}=\frac{\mathrm{Res}_{p_i}(F)}{\displaystyle 1-p_i/x}
\]
which will have constant term $\mathrm{Res}_{p_i}(F)$.
Thus, overall,
\[
\{xF(x)\}_0=\sum_{\substack{i=1\\|p_i|<1}}^m\mathrm{Res}_{p_i}(F)\qedhere
\]
\end{proof}
}

\subsubsection{Cyclic-shift operators}\label{Cyclic}
To describe the results of our computation, we need to introduce some difference operators that also permute variables.
As before, let $X_{N_\bullet}=\{x_{l}^{(i)}\}_{i\in I}^{1\le l\le N_i}$ denote our set of variables compatible with our $r$-core via \eqref{Compatibility}.
Define a \textit{shift pattern of $X_{N_\bullet}$} to be a subset of $X_{N_\bullet}$ that contains no more than one variable of each color.
A shift pattern \textit{contains color $p\in I$} if it contains a variable of color $p$.
Let $Sh_p(X_{N_\bullet})$ denote the set of all shift patterns containing color $p$. 

For a shift pattern $\Ju$, let $J\subset I$ denote the colors of the variables in $\Ju$.
We denote the variables in $\Ju$ by $x_{\Ju}^{(i)}$, so $\Ju=\{ x_{\Ju}^{(i)} \}_{i\in J}$.
To $\underline{J}$ we associate the following:
\begin{enumerate}
    \item \textit{Gap labels:} For $i\in I$, let $i^{\triangle}\in J$ be first element greater than \textit{or} equal to $i$ in the cyclic order.
Similarly, let $i^{\triangledown}\in J$ be the first element less than or equal to $i$ in the cyclic order. 
We stipulate that $0\le i^{\triangle}-i, i-i^{\triangledown}\le r-1$.
With this set, we define:
\begin{align*}
   x_{\Ju^\triangle}^{(i)}&=q^{(i-i^{\triangle})}x_{\Ju}^{(i^\triangle)}\\
   x_{\Ju^\triangledown}^{(i)}&=q^{(i-i^{\triangledown})}x_{\Ju}^{(i^\triangledown)}. 
\end{align*}
To clarify, $x_{\Ju^\triangle}^{(i)}=x_{\Ju^\triangledown}^{(i)}=x_{\Ju}^{(i)}$ if $i\in J$.
Thus, while $\Ju$ gives a list of variables colored by $J\subset I$, we `fill in the gaps' for values $i\in I\backslash J$ with certain $q$-shifts of the elements of $\Ju$.
Note that the $q$-shifts are negative for $x_{\Ju^\triangle}^{(i)}$ and positive for $x_{\Ju^\triangledown}^{(i)}$.
    \item \textit{A cyclic-shift operator:} For $i\in J$, let $i^{\blacktriangledown}\in J$ be the first element \textit{strictly} less than $i$ in the cyclic order.
        We set $1\le i-i^\blacktriangledown\le r$, where $r$ occurs if and only if $|J|=\{i\}$. 
        We then define the operator $T_{\Ju}$ on $\K[X_{N_\bullet}]$ as the algebra map induced by
        \begin{align*}
            T_{\Ju}(x_{l}^{(i)})&=
            \left\{
                \begin{array}{ll}
                    q^{(i-i^\blacktriangledown)}x_{\Ju}^{(i^\blacktriangledown)} & \text{if $i\in J$ and $x_{l}=x_{\Ju}^{(i)}$} \\
                     x_{l}^{(i)}&\hbox{otherwise.}
                \end{array}
            \right.
        \end{align*}
	Note that this $q$-shift is positive.
	If we let $i^{\blacktriangle}\in J$ be the first element strictly greater than $i$ in the cyclic order, then observe that
        \begin{align*}
            T^{-1}_{\Ju}(x_{l}^{(i)})&=
            \left\{
                \begin{array}{ll}
                    q^{(i-i^\blacktriangle)}x_{\Ju}^{(i^\blacktriangle)} & \text{if $i\in J$ and $x_{l}^{(i)}=x_{\Ju}^{(i)}$} \\
                     x_{l}^{(i)}&\hbox{otherwise} 
                \end{array}
            \right.
        \end{align*}
	where as before, we view $1\le i^{\blacktriangle}-i\le r$.
	Finally, we note the following: for $i\in J$
\begin{equation}
\begin{aligned}
T_\Ju(x_{\Ju}^{(i)})&= qx_{\Ju^\triangledown}^{(i-1)}\\
T_\Ju^{-1}(x_{\Ju}^{(i)})&= q^{-1}x_{\Ju^\triangle}^{(i+1)}.
\end{aligned}
\label{ShiftRename}
\end{equation}
\end{enumerate}

The cyclic-shift operators $T_{\Ju}^{\pm 1}$ will arise when evaluating the integrals of Lemma~\ref{IntForm} by iterated residues. For later use, and to clarify this relationship, we record the following:

\begin{lem}\label{LemShift}
For any $f=\prod_i f_i[X^{(i)}]\in\Lambda^I$, $\Ju\in Sh_p(X_{N_\bullet})$, define the following evaluations on a set of auxilliary variables $\{w_i\}_{i\in I}$:
\begin{align*}
\mathrm{ev}_{\Ju}^+ : \ \text{for $i=p,p+1,\dotsc,p-1$ in cyclic order,} \ w_i &\mapsto
\begin{cases}
x^{(i)}_{\Ju} & \text{if $i\in J$}\\ 
q^{-1}w_{i+1} & \text{if $i\in I\setminus J$}
\end{cases}\\
\mathrm{ev}_{\Ju}^- :
\ \text{for $i=p+1,p,\dotsc,p+2$ in reverse cyclic order,} \ w_i &\mapsto
\begin{cases}
x^{(i-1)}_{\Ju} & \text{if $i-1\in J$}\\ 
qw_{i-1} & \text{if $i-1\in I\setminus J$.}
\end{cases}
\end{align*}
We then have
\begin{equation}
    \begin{aligned}
        \mathrm{ev}_{\Ju}^+(w_i)&= x_{\Ju^\triangle}^{(i)},&
        \mathrm{ev}_{\Ju}^-(w_i)&= x_{\Ju^\triangledown}^{(i)}
    \end{aligned}
\label{GapEval}
\end{equation}
and
\begin{align}
\label{EqShift+}
T_{\Ju}^{-1}f\left[X_{N_\bullet}\right]
&=
\mathrm{ev}_{\Ju}^+\left(
\prod_{i\in I}
f_i\left[ \sum_{l=1}^{N_i} x_{l}^{(i)}-w_{i}+q^{-1}w_{i+1} \right]
\right)\\
\label{EqShift-}
T_{\Ju}f\left[X_{N_\bullet}\right]
&=
\mathrm{ev}_{\Ju}^-\left(
\prod_{i\in I}
f_i\left[ \sum_{l=1}^{N_i} x_{l}^{(i)}+qw_{i}-w_{i+1} \right]
\right).
\end{align}
\end{lem}

\begin{proof}
The equations (\ref{GapEval}) follow from the definitions.
Equipped with that, the right-hand-side of \eqref{EqShift+} becomes
\[
\prod_{i\in I}f_i\left[\sum_{l=1}^{N_i} x_i^{(i)}-x_{\Ju^\triangle}^{(i)}+q^{-1}x_{\Ju^\triangle}^{(i+1)}\right]
\]
If $i\in I\backslash J$, then $x_{\Ju^\triangle}^{(i)}=q^{-1}x_{\Ju^\triangle}^{(i+1)}$ and so $f_i$ is unchanged.
On the other hand, if $i\in J$, then $x_{\Ju^\triangle}^{(i)}=x_{\Ju}^{(i)}$ and we obtain $T_{\Ju}^{-1}f_i$ by \eqref{ShiftRename}.
The case of \eqref{EqShift-} is similar.
\end{proof}

\begin{ex}
For instance, suppose $r=3$, $p=0$, and $\Ju=\{x^{(0)}_1,x^{(2)}_1\}$. Then the right-hand side of \eqref{EqShift+} is
$$
\mathrm{ev}_{\Ju}^+\left(f_{0}\left[ \sum_{l=1}^{N_{0}} x_{l}^{(0)}-w_0+q^{-1}w_{1} \right] f_1\left[ \sum_{l=1}^{N_1} x_{l}^{(1)}-w_{1}+q^{-1}w_{2} \right] f_{2}\left[ \sum_{l=1}^{N_{2}} x_{l}^{(2)}-w_{2}+q^{-1}w_0 \right]\right)
$$
with $\mathrm{ev}_{\Ju}^+$ given by evaluating $w_0\mapsto x^{(0)}_1,w_1\mapsto q^{-1}w_2,w_2\mapsto x^{(2)}_1$ \textit{in this order}. The result is
$$
f_{0}\left[ \sum_{l=1}^{N_{0}} x_{l}^{(0)}-x^{(0)}_1+q^{-2}x^{(2)}_1 \right] f_1\left[ \sum_{l=1}^{N_1} x_{l}^{(1)}+0\right] f_{2}\left[ \sum_{l=1}^{N_{2}} x_{l}^{(2)}-x^{(2)}_1+q^{-1}x^{(0)}_1\right] = T_{\Ju}^{-1}f\left[X_{N_\bullet}\right].
$$
\end{ex}

We will also make use of $n$-tuples of shift patterns.
For such an $n$-tuple $\mathbf{\underline{J}}=(\underline{J}_1,\ldots, \underline{J}_n)$ and $0\le k\le n$, we denote
\begin{align*}
    |\mathbf{\underline{J}}|&=\underline{J}_1\cup\cdots\cup\underline{J}_n\subset X_{N_\bullet}\\
    |\mathbf{\underline{J}}|_{\le k}&=\underline{J}_1\cup\cdots\cup\underline{J}_{k}\subset X_{N_\bullet}\\
    |\mathbf{\underline{J}}|_{\ge k}&=\underline{J}_k\cup\cdots\cup\underline{J}_{n}\subset X_{N_\bullet}.
\end{align*}
If $\mathbf{\underline{J}}$ is an $n$-tuple of shift patterns all containing color $p$, we say $\mathbf{\underline{J}}$ is \textit{$p$-distinct} if the $p$-colored variables $x_{\underline{J}_k}^{(p)}$ are all distinct.
Let $Sh_p^{[n]}(X_{N_\bullet})$ denote the set of all $p$-distinct $n$-tuples of shift patterns containing color $p$.

\subsection{Degree one case}\label{Deg1}
We will first compute the integrals from Lemma \ref{IntForm} for the case $n=1$.
The \textit{first order wreath Macdonald operators} are defined as follows:
\begin{align*}
D_{p,1}^*(X_{N_\bullet};q,t^{-1})
&:=
\frac{1}{1-q^{-1}t^{-1}}
\sum_{\Ju\in Sh_p(X_{N_\bullet})}
\left( 1-q^{-1}t^{-1} \right)^{|J|}
\frac{x^{(p+1)}_{\Ju^\triangle}}{x_{\Ju^\triangle}^{(0)}}\\
&\times
\left(\prod_{i\in I}
\prod_{\substack{l=1\\x_{l}^{(i)}\not= x^{(i)}_{\Ju^\triangle}}}^{N_i}
\frac{\left(t x^{(i+1)}_{\Ju^\triangle}-x^{(i)}_{l} \right)}
{\left( x^{(i)}_{\Ju^\triangle}-x^{(i)}_{l} \right)}\right)
\left(\prod_{i\in J\backslash\left\{ p \right\}}
\frac{qtT_\Ju^{-1}(x^{(i)}_{\Ju})}{\left( x^{(i)}_{\Ju}-T_\Ju^{-1}(x^{(i)}_{\Ju}) \right)}\right)
T_\Ju^{-1}\\
D_{p,1}(X_{N_\bullet}; q, t^{-1})
&:=
\frac{1}{1-qt}
\sum_{\Ju\in Sh_{p}(X_{N_\bullet})}\left( 1-qt \right)^{|J|}
\frac{x^{(r-1)}_{\Ju^\triangledown}}{x^{(p)}_{\Ju}}\\
&\times
\left(\prod_{i\in I}
\prod_{\substack{l=1\\x^{(i)}_{l}\not= x^{(i)}_{\Ju^\triangledown}}}^{N_i}
\frac{\left( t^{-1}x^{(i-1)}_{\Ju^\triangledown}-x^{(i)}_{l} \right)}
{\left( x^{(i)}_{\Ju^\triangledown}-x^{(i)}_{l} \right)}\right)
\left(\prod_{i\in J\backslash\left\{ p \right\}}
\frac{q^{-1}t^{-1}T_\Ju(x^{(i)}_{\Ju})}{\left( x^{(i)}_{\Ju}-T_\Ju (x^{(i)}_{\Ju}) \right)}\right)
T_\Ju.
\end{align*}
Observe that when $r=1$, $D_{0,1}(x_{0,\bullet};q,t)$ and $D_{0,1}^*(x_{0,\bullet};q,t)$ are the first Macdonald and dual Macdonald operators, respectively.
\begin{prop}\label{Deg1Int}
The integrals from Lemma \ref{IntForm} for $n=1$ yield the following:
\begin{enumerate}
\item[$(+)$]
For
\[
|x_{l}^{(i)}|<1,\, |q|\gg1,\, |t|\gg1,
\]
we have
\begin{equation*}
c_1^+
\pi_{N_\bullet}\biggl( (\rho_{\vec{c}}^+\circ\Psi_+)(\mathcal{E}_{p,1}^+)(f\otimes e^\alpha)\biggr)
=
\left(t^{-|N_\bullet|}D_{p,1}^*(X_{N_{\bullet}};q,t^{-1})+\frac{t^{-p-1-|N_\bullet|}}{1-t^{-r}}\right)f\left[ X_{N_\bullet} \right].
\end{equation*}
\item[$(-)$]
For 
\[
|x_{l}^{(i)}|<1,\, |q|\ll 1,\, |t|\ll 1,
\]
we have
\begin{equation*}
c_1^-
\pi_{N_\bullet}\biggl( (\rho_{\vec{c}}^-\circ\Psi_-)(\mathcal{E}_{p,1}^-)(f\otimes e^\alpha)\biggr)
=
\left(t^{|N_\bullet|}D_{p,1}(X_{N_{\bullet}};q,t^{-1})+\frac{t^{p+1+|N_\bullet|}}{1-t^{r}}\right)f\left[ X_{N_\bullet} \right].
\end{equation*}
\end{enumerate}
\end{prop}

\begin{proof}
In the `$+$' case, the integral from Lemma \ref{IntForm} is:
\begin{align}
\label{E11}
&t^{-1}(1-q^{-1}t^{-1})^{r-1}
\underset{|w_{i,1}|=1}{\oint\cdots\oint}
\prod_{i\in I}\prod_{l=1}^{N_i}
\frac{\left( w_{i+1,1}-t^{-1}x_{l}^{(i)} \right)}
{\left( w_{i,1}-x_{l}^{(i)} \right)}\\
\label{E12}
&\times \left( \frac{w_{p,1}}{w_{0,1}} \right)
\left( \frac{w_{p+1,1}}{w_{p+1,1}-t^{-1}w_{p,1}} \right)
\prod_{i\in I\backslash \{p+1\}}
\left( \frac{w_{i,1}}{w_{i,1}-t^{-1}w_{i-1,1}} \right)
\left( \frac{w_{i-1,1}}{w_{i-1,1}-q^{-1}w_{i,1}} \right)\\
\label{E13}
&\times
\prod_{i\in I}
f_i\left[ \sum_{l=1}^{N_i}x_{l}^{(i)}-w_{i,1}+q^{-1}w_{i+1,1} \right]
\frac{dw_{i,1}}{2\pi \sqrt{-1}w_{i,1}}.
\end{align}
We will first integrate $w_{p,1}$.
Based on (\ref{EExpand}), the residues within the unit circle $|w_{p,1}|=1$ come from the factors:
\[
\frac{1}{\underbrace{\left( w_{p,1}-t^{-1}w_{p-1,1} \right)}_{(\ref{E12})}
\underbrace{\prod_{l=1}^{N_p}\left( w_{p,1}-x_{l}^{(p)} \right)}_{(\ref{E11})}}.
\]
We will call the first type of pole a $t$-pole and the second type an $x$-pole.

The `$-$' case is
\begin{align}
\label{F11} 
&t(1-qt)^{r-1}
\underset{|w_{i,1}|=1}{\oint\cdots\oint}
\prod_{i\in I}\prod_{l=1}^{N_i}
\frac{\left( w_{i,1}-tx_{l}^{(i)} \right)}
{\left( w_{i+1,1}-x_{l}^{(i)} \right)}\\
\label{F12}&\times
\left(\frac{w_{0,1}}
{w_{p,1}-tw_{p+1,1}} \right)
\prod_{i\in I\backslash\left\{ p \right\}}
\left( \frac{w_{i,1}}{w_{i,1}-tw_{i+1,1}} \right)
\left( \frac{w_{i+1,1}}{w_{i+1,1}-qw_{i,1}} \right)\\
\label{F13}&\times
\prod_{i\in I}
f_i\left[ \sum_{l=1}^{N_i}x_{l}^{(i)}+qw_{i,1}-w_{i+1,1} \right]
\frac{dw_{i,1}}{2\pi \sqrt{-1}w_{i,1}}.
\end{align}
Here, we will instead start by integrating $w_{p+1,1}$.
As before, there are $x$-poles and a $t$-pole coming from:
\[
\frac{1}{ \underbrace{\left( w_{p+1,1}-tw_{p+2,1} \right)}_{(\ref{F12})}
\underbrace{\prod_{l=1}^{N_{p+1}}\left( w_{p+1,1}-x_{l}^{(p)} \right)}_{(\ref{F11})}}.
\]
Our analysis of the integrals at these two kinds of poles is addressed in \ref{TPole1} and \ref{XPole1} below.
\end{proof}

\subsubsection{The $t$-poles}\label{TPole1}
First consider the `$+$' case.
Here, we begin with the residue $w_{p,1}=t^{-1}w_{p-1,1}$.
Let us group together the factors
\[
\frac{w_{p,1}w_{p-1,1}}{w_{p,1}\left( w_{p-1,1}-q^{-1}w_{p,1} \right)\left( w_{p,1}-t^{-1}w_{p-1,1} \right)
\displaystyle
\prod_{l=1}^{N_p}\left( w_{p,1}-x_{l}^{(p)} \right)}
\prod_{l=1}^{N_{p-1}}\frac{\left( w_{p,1}-t^{-1}x_{l}^{(p-1)} \right)}
{\left( w_{p-1,1}-x_{l}^{(p-1)} \right)}.
\]
Upon taking taking the residue, this becomes
\[
\frac{t^{-N_{p-1}}}{\left(1-q^{-1}t^{-1}\right)
\displaystyle
\prod_{l=1}^{N_p}\left( t^{-1}w_{p-1,1}-x_{l}^{(p)} \right)}.
\]
Because of the additional restriction $|t|\gg 1$, the poles above will be outside the unit circle $|w_{p-1,1}|=1$.

This pattern persists as we continue \textit{downwards} in cyclic order until we reach $w_{p+1,1}$.
Here, we have
\begin{align*}
&\left.
\frac{w_{p+1,1}w_{p,1}}{w_{p+1,1}w_{0,1}
\left( w_{p+1,1}-t^{-1}w_{p,1} \right)}
\right|_{\substack{w_{0,1}\mapsto t^{p+1-r}w_{p+1,1}\\ w_{p,1}\mapsto t^{-(r-1)}w_{p+1,1}}}
\prod_{l=1}^{N_{p}}
\frac{\left( w_{p+1,1}-t^{-1}x_{l}^{(p)} \right)}
{\left( t^{-r+1}w_{p+1,1}-x_{l}^{(p)} \right)}\\
&= \frac{t^{-p}}{1-t^{-r}}
\cdot\frac{1}{w_{p+1,1}}\prod_{l=1}^{N_{p}}
\frac{\left( w_{p+1,1}-t^{-1}x_{l}^{(p)} \right)}
{\left( t^{-r+1}w_{p+1,1}-x_{l}^{(p)} \right)}.
\end{align*}
The only pole here is the simple pole at $w_{p+1,1}=0$.
After taking this residue, (\ref{E13}) becomes just $f\left[ X_{N_\bullet} \right]$.
Bringing in the front matter in (\ref{E11}), we are left with
\begin{equation*}
\frac{t^{-p-1-|N_\bullet|}}
{1- t^{-r } }
f\left[ X_{N_\bullet} \right].
\end{equation*}
Here, we recall that $N_\bullet=(N_0,\ldots, N_{r-1})$ records the number of $x$-variables and $|N_\bullet|=\sum_{i\in I}N_i$.

For the `$-$' case, recall that we begin at $w_{p+1,1}$ and take the residue $w_{p+1,1}=tw_{p+2}$.
We group together the factors
\[
\frac{w_{p+1,1}w_{p+2,1}}{w_{p+1,1}\left( w_{p+2,1}-qw_{p+1,1} \right)\left( w_{p+1,1}-tw_{p+2,1} \right)
\displaystyle
\prod_{l=1}^{N_p}\left( w_{p+1,1}-x_{l}^{(p)} \right)}
\prod_{l=1}^{N_{p+1}}
\frac{\left( w_{p+1,1}-tx_{l}^{(p+1)} \right)}
{\left( w_{p+1,2}-x_{l}^{(p+1)} \right)}
\]
which upon taking the residue becomes
\[
\frac{t^{N_{p+1}}}{\left( 1-qt \right)
\displaystyle
\prod_{l=1}^{N_p}\left( tw_{p+2,1}-x_{l}^{(p)} \right)}.
\]
The remaining poles above lie outside the unit circle $|w_{p+2,1}|=1$ because we have assumed $|t|\ll 1$.
We continue \textit{upwards} in cyclic order, yielding similar calculations until we arrive at $w_{p,1}$.
Here, we have the factors
\begin{align*}
&\left.\frac{w_{0,1}}
{w_{p,1}\left( w_{p,1}-tw_{p+1,1} \right)}\right|_{\substack{w_{0,1}\mapsto t^p w_{p,1}\\w_{p+1,1}\mapsto t^{r-1}w_{p,1}}}
\prod_{l=1}^{N_p}
\frac{\left( w_{p,1}-tx_{l}^{(p)} \right)}
{\left( t^{r-1}w_{p,1}-x_{l}^{(p)} \right)}\\
&= \frac{t^p}{1-t^r}
\cdot \frac{1}{w_{p,1}}
\prod_{l=1}^{N_p}
\frac{\left( w_{p,1}-tx_{l}^{(p)} \right)}
{\left( t^{r-1}w_{p,1}-x_{l}^{(p)} \right)}.
\end{align*}
The only pole within the unit circle $|w_{p,1}|=1$ is $w_{p,1}=0$.
After taking this residue, the final result (after including the front matter) is
\begin{equation*}
\frac{t^{p+1+|N_\bullet|}}{1-t^r}f\left[ X_{N_\bullet} \right].
\end{equation*}

\subsubsection{The $x$-poles}\label{XPole1}
We will first work out the `$+$' case.
Thus, we have taken the residue of $w_{p,1}$ at the pole $w_{p,1}=x_{l}^{(p)}$ for some $1\le l \le N_p$.
This variable $x_{l}^{(p)}$ will be an element of a shift pattern $\Ju$.
Therefore, we call it $x_{\Ju}^{(p)}$.
It will be advantageous to now group together the factors
\[
\frac{\displaystyle w_{p,1}w_{p+1,1}}
{w_{0,1}w_{p,1}\left( w_{p+1,1}-t^{-1}w_{p,1} \right)}
\prod_{l=1}^{N_{p}}
\frac{\left( w_{p+1,1}-t^{-1}x_{l}^{(p)} \right)}
{\left( w_{p,1}-x_{l}^{(p)} \right)}.
\]
After taking the residue, we leave behind
\[
\frac{w_{p+1,1}}{w_{0,1}}
\prod_{\substack{l=1\\x_{l}^{(p)}\not=x_{\Ju}^{(p)}}}^{N_{p}}
\frac{\left( w_{p+1,1}-t^{-1}x_{1}^{(p)} \right)}
{\left( x_{\Ju}^{(p)}-x_{l}^{(p)} \right)}.
\]

Next, we consider $w_{p+1,1}$.
We group together the factors
\[
\frac{ w_{p+1,1}w_{p+2,1}}
{w_{p+1,1}
\left( w_{p+2,1}-t^{-1}w_{p+1,1} \right)
\underbrace{\left( w_{p+1,1}-q^{-1}w_{p+2,1} \right)}_{(1)}
}
\prod_{l=1}^{N_{p+1}}
\frac{\left( w_{p+2,1}-t^{-1}x_{l}^{(p+1)} \right)}
{\underbrace{\left( w_{p+1,1}-x_{l}^{(p+1)} \right)}_{(2)}}.
\]
The only (nonremovable) poles within the unit circle $|w_{p+1,1}|=1$ are marked (1) and (2).
We thus have two cases:
\begin{enumerate}
\item \textit{Residue at $w_{p+1,1}=q^{-1}w_{p+2,1}$}: In this case, $\left( w_{p+2,1}-t^{-1}w_{p+1,1} \right)$ cancels with a $w_{p+2,1}$ in the numerator, leaving behind
\[
\frac{1}
{\left(1 -q^{-1}t^{-1} \right)}
\left.
\prod_{l=1}^{N_{p+1}}
\frac{
\left( w_{p+2,1}-t^{-1}x_{l}^{(p+1)} \right)}
{\left( w_{p+1,1}-x_{l}^{(p+1)} \right)}\right|_{w_{p+1,1}\mapsto q^{-1}w_{p+2,1}}.
\]
Because $|q|\gg 1$, the poles above lie outside the unit circle $|w_{p+2,1}|=1$.
\item \textit{Residue at $w_{p+1,1}=x_{l}^{(p+1)}=:x_{\Ju}^{(p+1)}$}: Here, $\left( w_{p+2,1}-t^{-1}w_{p+1,1} \right)$ cancels with a factor in the numerator, leaving behind
\begin{equation}
\frac{w_{p+2,1}}{\left(x_{\Ju}^{(p+1)}-q^{-1}w_{p+2,1}\right)}
\prod_{\substack{l=1\\x_{l}^{(p+1)}\not=x_{\Ju}^{(p+1)}}}^{N_{p+1}}
\frac{\left( w_{p+2,1}-t^{-1}x_{1}^{(p+1)} \right)}
{\left( x_{\Ju}^{(p+1)}-x_{l}^{(p+1)} \right)}.
\label{XResE}
\end{equation}
Again because $|q|\gg 1$, the first pole above lies outside the unit circle $|w_{p+2,1}|=1$.
\end{enumerate}
This pattern and dichotomy for residues continues \textit{upwards} in cyclic order.
The $x$-variables in the type (2) residues constitute a shift pattern $\Ju$ and our gap labels $x_{\Ju^\triangle}^{(i)}$ incorporate the $q$-shifts from the type (1) residues.
Therefore, $w_{i,1}$ is always evaluated at $x_{\Ju^\triangle}^{(i)}$.
Finally, observe that by Lemma \ref{LemShift}, (\ref{E13}) becomes $T_{\Ju}^{-1}f\left[ X_{N_\bullet} \right]$.
The end result is $t^{-|N_\bullet|}D_{p,1}^*(q,t^{-1})f\left[ X_{N_\bullet} \right]$.

The `$-$' case is similar.
Our first variable is $w_{p+1,1}$, for which we take the residue at $w_{p+1,1}=x_{l}^{(p)}=:x_{\Ju}^{(p)}$.
We consider the factors
\[
\left(\frac{w_{0,1}}{w_{p+1,1}}  \right)
\frac{1}{\left(w_{p,1}-tw_{p+1,1}\right)}
\prod_{l=1}^{N_p}
\frac{\left(w_{p,1}-tx_{l}^{(p)}  \right)}
{\left( w_{p+1,1}-x_{l}^{(p)} \right)}.
\]
After taking the residue, the pole from $(w_{p,1}-tw_{p+1,1})$ cancels with a factor in the numerator, leaving behind
\[
\frac{w_{0,1}}{x_{\Ju}^{(p)}}
\prod_{\substack{l=1\\x_{l}^{(p)}\not=x_{\Ju}^{(p)}}}^{N_p}
\frac{\left(w_{p,1}-tx_{l}^{(p)}  \right)}
{\left( x_{\Ju}^{(p)}-x_{l}^{(p)} \right)}.
\]
We now proceed \textit{downward} in cyclic order.
For each $w_{i,1}$, we consider the factors
\[
\frac{w_{i,1}w_{i-1,1}}
{w_{i,1}\left( w_{i-1,1}-tw_{i,1} \right)\underbrace{\left( w_{i,1}-qw_{i-1,1} \right)}_{(1)}}
\prod_{l=1}^{N_{i-1}}
\frac{\left( w_{i-1,1}-tx_{l}^{(i-1)} \right)}
{\underbrace{\left( w_{i,1}-x_{l}^{(i-1)} \right)}_{(2)}}.
\]
Because $\left( w_{i,1}-tw_{i+1,1} \right)$ has been canceled at this point, the only poles within the unit circle $|w_{i,1}|=1$ are those marked (1) and (2).
The analysis is as before:
\begin{enumerate}
\item \textit{Residue at $w_{i,1}=qw_{i-1,1}$}: This leaves behind
\[
\frac{1}{(1-qt)}
\prod_{l=1}^{N_{i-1}}
\left.
\frac{\left( w_{i-1,1}-tx_{l}^{(i-1)} \right)}
{\left( w_{i,1}-x_{l}^{(i-1)} \right)}\right|_{w_{i,1}\mapsto qw_{i-1,1}}.
\]
\item \textit{Residue at $w_{i,1}=x_{l}^{(i-1)}=:x_{\Ju}^{(i-1)}$}: The leftovers are now
\begin{equation}
\frac{w_{i-1,1}}{\left( x_{i-1,\Ju}-qw_{i-1,1} \right)}
\prod_{\substack{l=1\\x_{l}^{(i-1)}\not=x_{\Ju}^{(i-1)}}}^{N_{i-1}}
\frac{\left( w_{i-1,1}-tx_{l}^{(i-1)} \right)}
{\left( x_{\Ju}^{(i-1)}-x_{l}^{(i-1)} \right)}.
\label{XResF}
\end{equation}
\end{enumerate}
The $x$-variables where we have taken residues constitute a shift pattern $\Ju$ and $w_{i,1}$ is always evaluated at $x_{\Ju^\triangledown}^{(i-1)}$.
Again by Lemma \ref{LemShift}, (\ref{F13}) becomes $T_{\Ju}f\left[ X_{N_\bullet} \right]$.
Here, we obtain $t^{|N_\bullet|}D_{p,1}(q,t^{-1})f\left[ X_{N_\bullet} \right]$.

\subsubsection{Degree one eigenfunction equation}\label{FirstOp}
Finally, we enhance Proposition \ref{Deg1Int} by obtaining eigenfunction equations for $D_{p,1}^*(q,t)$ and $D_{p,1}(q,t)$ for generic values of the parameters.

\begin{thm}\label{Deg1Thm}
For generic values of $q$, $t$,
\begin{align}
\label{EigenF1}D_{p,1}^*\left( X_{N_\bullet};q,t \right)P_\lambda[X_{N_\bullet};q,t]
&=\left(\sum_{\substack{b=1\\ b-\lambda_b\equiv p+1}}^{|N_\bullet|}
q^{-\lambda_b}t^{-|N_\bullet|+b}\right)
P_\lambda[X_{N_\bullet};q,t]\\
\label{EigenE1}D_{p,1}\left( X_{N_\bullet};q,t \right)P_\lambda[X_{N_\bullet};q,t]
&=\left(\sum_{\substack{b=1\\ b-\lambda_b\equiv p+1}}^{|N_\bullet|}
q^{\lambda_b}t^{|N_\bullet|-b}\right)
P_\lambda[X_{N_\bullet};q,t].
\end{align}
\end{thm}
\begin{proof}
%
We will only consider the `$+$' case---the `$-$' case is similar.
Combining Lemma \ref{FockEigen} and Proposition \ref{Deg1Int}, we have for $\lambda\in\mathbb{Y}$ with $\kb(\lambda)=\alpha$ and $|\mathrm{quot}(\lambda)|\le |N_\bullet|$,
\begin{align*}
\left(t^{-|N_\bullet|}D_{p,1}^*(X_{N_\bullet};q,t^{-1})+\frac{t^{-p-1-|N_\bullet|}}{1-t^{-r}}\right)P_\lambda\left[ X_{N_\bullet}; q,t^{-1} \right]
&= 
\left(\sum_{\substack{b>0\\ b-\lambda_b\equiv p+1}}
q^{-\lambda_b}t^{-b}\right)
P_\lambda[X_{N_\bullet}; q,t^{-1}] 
\end{align*}
where we assume $|q|\gg 1$, $|t|\gg 1$, and $|x_{i,l}|<1$.
Even here, it is essential that $|t|\gg 1$ as we are working with series in $t^{-1}$.
We can do away with this once we notice that since $|N_\bullet|$ is divisible by $r$ (Proposition \ref{SumR}) and $\ell(\lambda)\le |N_\bullet|$,
\begin{equation}
 \begin{aligned}
 \sum_{\substack{b>0\\ b-\lambda_b\equiv p+1}}
q^{-\lambda_b}t^{-b}
&=
\left(\sum_{k=0}^\infty t^{-p-1-|N_\bullet|-rk}\right)+
\left(\sum_{\substack{b=1\\ b-\lambda_b\equiv p+1}}^{|N_\bullet|}
q^{-\lambda_b}t^{-b}\right)\\
&=\frac{t^{-p-1-|N_\bullet|}}{1-t^{-r}}+\left(\sum_{\substack{b=1\\ b-\lambda_b\equiv p+1}}^{|N_\bullet|}
q^{-\lambda_b}t^{-b}\right).   
\end{aligned} 
\label{eigenTail}
\end{equation}
Here, we have split off the terms corresponding to rows above height $|N_\bullet|$.
Thus (\ref{EigenE1}) holds under our conditions on $|q|$, $|t|$, and $|x_{l}^{(i)}|$.

Finally, we address the genericity of parameters.
The equations (\ref{EigenF1}) and (\ref{EigenE1}) are equalities of rational functions in the space $(X_{N_\bullet}, q,t)$.
We have established them over an analytic open subset of $(X_{N_\bullet}, q,t)$.
After subtracting one side to the other, this is equivalent saying a rational function is zero on a codimension zero subspace, and thus it must be zero.
\end{proof}

\noindent
The eigenvalues of $\left\{ D_{p,1}(X_{N_\bullet};q,t) \right\}_{p\in I}$ on $\left\{ P_\lambda[X_{N_\bullet};q,t] \right\}$ are nondegenerate.
Therefore, we have
\begin{cor}\label{CharCor}
For $\lambda$ with core $\kb(\lambda)$ compatible with $N_\bullet$ (cf. \ref{Finite}), the line spanned by $P_\lambda[X_{N_\bullet};q,t]$ is characterized by the eigenfunction equations (\ref{EigenF1}) ranging over all $p\in I$.
\end{cor}

\begin{ex}
Let $r=3$, $p=1$, $N_\bullet=(2,1,0)$, and $\lambda=(3,1,1)$.
In this case, $\lambda$ is a $3$-core and so 
\[P_\lambda[X_{N_\bullet};q,t]=1.\]
There are three shift patterns containing $p=1$:
\begin{align*}
    \Ju_1&=\{x_1^{(1)}\}\\
    \Ju_2&=\{x_1^{(0)},x_1^{(1)}\}\\
    \Ju_3&=\{x_2^{(0)},x_1^{(1)}\}
\end{align*}
The operator $D_{1,1}(X_{N_\bullet};q,t)$ is then
\begin{align}
    \label{Ex11}
    D_{1,1}(X_{N_\bullet};q,t)&=q\left(\frac{qtx_1^{(1)}-x_1^{(0)}}{q^2x_1^{(1)}-x_1^{(0)}}\right)\left(\frac{qtx_1^{(1)}-x_2^{(0)}}{q^2x_1^{(1)}-x_2^{(0)}}\right)T_{\Ju_1}\\
    \label{Ex12}
    &+(1-qt^{-1})q\left(\frac{qtx_1^{(1)}-x_2^{(0)}}{x_1^{(0)}-x_2^{(0)}}\right)\left(\frac{qtx_1^{(1)}}{x_1^{(0)}-q^2x_1^{(1)}}\right)T_{\Ju_2}\\
    \label{Ex13}
    &+(1-qt^{-1})q\left(\frac{qtx_1^{(1)}-x_1^{(0)}}{x_2^{(0)}-x_1^{(0)}}\right)\left(\frac{qtx_1^{(1)}}{x_2^{(0)}-q^2x_1^{(1)}}\right)T_{\Ju_3}.
\end{align}
The cyclic-shift operators act trivially on $P_\lambda(X_{N_\bullet};q,t)$.
Consolidating (\ref{Ex12}) and (\ref{Ex13}) gets us
\begin{align}
    \nonumber
    &(1-qt^{-1})q\left\{\left(\frac{qtx_1^{(1)}-x_2^{(0)}}{x_1^{(0)}-x_2^{(0)}}\right)\left(\frac{qtx_1^{(1)}}{x_1^{(0)}-q^2x_1^{(1)}}\right)+
    \left(\frac{qtx_1^{(1)}-x_1^{(0)}}{x_2^{(0)}-x_1^{(0)}}\right)\left(\frac{qtx_1^{(1)}}{x_2^{(0)}-q^2x_1^{(1)}}\right)\right\}\\
    \nonumber
    &=(1-qt^{-1})q\left\{
   \frac{qtx_1^{(1)}\left(
   qtx^{(1)}{\color{red}x_2^{(0)}}-qtx_1^{(1)}{\color{red}x_1^{(0)}}-{\color{red}x_2^{(0)}x_2^{(0)}+x_1^{(0)}x_1^{(0)}}+q^2{\color{red}x_2^{(0)}}x_1^{(1)}-q^2{\color{red}x_1^{(0)}}x_1^{(1)}
   \right)}
   {{\color{red}(x_1^{(0)}-x_2^{(0)})}(x_1^{(0)}-q^2x_1^{(1)})(x_2^{(0)}-q^2x_1^{(1)})}
    \right\}\\
    \nonumber
    &=(1-qt^{-1})q\left\{
   \frac{qtx_1^{(1)}\left(
   -qtx_1^{(1)}+x_1^{(0)}+x_2^{(0)}-q^2x_1^{(1)}
   \right)}
   {(x_1^{(0)}-q^2x_1^{(1)})(x_2^{(0)}-q^2x_1^{(1)})} 
    \right\}\\
    \nonumber
    &=(1-qt^{-1})q\left\{
   \frac{
   -(q^2t^2+q^3t)x_1^{(1)}x_1^{(1)}+qtx_1^{(0)}x_1^{(1)}+qtx_2^{(0)}x_1^{(1)}
   }
   {(x_1^{(0)}-q^2x_1^{(1)})(x_2^{(0)}-q^2x_1^{(1)})}
    \right\}\\
   \label{Ex14} 
    &=q\left\{
   \frac{
   (-q^2t^2+q^4)x_1^{(1)}x_1^{(1)}+(qt-q^2)x_1^{(0)}x_1^{(1)}+(qt-q^2)x_2^{(0)}x_1^{(1)}
   }
   {(x_1^{(0)}-q^2x_1^{(1)})(x_2^{(0)}-q^2x_1^{(1)})}
    \right\}.
\end{align}
On the other hand, (\ref{Ex11}) becomes
\begin{align}
    \nonumber
    &q\left(\frac{qtx_1^{(1)}-x_1^{(0)}}{q^2x_1^{(1)}-x_1^{(0)}}\right)\left(\frac{qtx_1^{(1)}-x_2^{(0)}}{q^2x_1^{(1)}-x_2^{(0)}}\right)\\
    \label{Ex15}
    &=q\left\{
   \frac{q^2t^2x_1^{(1)}x_1^{(1)}-qtx_1^{(1)}x_2^{(0)}-qtx_1^{(0)}x_1^{(1)}+x_1^{(0)}x_2^{(0)}}
   {(x_1^{(0)}-q^2x_1^{(1)})(x_2^{(0)}-q^2x_1^{(1)})} 
    \right\}.
\end{align}
Combining (\ref{Ex14}) and (\ref{Ex15}) gets us
\begin{align*}
    D_{1,1}(X_{N_\bullet};q,t)P_\lambda[X_{N_\bullet};q,t]&=
    q\left\{
    \frac{q^4x_1^{(1)}x_1^{(1)}-q^2x_1^{(1)}x_2^{(0)}-q^2x_1^{(0)}x_1^{(1)}+x_1^{(0)}x_2^{(0)}}
    {(x_1^{(0)}-q^2x_1^{(1)})(x_2^{(0)}-q^2x_1^{(1)})}
    \right\}\\
    &=q\frac{(x_1^{(0)}-q^2x_1^{(1)})(x_2^{(0)}-q^2x_1^{(1)})}{(x_1^{(0)}-q^2x_1^{(1)})(x_2^{(0)}-q^2x_1^{(1)})}\\
    &=qP_\lambda[X_{N_\bullet};q,t].
\end{align*}
\end{ex}

\begin{ex}
Let $r=2$, $p=0$, $N_\bullet=(1,1)$, and $\lambda=(1,1)$.
Here, 
\[P_\lambda[X_{N_\bullet};q,t]=x_1^{(1)}.\]
There are two shift patterns containing $0$:
\begin{align*}
    \Ju_1&=\{x_1^{(0)}\}\\
    \Ju_2&=\{x_1^{(0)},x_1^{(1)}\}
\end{align*}
We then have
\begin{align*}
    D_{0,1}(X_{N_\bullet};q,t)&=
    q\left(
   \frac{tx_1^{(0)}-x_1^{(1)}}{qx_1^{(0)}-x_1^{(1)}} 
    \right)T_{\Ju_1}
    +(1-qt^{-1})
    \frac{x_1^{(1)}}{x_1^{(0)}}
    \left(
    \frac{tx_1^{(0)}}{x_1^{(1)}-qx_1^{(0)}}
    \right)T_{\Ju_2}.
\end{align*}
Observe that
\begin{align*}
    T_{\Ju_1}x_1^{(1)}&=x_1^{(1)}\\
    T_{\Ju_2}x_1^{(1)}&=qx_1^{(0)}.
\end{align*}
Altogether then,
\begin{align*}
    D_{0,1}(X_{N_\bullet};q,t)P_\lambda[X_{N_\bullet};q,t]
    &=q\left(
   \frac{tx_1^{(0)}-x_1^{(1)}}{qx_1^{(0)}-x_1^{(1)}} 
    \right)x_1^{(1)}
    +(1-qt^{-1})
    \frac{x_1^{(1)}}{x_1^{(0)}}
    \left(
    \frac{tx_1^{(0)}}{x_1^{(1)}-qx_1^{(0)}}
    \right)qx_1^{(0)}\\
    &=qx_1^{(1)}\left(\frac{tx_1^{(0)}-x_1^{(1)}-(t-q)x_1^{(0)}}{qx_1^{(0)}-x_1^{(1)}}\right)\\
    &=qx_1^{(1)}\left(\frac{qx_1^{(0)}-x_1^{(1)}}{qx_1^{(0)}-x_1^{(1)}}\right)\\
    &=qP_\lambda[X_{N_\bullet};q,t].
\end{align*}
\end{ex}

\subsection{Higher degree operators}
Now we consider higher values of $n$.
The \textit{order $n$ wreath Macdonald operators} are defined as follows:
\begin{align}
\label{D*pnDef}
D_{p,n}^*(X_{N_\bullet};q,t^{-1})
&:=
\frac{(-1)^{\frac{n(n-1)}{2}}}{\prod_{k=1}^n(1-q^{-k}t^{-k})}\\
\nonumber
&\times
\sum_{\mathbf{\Ju}\in Sh_p^{[n]}(X_{N_\bullet})}
\overset{\curvearrowright}{\prod_{a=1}^n}
\left\{(1-q^{-1}t^{-1})^{|\Ju_a|}\left(\frac{x^{(p+1)}_{\Ju_a^\triangle}}{x^{(0)}_{\Ju_a^\triangle}}\right)
\frac{\displaystyle
\prod_{\substack{l=1\\x^{(p)}_{l}\not\in|\mathbf{\Ju}|_{\ge a}}}^{N_p}\left( tx^{(p+1)}_{\Ju_a^\triangle}-x^{(p)}_{l} \right)}
{\displaystyle
\prod_{\substack{l=1\\x^{(p)}_{l}\not\in|\mathbf{\Ju}|_{\le a}}}^{N_p}\left( x^{(p)}_{\Ju_a}-x^{(p)}_{l} \right)}\right.\\
\nonumber
&\times
\left.
\left( \prod_{\substack{i\in I\\i\not= p}}\, 
\prod_{\substack{l=1\\x^{(i)}_{l}\not=x^{(i)}_{\Ju_a^\triangle}}}^{N_i}\frac{\left(tx^{(i+1)}_{\Ju_a^\triangle}-x^{(i)}_{l}\right)}
{\left( x^{(i)}_{\Ju_a^\triangle}-x^{(i)}_{l} \right)}
\right)
\left( \prod_{i\in J_a\backslash \left\{ p \right\}} 
\frac{qt T_{\Ju_a}^{-1}(x^{(i)}_{\Ju_a})}
{\left( x^{(i)}_{\Ju_a}-T_{\Ju_a}^{-1}(x^{(i)}_{\Ju_a}) \right)}
\right) T_{\Ju_a}^{-1}\right\}\\
\label{DpnDef}
D_{p,n}(X_{N_\bullet};q,t^{-1})
&:=
\frac{(-1)^{\frac{n(n-1)}{2}}}{\prod_{k=1}^n(1-q^{k}t^{k})}\\
\nonumber
&\times\sum_{\mathbf{\Ju}\in Sh_p^{[n]}(X_{N_\bullet})}
\overset{\curvearrowright}{\prod_{a=1}^n}
\left\{(1-qt)^{|\Ju_a|}\left(\frac{x^{(r-1)}_{\Ju_a^\triangledown}}{x^{(p)}_{\Ju_a}}\right)
\frac{\displaystyle
\prod_{\substack{l=1\\x^{(p)}_{l}\not\in|\mathbf{\Ju}|_{\ge a}}}^{N_p}\left( t^{-1}x^{(p-1)}_{\Ju_a^\triangledown}-x^{(p)}_{l} \right)}
{\displaystyle
\prod_{\substack{l=1\\x^{(p)}_{l}\not\in|\mathbf{\Ju}|_{\le a}}}^{N_p}\left( x^{(p)}_{\Ju_a}-x^{(p)}_{l} \right)}\right.\\
\nonumber
&\times
\left.
\left( \prod_{\substack{i\in I\\i\not= p}}\, 
\prod_{\substack{l=1\\x^{(i)}_{l}\not=x^{(i)}_{\Ju_a^\triangledown}}}^{N_i}\frac{\left(t^{-1}x^{(i-1)}_{\Ju_a^\triangledown}-x^{(i)}_{l}\right)}
{\left( x^{(i)}_{\Ju_a^\triangledown}-x^{(i)}_{l} \right)}
\right)
\left( \prod_{i\in J_a\backslash \left\{ p \right\}} 
\frac{q^{-1}t^{-1} T_{\Ju_a}(x^{(i)}_{\Ju_a})}
{\left( x_{i,\Ju_a}-T_{\Ju_a}(x^{(i)}_{\Ju_a}) \right)}
\right) T_{\Ju_a}\right\}.
\end{align}
Here, recall our notation for ordered products/compositions \eqref{OrdProd}.

\begin{rem}\label{MacCompare}
In contrast with the $n=1$ case, it is less obvious that these yield the higher order Macdonald operators with $t$ inverted when $r=1$.
When $r=1$, note that our sum is over \textit{ordered} $n$-tuples of distinct shift operators, whereas the usual formula for the $n$th Macdonald operator is over \textit{unordered} $n$-tuples.
Summing over the orderings for a given $n$-tuple, the numerator will contain a factor that is antisymmetric, while the denominator will contain a Vandermonde determinant.
The quotient of these two will yield $(-t)^{\pm \frac{n(n-1)}{2}}$ times the $(qt)^{\mp 1}$-generating function of lengths of elements in $\Sym_n$.
After consolidating all constants, one is indeed left with the $n$th Macdonald operator.
\end{rem}

\begin{prop}\label{ResidueCalc}
For general $n$, the integrals from Lemma \ref{IntForm} yield the following:
\begin{enumerate}
\item[$(+)$]
Assuming $|x^{(i)}_{l}|<1$ and $|q|,|t|\gg 1$, we have
\[
\pi_{N_\bullet}\biggl( (\rho_{\vec{c}}^+\circ\Psi_+)(c^+_n\mathcal{E}_{p,n}^+)f\biggr)
= \left(t^{-n|N_\bullet|}D_{p,n}^*(X_{N_\bullet}; q,t^{-1})+\sum_{k=0}^{n-1}c_{p,k,n}^+ D_{p,k}^*(X_{N_\bullet}; q,t^{-1}) \right)f\left[ X_{N_\bullet} \right]
\]
for some $c_{p,k,n}^+\in\C(q,t)$.
\item[$(-)$]
For $|q|,|t|\ll 1$, we have
\[
\pi_{N_\bullet}\biggl( (\rho_{\vec{c}}^-\circ\Psi_-)(c^-_n\mathcal{E}_{p,n}^-)f\biggr)
=\left(t^{n|N_\bullet|}D_{p,n}(X_{N_\bullet};q,t^{-1}) +\sum_{k=0}^{n-1}c_{p,k,n}^-D_{p,k}(X_{N_\bullet};q,t^{-1})\right)f\left[ X_{N_\bullet} \right]
\]
for some $c_{p,k,n}^-\in\C(q,t)$.
\end{enumerate}
\end{prop}

\begin{proof}
In the `$+$' case, we will start by integrating the $p$-colored variables $\{w_{p,\bullet}\}$.
There are two kinds of poles inside the unit circle $|w_{p,b}|=1$:
\begin{enumerate}
\item[($x$)] the poles $\left( w_{p,b}-x_{l}^{(p)} \right)$ in (\ref{En1}) and
\item [($t$)] the poles $\left( w_{p,b}-t^{-1}w_{p-1,a} \right)$ for $a\le b$ in (\ref{En3}) and (\ref{En5}).
\end{enumerate}
As in \ref{Deg1}, we call them $x$- and $t$-poles, respectively.
We note that evaluating two variables $w_{p,b}$ and $w_{p,b'}$ at the same pole will result in zero due to the factor $\left(w_{p,b}-w_{p,b'}\right)$ in (\ref{En3}).
Besides that, for $r>1$, these residues can be evaluated independently and we elect to do so.
For the '$-$' case, we instead start with $\{w_{p+1,\bullet}\}$, for which the relevant poles are now
\begin{enumerate}
\item[($x$)] $\left( w_{p+1,a}-x_{l}^{(p)} \right)$ in (\ref{Fn1}) and
\item[($t$)] $\left( w_{p+1,a}-tw_{p+2,b} \right)$ for $a\le b$ in (\ref{Fn3}) and (\ref{Fn4}).
\end{enumerate}
In \ref{XPoleN} and \ref{MPoleN} below, we analyze the results of the two possibilities:
\begin{enumerate}
\item integrating all $w_{p,\bullet}$ at $x$-poles;
\item the `mixed' case where some $w_{p,\bullet}$ is integrated at a $t$-pole.
\end{enumerate}
The first case produces $D_{p,n}^*(q,t^{-1})$ and $D_{p,n}(q,t^{-1})$, whereas the second case yields a combination of lower order wreath Macdonald operators.
\end{proof}

\subsubsection{Only $x$-poles}\label{XPoleN}
In both the `$+$' and `$-$' cases, each of the $n$ variables $\{ x_{l_a}^{(p)} \}_{a=1}^n$ will become part of a shift pattern containing $p$, so we set $x_{\Ju_a}^{(p)}:=x_{l_a}^{(p)}$.
Furthermore, as these variables must be distinct, we have that the tuple $\mathbf{\Ju}:=\left( \Ju_1,\ldots,\Ju_n \right)$ will be $p$-distinct.
After taking these residues, we will proceed as in \ref{XPole1} for a specific value of $a$.

First consider the `$+$' case.
To see the effect of taking the residues $w_{p,b}=x_{\Ju_b}^{(p)}$, we group together the factors
\begin{align*}
&\left(\frac{w_{p,b}}{w_{0,b}}\right)\frac{w_{p+1,b}}{w_{p,b}\left( w_{p+1,b}-t^{-1}w_{p,b} \right)}
\prod_{l=1}^{N_p}
\frac{\left( w_{p+1,b}-t^{-1}x_{l}^{(p)} \right)}
{\displaystyle\left( w_{p,b}-x_{l}^{(p)} \right)}\\
&\times
\frac{1}
{\displaystyle
\prod_{b<c}
\left( w_{p+1,b}-t^{-1}w_{p,c} \right)}
\prod_{a<b}\frac{
\displaystyle
\left( w_{p,b}-w_{p,a} \right)\left( w_{p+1,b}-q^{-1}t^{-1}w_{p+1,a} \right)}
{\left( w_{p+1,a}-t^{-1}w_{p,b} \right)}.
\end{align*}
Upon taking residues, this becomes
\begin{equation}
\left(\frac{w_{p+1,b}}{w_{0,b}}\right)
\frac{\displaystyle
\prod_{\substack{l=1\\x_{l}^{(p)}\not\in|\mathbf{\Ju}|}}^{N_p}
\left( w_{p+1,b}-t^{-1}x_{l}^{(p)} \right)}
{\displaystyle
\prod_{\substack{l=1\\x_{l}^{(p)}\not\in|\mathbf{\Ju}|_{\le b}}}^{N_p}
\left( x_{\Ju_b}^{(p)}-x_{l}^{(p)} \right)}
\underbrace{\prod_{a<b}\left( w_{p+1,b}-q^{-1}t^{-1}w_{p+1,a} \right)}_{(\dagger)}.
\label{APLeft}
\end{equation} 

The next variable we consider is $w_{p+1,1}$.
Notice that we have canceled the poles $\left( w_{p+1,1}-t^{-1}w_{p,b} \right)$ for all $b\ge 1$, and consequently, the only two kinds of poles within the unit circle $|w_{p+1,1}|=1$ are as before in \ref{XPole1}.
We group together the factors
\begin{align*}
&\frac{w_{p+1,1}w_{p+2,1}}
{w_{p+1,1}\left( w_{p+2,1}-t^{-1}w_{p+1,1} \right)\left( w_{p+1,1}-q^{-1}w_{p+2,1} \right)}\\
&\times
\underbrace{\prod_{1<b}
\frac{ \left( w_{p+1,b}-w_{p+1,1} \right)\left(w_{p+2,b}-q^{-1}t^{-1}w_{p+2,1} \right)}
{\left( w_{p+2,b}-t^{-1}w_{p+1,1} \right)\left( w_{p+1,b}-q^{-1}w_{p+2,1} \right)}}_{(*)}
\underbrace{\prod_{b=1}^n\prod_{l=1}^{N_{p+1}}
\frac{\left( w_{p+2,b}-t^{-1}x_{l}^{(p+1)} \right)}
{\left( w_{p+1,b}-x_{l}^{(p+1)} \right)}}_{(**)}.
\end{align*}
The residues are
\begin{enumerate}
\item \textit{Residue at $w_{p+1,1}=q^{-1}w_{p+2,1}$}:
In this case, the factors in $(*)$ cancel out, leaving behind
\[
\frac{1}{(1-q^{-1}t^{-1})}
\left.\prod_{b=1}^n\prod_{l=1}^{N_{p+1}}
\frac{\left( w_{p+2,b}-t^{-1}x_{l}^{(p+1)} \right)}
{\left( w_{p+1,b}-x_{l}^{(p+1)} \right)}\right|_{w_{p+1,1}\mapsto q^{-1}w_{p+2,1}}.
\]
As in \ref{XPole1}, $w_{p+1,1}$ will ultimately be evaluated at $x^{(p+1)}_{\Ju_1^\triangle}$ and the poles above lie outside the unit circle $|w_{p+2,1}|=1$ because $|q|\gg1$.
\item \textit{Residue at $w_{p+1,1}=x_{l}^{(p+1)}=:x^{(p+1)}_{\Ju_1}$}:
Here, the factors in $(*)$ cancel with those in $(**)$ containing $x^{(p+1)}_{\Ju_1}$.
We are left with
\begin{align}
\nonumber&\frac{w_{p+2,1}}{\left( x^{(p+1)}_{\Ju_1}-q^{-1}w_{p+2,1} \right)}
\prod_{\substack{l=1\\x^{(p+1)}_{l}\not=x^{(p+1)}_{\Ju_1}}}^{N_{p+1}}
\frac{\left( w_{p+2,1}-t^{-1}x^{(p+1)}_{l} \right)}
{\left( x_{p+1,\Ju_1}-x^{(p+1)}_{l} \right)}\\
\label{nRename}
&\times\prod_{1<b}
\left\{\frac{\left( w_{p+2,b}-q^{-1}t^{-1}w_{p+2,1} \right)}
{\left( w_{p+1,b}-q^{-1}w_{p+2,1} \right)}
\prod_{\substack{l=1\\x^{(p+1)}_{l}\not=x^{(p+1)}_{\Ju_1}}}^{N_{p+1}}
\frac{\left( w_{p+2,b}-t^{-1}x^{(p+1)}_{l} \right)}
{\left(w_{p+1,b}-x^{(p+1)}_{l}\right)}\right\}.
\end{align}
Because $|q|\gg 1$, the pole $\left( x^{(p+1)}_{\Ju_1}-q^{-1}w_{p+2,1} \right)$ lies outside the unit circle $|w_{p+2,1}|=1$.
Our key organizational trick here is that when $w_{p+2,1}$ is ultimately evaluated at $x^{(p+2)}_{\Ju_1^\triangle}$, then we can use (\ref{ShiftRename}) to write (\ref{nRename}) as
\[
T_{\Ju_1}^{-1}\left(
\prod_{1<b}^n\prod_{l=1}^{N_{p+1}}
\frac{\left( w_{p+2,b}-t^{-1}x^{(p+1)}_{l} \right)}
{\left( w_{p+1,b}-x^{(p+1)}_{l} \right)}
\right)
\]
since $T_{\Ju_1}$ will only affect $x^{(p+1)}_{\Ju_1}$.
\end{enumerate}

This pattern continues \textit{upwards} in cyclic order for the variables $w_{i,1}$.
The $x$-variables where we take residues gives a shift pattern $\Ju_1$ containing $p$ and $w_{i,1}$ is evaluated at $x^{(i)}_{\Ju_1^\triangle}$.
In (\ref{APLeft}), the term in $(\dagger)$ for $a=1$ can be rewritten as $\left( w_{p+1,b}-t^{-1}T_{\Ju_1}^{-1}x^{(p)}_{\Ju_1} \right)$.
Finally, we note that by Lemma \ref{LemShift}, these residues result in $T_{\Ju_1}^{-1}$ applied to $f\left[ X_{N_\bullet} \right]$.
Thus, we can rewrite the result after taking the residues for $a=1$ as:
\begin{align*}
&\frac{(-1)^{\frac{n(n-1)}{2}}t^{-\frac{n(n+1)}{2}}(1-q^{-1}t^{-1})^{r(n-1)}}{\prod_{a=1}^n(1-q^{-a}t^{-a})}
\sum_{\Ju_1\in Sh_p}(1-q^{-1}t^{-1})^{|\Ju_1|}\left(\frac{x^{(p+1)}_{\Ju_1^\triangle}}{x^{(0)}_{\Ju_1^\triangle}}\right)
\\
&\times
\frac{\displaystyle
\prod_{\substack{l=1\\x^{(p)}_{l}\not\in|\mathbf{\Ju}|_{\ge 1}}}^{N_p}\left( x^{(p+1)}_{\Ju_1^\triangle}-t^{-1}x^{(p)}_{l} \right)}
{\displaystyle
\prod_{\substack{l=1\\x^{(p)}_{l}\not\in|\mathbf{\Ju}|_{\le 1}}}^{N_p}\left( x^{(p)}_{\Ju_1}-x^{(p)}_{l} \right)}
\left( \prod_{\substack{i\in I\\i\not= p}}\, 
\prod_{\substack{l=1\\x^{(i)}_{l}\not=x^{(i)}_{\Ju_1^\triangle}}}^{N_i}\frac{\left(x^{(i+1)}_{\Ju_1^\triangle}-t^{-1}x^{(i)}_{l}\right)}
{\left( x^{(i)}_{\Ju_1^\triangle}-x^{(i)}_{l} \right)}
\right)
\left( \prod_{i\in J_1\backslash \left\{ p \right\}} 
\frac{q T_{\Ju_1}^{-1}x^{(i)}_{\Ju_1}}
{\left( x^{(i)}_{\Ju_1}-T_{\Ju_1}^{-1}x^{(i)}_{\Ju_1} \right)}
\right)\\
&\times
\underset{|w_{i,a}|=1}{\oint\cdots\oint}
T_{\Ju_1}^{-1}\left(
\prod_{a=1}^{n}\left\{
\left( w_{p+1,a}-t^{-1}x^{(p)}_{\Ju_1} \right)
\prod_{1< a<b\le n}
\left( w_{p+1,b}-q^{-1}t^{-1}w^{(p+1)}_{a} \right)
\right.\right.\\
&\times
\left.
\frac{\displaystyle
\prod_{\substack{l=1\\ x^{(p)}_{l}\not\in|\mathbf{\Ju}|}}^{N_p}\left( w_{p+1,a}-t^{-1}x^{(p)}_{l} \right)}
{\displaystyle
\prod^{N_p}_{\substack{l=1\\x^{(p)}_{l}\not\in |\mathbf{\Ju}|_{\le a}}}\left( x^{(p)}_{\Ju_a}-x^{(p)}_{l} \right)}
\prod_{i\in I\backslash\left\{ p \right\}}\prod_{l=1}^{N_{i}}
\frac{\displaystyle \left( w_{i+1,a}-t^{-1}x^{(i)}_{l} \right)}{\displaystyle\left( w_{i,a}-x^{(i)}_{l} \right)}\right\}\\
&\times\prod_{1< a<b\le n}
\prod_{i\in I\backslash\{p\}}
\frac{(w_{i,b}-w_{i,a})(w_{i+1,b}-q^{-1}t^{-1}w_{i+1,a})}{(w_{i+1,b}-t^{-1}w_{i,a})(w_{i,b}-q^{-1}w_{i+1,a})}\\
&\times\prod_{a=2}^{n} \left\{
\left(\frac{w_{p+1,a}}{w_{0,a}}\right)
\prod_{i\in I\setminus\{p\}}
\frac{w_{i,a}w_{i+1,a}}{\left(w_{i+1,a}-t^{-1}w_{i,a}\right)\left(w_{i,a}-q^{-1}w_{i+1,a}\right)} \right\}\\
&\left.
\prod_{i\in I}f_i\left[ \sum_{l=1}^{N_i}x_{l}^{(i)}-\sum_{a=2}^n w_{i, a}+\sum_{a=2}^nq^{-1} w_{i+1,a} \right]\right)\prod_{i\in I\backslash\{p\}}\prod_{a=2}^n\frac{dw_{i,a}}{2\pi\sqrt{-1}w_{i,a}}.
\end{align*}
We have written this so that we can repeat the calculation for $a=1$ for general $a$ in increasing order.
Note that as we do this, we can rewrite factors in $(\dagger)$ of (\ref{APLeft}) in terms of $T_{\Ju_a}^{-1}x_{p,\Ju_a}$ using (\ref{ShiftRename}).
The end result of the residue calculation is 
\[t^{-n|N_\bullet|}D_{p,n}^*(X_{N_\bullet}; q,t^{-1})f\left[ X_{N_\bullet} \right].\]

The `$-$' case is similar.
We begin by taking residues of $\{w_{p+1,\bullet}\}$ and then start instead at $x^{(p)}_{n}$.
Afterwards, we continue \textit{downwards} in cyclic order until we have taken constant terms of all variables with $a=n$.
We then continue \textit{downwards} in $a$.
The end result is then
\[
t^{n|N_\bullet|}D_{p,n}(X_{N_\bullet};q,t^{-1})f\left[ X_{N_\bullet} \right].
\]

\subsubsection{Mixed poles}\label{MPoleN}
In the case where there are $t$-poles, our goal is to show that the result is a linear combination of the lower order operators applied to $f\left[ X_{N_\bullet} \right]$: $D_{p,k}^*(X_{N_\bullet}; q,t^{-1})$ in the `$+$' case and $D_{p,k}(X_{N_\bullet}; q,t^{-1})$ in the `$-$' case, where $k<n$.
Unlike in the case of $n=1$, we will not try to compute the coefficients of this linear combination---we will compute them indirectly in \ref{Eigen}.
As in all the previous cases, the initial residues force a string of other residues, and we will first compute these strings that start from the initial $t$-poles.
Once these variables are evaluated, the remaining terms will evaluate like \ref{XPoleN}.

In the `$+$' case, let $1\le b_1^p\le n$ be any index where the residue for $w_{p, b_{1}^p}$ is taken at a $t$-pole.
Denote this pole by $w_{p,b_1^p}=t^{-1}w_{p-1,b_1^{p-1}}$.
In contrast to our previous calculations, we will not always cancel out factors but rather remark on why taking residues at certain poles will result in zero.
The poles contributing within the unit circle $|w_{p-1,b_1^{p-1}}|=1$ are as follows.
\begin{enumerate}
\item \textit{$(w_{p-1,b_1^{p-1}}-q^{-1}w_{p,a}  )$ for $a\ge b_1^{p-1}$}: If $a<b_1^p$, then the factor $( w_{p,b_1^p}-q^{-1}t^{-1}w_{p,a } )$ in the numerator of (\ref{En3}) becomes zero when taking this residue.
If $a=b_1^{p-1}=b_1^p$, then this is a pole at $0$, which cancels with the extra factor of $w_{p-1,b_1^{p-1}}$ as in \ref{TPole1}. 
\item \textit{$( w_{p-1,b_1^{p-1}}-x^{(p-1)}_{l} )$}: The factor $( w_{p,b_1^p}-t^{-1}x^{(p-1)}_{l} )$ in the numerator of (\ref{En1}) will evaluate to zero.
\item \textit{$( w_{p-1,b_1^{p-1}}-t^{-1}w_{p-2,a} )$ for $a\le b_1^{p-1}$}: These poles possibly yield nonzero residues.
\end{enumerate}
Taking a residue of the third kind, we evaluate $w_{p-1,b_1^{p-1}}=t^{-1}w_{p-2,b_1^{p-2}}$ for some $b_1^{p-2}\le b_1^{p-1}$.

This pattern continues \textit{downwards} in cyclic order, picking out variables $w_{i, b_1^i}$ where $b_1^p\ge b_1^{p-1}\ge\cdots\ge b_1^{p+1}$.
At $w_{p+1,b_1^{p+1}}$, the pole of type (3) becomes
\begin{enumerate}
\item[(3')] \textit{$(w_{p+1,b_1^{p+1}}-t^{-1}w_{p,a})$ for all $a$}: If $w_{p,a}$ is evaluated at an $x$-variable $x^{(p)}_{l}$, then as in \ref{XPoleN}, the factor
$(w_{p+1,b_1^{p+1}}-t^{-1}x^{(p)}_{l})$ will evaluate to zero upon taking this residue.
Thus, only the case where $w_{p,a}$ is evaluated at a $t$-pole yields a nonzero residue.
For $a=b_1^p$, this is a pole at $w_{p+1,b_1^{p+1}}=0$.
If $b_{1}^{p+2}=b_1^{p+1}$, then because of the analogue of case (1), there are no extra powers of $w_{p+1,b_1^{p+1}}$ to cancel this pole.
\end{enumerate}

If we take the residue in (3') at $w_{p,a}$ evaluated at a $t$-pole but $a\not=b_1^p$, then we set $b_2^p:=a$.
Letting the $t$-pole be $(w_{p,b_2^p}-t^{-1}w_{p-1,b_2^{p-1}})$ for $b_2^{p-1}\le b_2^p$, the process is similar to as before.
There is just one alteration to the poles of type (3):
\begin{enumerate}
\item[(3'')]\textit{$(w_{i, b_2^{i}}-t^{-1}w_{i-1,b_1^{i-1}})$}: This is a pole at $0$, which cancels with the factor $\pm(w_{i, b_2^i}- w_{i, b_1^i})$ in the numerator of (\ref{En2}).
\end{enumerate}
Thus, we avoid variables that we have already evaluated.
Note that at first glance, the product of factors in (\ref{En2}) and (\ref{En3}) involving $w_{i, b_2^i}$ and $w_{j,b_1^j}$ may contribute a pole at $0$, but in fact, their products have total degree zero and thus become a constant.
There is an outlier case of $(w_{p+1,b_1^{p+1}}-t^{-1}w_{p,b_{2}^p})$, which has been removed when we take residues, but this can be replaced with $(w_{p+1,b_1^{p+1}}-t^{-1}w_{p,b_{1}^p})$ to restore the degree zero balance.
We continue like this to new indices $\{b_3^i\}_{i\in I}, \left\{ b_{4}^i \right\}_{i\in I},$ etc. until either there are no more nonzero residues or we finally take the residue at $0$ of $z_{p+1, b_k^{p+1}}$ for some final value $k$. 

For $1\le m<m'\le k$, we note that as in the $(m,m')=(1,2)$ case, the product of the binomials in (\ref{En2}) and (\ref{En3}) involving one variable from $\{w_{i,b_m^i}\}_{i\in I}$ and another variable from $\{w_{j, b_{m'}^j}\}_{i\in I}$ has degree zero provided we make the same adjustment for $i=p+1$ and $m'=m+1$.
Thus, these factors turn into a constant.
To consider binomials involving only $\{w_{i,b_m^i}\}_{i\in I}$ for one value of $m$, we note that when we take the residues, we remove
\begin{alignat*}{2}
&\frac{1}{w_{i,b_m^i}-t^{-1}w_{i+1,b_m^{i+1}}} &&\hspace{1cm}\hbox{for }i\not=p+1,\\
&\frac{1}{w_{p+1,b_m^{p+1}}-t^{-1}w_{p,b_{m+1}^{p}}} &&\hspace{1cm}\hbox{for }1\le m<k.
\end{alignat*}
There is a leftover power of $w_{i,b_{m}^i}$ for $i\not=p$ from (\ref{En4}) and (\ref{En5}), and as discussed in the pole of type (1) above, these are only absorbed when $b_{m}^{i+1}=b_{m}^i$.
These unabsorbed powers turn the entire integral zero when we take the final residue $w_{p+1,b_k^{p+1}}=0$.
Thus, we only need to consider the case where for each $m$,
\[b_m^p=b_m^{p-1}=\cdots=b_m^{p+1}=:b_m.\]
In this case, all factors only involving $\{w_{i,b_m}\}_{i\in I}^{1\le m\le k}$ leave behind a constant.
Evidently, the corresponding terms in (\ref{En4}) and (\ref{En5}) disappear.
The terms involving $w_{i,b_m}$ and an $x$-variable in (\ref{En1}) leave behind a power of $t$ when we cancel
\[
\prod_{m=1}^k\prod_{i\in I}\prod_{l=1}^{N_i}\frac{\left( w_{i+1,b_{m}}-t^{-1}x^{(i)}_{l} \right)}
{\left( w_{i,b_{m}}-x^{(i)}_{l} \right)}.
\]
Finally, the product of terms in (\ref{En2}) involving any index $1\le a\le n$ and $b_m$ leave behind a constant when we evaluate $w_{i,b_m}=0$ for all $i\in I$.
The remaining factors are a scalar multiple of the calculation for $\mathcal{E}_{p,n-k}$.
The `$-$' case is analyzed similarly.

\subsection{Eigenvalues}\label{Eigen}
To describe the eigenvalues of the operators (\ref{D*pnDef}) and (\ref{DpnDef}), we will use the elementary symmetric functions $ e_k$.
As in the proof of Theorem \ref{Deg1Thm}, Proposition \ref{ResidueCalc} gives us:
\begin{prop}\label{CombiEigen}
Recall the coefficients $\{c_{p,k,n}^\pm\}$ from Proposition \ref{ResidueCalc}.
We have:
\begin{equation}
\begin{aligned}
&\left(t^{-n|N_\bullet|}D_{p,n}^*(X_{N_\bullet}; q,t^{-1})+\sum_{k=0}^{n-1}c_{p,k,n}^+ D_{p,k}^*(X_{N_\bullet}; q,t^{-1}) \right)P_\lambda[X_{N_\bullet};q,t]\\
&= e_n\left[\sum_{\substack{b=1\\b-\lambda_b\equiv p+1}}^\infty q^{-\lambda_b}t^{-b}\right]P_\lambda[X_{N_\bullet};q,t^{-1}]
\end{aligned}
\label{EEigenEq}
\end{equation}
for $|x^{(i)}_{l}|<1$, $|q|\gg1$, and $|t|\gg 1$ and
\begin{equation}
\begin{aligned}
&\left(t^{n|N_\bullet|}D_{p,n}(X_{N_\bullet};q,t^{-1}) +\sum_{k=0}^{n-1}c_{p,k,n}^-D_{p,k}(X_{N_\bullet};q,t^{-1})\right)P_\lambda[X_{N_\bullet};q,t^{-1}]\\
&= e_n\left[\sum_{\substack{b=1\\b-\lambda_b\equiv p+1}}^\infty q^{\lambda_b} t^{b}\right]P_\lambda[X_{N_\bullet};q,t^{-1}]
\end{aligned}
\label{FEigenEq}
\end{equation}
for $|x^{(i)}_{l}|<1$, $|q|\ll1$, and $|t|\ll 1$.
\end{prop}

\begin{cor}
For variables and parameters satisfying the conditions in Proposition \ref{CombiEigen}, the operators $D_{p,n}(q,t)$ and $D_{p,n}^*(q,t)$ act diagonally on $\left\{ P_\lambda\left[ X_{N_\bullet};q,t \right] \right\}$.
\end{cor}

\begin{proof}
Using induction starting with the case $n=1$ from Theorem \ref{Deg1Thm}, we can use the equations in Proposition \ref{CombiEigen} to show that $D_{p,n}(X_{N_\bullet};q,t)$ and $D_{p,n}^*(X_{N_\bullet};q,t)$ act diagonally on $P_\lambda[X_{N_\bullet};q,t]$ under the appropriate conditions on variables and parameters.
\end{proof}

Our goal in this subsection is to extract the eigenvalues from (\ref{EEigenEq}) and (\ref{FEigenEq}) and extend their validity to generic values.

\subsubsection{Spectral variables}
Letting $\lambda$ vary over partitions with $\core(\lambda)$ compatible with $N_\bullet$ and $\ell(\lambda)\le |N_\bullet|$, we note that by Proposition \ref{SumR}, the stabilized eigenvalues
\[
e_n\left[\sum_{\substack{b=1\\b-\lambda_b\equiv p+1}}^\infty q^{-\lambda_b}t^{-b}\right]
\hbox{ and }
e_n\left[\sum_{\substack{b=1\\b-\lambda_b\equiv p+1}}^\infty q^{\lambda_b} t^{b}\right]
\]
depend only on the $N_p$ values of $b$ where $1\le b\le |N_\bullet|$ and $b-\lambda_b=p+1$.
We define the \textit{color $p$ spectral variables} $\{ s^{(p)}_{ a} \}_{a=1}^{N_p}$ by setting
\[
s^{(p)}_{a}=q^{\lambda_{b_a}}t^{b_a}
\]
where $1\le b_a\le|N_\bullet|$ is the $a$th number where $b_a-\lambda_{b_a}\equiv p+1$.
Using these variables, we can rewrite
\begin{align*}
 e_n\left[\sum_{\substack{b=1\\b-\lambda_b\equiv p+1}}^\infty q^{-\lambda_b}t^{-b}\right]
&=e_n\left[\sum_{k=0}^{\infty}t^{|N_\bullet|-p-1-kr}+\sum_{\substack{b=1\\b-\lambda_b\equiv p+1}}^{|N_\bullet|} q^{-\lambda_b}t^{-b}\right]\\
&=e_n\left[\frac{t^{-|N_\bullet|-p-1}}{1-t^{-nr}}+\sum_{a=1}^{N_p}\left(s^{(p)}_{a}\right)^{-1} \right]   
\end{align*}
where $|t|\gg 1$.
Here, we have split off the parts above row $|N_\bullet|$ as in \eqref{eigenTail}.
Similarly,
\begin{align*}
e_n\left[\sum_{\substack{b=1\\b-\lambda_b\equiv p+1}}^\infty q^{\lambda_b} t^{b}\right]
&=e_n\left[\sum_{k=0}^\infty t^{|N_\bullet|+p+1+kr}+\sum_{\substack{b=1\\b-\lambda_b\equiv p+1}}^{|N_\bullet|} q^{\lambda_b} t^{b}\right]\\
&=e_n\left[ \frac{t^{|N_\bullet|+p+1}}{1-t^{nr}}+\sum_{a=1}^{N_p}s^{(p)}_{a} \right]   
\end{align*}
where $|t|\ll 1$.
The following is but a slight alteration of Lemma 3.2 from \cite{FHHSY}:
\begin{lem}\label{EigenStab}
For $|t|\gg 1$, we have
\begin{equation}
e_n\left[\frac{t^{-|N_\bullet|-p-1}}{1-t^{-nr}}+\sum_{a=1}^{N_p}\left(s^{(p)}_{a}\right)^{-1} \right]=\sum_{k=0}^n\frac{t^{-n|N_\bullet|-(n-k)(p+1)-r\binom{n-k}{2}}}{\displaystyle\prod_{l=1}^{n-k}(1-t^{-rl})}
e_k\left[ \sum_{a=1}^{N_p}t^{|N_\bullet|}\left(s^{(p)}_{a}\right)^{-1} \right].
\label{MinusSpec}
\end{equation}
while for $|t|\ll 1$, we have
\begin{equation}
e_n\left[\frac{t^{|N_\bullet|+p+1}}{1-t^{nr}}+\sum_{a=1}^{N_p}s^{(p)}_{a} \right]=\sum_{k=0}^n\frac{t^{n|N_\bullet|+(n-k)(p+1)+r\binom{n-k}{2}}}{\displaystyle\prod_{l=1}^{n-k}(1-t^{rl})}
e_k\left[ \sum_{a=1}^{N_p}t^{-|N_\bullet|}s^{(p)}_{a} \right]
\label{PlusSpec}
\end{equation}
\end{lem}

\begin{proof}
The basic observation is that for two alphabets $X$ and $Y$ and an auxilliary variable $u$,
\begin{align*}
    \sum_{n=0}^\infty e_n[X+Y]u^n
    &=\exp\left(-\sum_{k>0} p_k[X+Y] \frac{(-u)^k}{k}\right)\\
    &=\exp\left(-\sum_{k>0} p_k[X] \frac{(-u)^k}{k}\right)\exp\left(-\sum_{k>0} p_k[Y] \frac{(-u)^k}{k}\right)\\
    &=\left(\sum_{n=0}^\infty e_n[X]u^n\right)\left(\sum_{n=0}^\infty e_n[Y]u^n\right)\\
    &=\sum_{n=0}^\infty\sum_{k=0}^ne_{n-k}[X]e_k[Y] u^n
\end{align*}
Comparing the coefficients of $u^n$, we thus have
\begin{equation}
 e_n[X+Y]=\sum_{k=0}^n e_{n-k}[X]e_{k}[Y]
 \label{FactorEl}
\end{equation}

For \eqref{MinusSpec}, we take \eqref{FactorEl} and set
\begin{equation*}
 \begin{aligned}
    X&=\frac{t^{-|N_\bullet|-p-1}}{1-t^{-nr}}=t^{-p-1}\sum_{k=0}^\infty t^{-rk},&
    Y&=\sum_{a=1}^{N_p}\left(s_a^{(p)}\right)^{-1}
\end{aligned}   
\end{equation*}
By the quantum binomial theorem (cf. \cite[Example I.2.5]{Mac}), we have
\[
e_{n-k}\left[\frac{t^{-|N_\bullet|-p-1}}{1-t^{-nr}}\right]=
\frac{t^{(-|N_\bullet|-p-1)(n-k)-r\binom{n-k}{2}}}{\displaystyle\prod_{l=1}^{n-k}(1-t^{-rl})}
\]
To obtain \eqref{MinusSpec}, we break off $t^{k|N_\bullet|}$ and place it inside $e_k\left[\sum (s_a^{(p)})^{-1}\right]$.
The proof of \eqref{PlusSpec} is similar.
\end{proof}

\subsubsection{Spectral shift}
By Lemma \ref{EigenStab}, the stabilized eigenvalues are polynomial in the spectral variables.
Moreover, its degree $k$ part is given by $e_k$ evaluated at $\{t^{-|N_\bullet|}s^{(p)}_{\bullet}\}$.
We would like to show that the summations in (\ref{EEigenEq}) and (\ref{FEigenEq}) correspond in some sense to this decomposition by the degree.
The degree of a homogeneous polynomial can be measured using $q$-shifts.
On the other hand, by the definition of the spectral variables, multiplying $s^{(p)}_{a}$ by $q$ corresponds to adding a node to the end of a row.
However, we must do this in a way that is color-insensitive.
This motivates the following:
\begin{prop}\label{ShiftProp}
Let $\lambda$ be a partition with core $\kb(\lambda)$ compatible with $N_\bullet$ and $\ell(\lambda)\le|N_\bullet|$.
Then
\[
\left(\prod_{i\in I}\prod_{l=1}^{N_i}x^{(i)}_{l}\right)P_\lambda[X_{N_\bullet}; q,t]=P_{\lambda+r^{|N_\bullet|}}[X_{N_\bullet};q,t].
\]
Here, $\lambda+r^{|N_\bullet|}$ denotes the partition obtained by adding $r$ boxes to the first $|N_\bullet|$ rows of $\lambda$.
\end{prop}

\begin{proof}
By Corollary \ref{CharCor}, $P_{\lambda+r^{|N_\bullet|}}[X_{N_\bullet};q,t]$ is characterized by the eigenvalue equations
\[
D_{p,1}(X_{N_\bullet};q,t)P_{\lambda+r^{|N_\bullet|}}[X_{N_\bullet};q,t]=\left( \sum_{\substack{b=1\\ b-\lambda_b\equiv p+1}}^{|N_\bullet|} q^{\lambda_b+r}t^{|N_\bullet|-b}\right)P_{\lambda+r^{|N_\bullet|}}[X_{N_\bullet};q,t]
\]
ranging over all $p\in I$.
Note that we have used $b-\lambda_b\equiv b-\lambda_b+r$.
Now, for a shift pattern $\Ju$, it is easy to see that
\begin{equation}
T_{\Ju}\left(\prod_{i\in I}\prod_{l=1}^{N_i}x^{(i)}_{l}\right)=q^r\left(\prod_{i\in I}\prod_{l=1}^{N_i}x^{(i)}_{l}\right)
\label{SpecShift}
\end{equation}
from which the proposition follows.
\end{proof}

\subsubsection{Eigenfunction equation}
We are now ready to derive the eigenvalues of the higher order wreath Macdonald operators.

\begin{thm}
For $\lambda$ with core $\kb(\lambda)$ compatible with $N_\bullet$ according to \eqref{Compatibility} and $\ell(\lambda)\le|N_\bullet|$, the wreath Macdonald polynomial $P_\lambda[X_{N_\bullet};q,t]$ satisfies the equations:
\begin{align*}
D_{p,n}^*(X_{N_\bullet};q,t)P_\lambda[X_{N_\bullet};q,t] &= e_n\left[ \sum_{\substack{b=1\\ b-\lambda_b\equiv p+1}}^{|N_\bullet|}q^{-\lambda_b}t^{-|N_\bullet|+b} \right]
P_\lambda[X_{N_\bullet};q,t]\\
D_{p,n}(X_{N_\bullet};q,t)P_\lambda[X_{N_\bullet};q,t] &= e_n\left[ \sum_{\substack{b=1\\ b-\lambda_b\equiv p+1}}^{|N_\bullet|}q^{\lambda_b}t^{|N_\bullet|-b} \right]
P_\lambda[X_{N_\bullet};q,t].
\end{align*}
Here, $x_{i,l}$, $q$, and $t$ take generic values.
\label{ThmEigenHigher}
\end{thm}

\begin{proof}
Let $\mathfrak{e}_{p,n}(\lambda;q,t^{-1})$ and $\mathfrak{e}_{p,n}^*(\lambda;q,t^{-1})$ be the eigenvalues of $D_{p,n}(X_{N_\bullet};q,t^{-1})$ and $D_{p,n}^*(X_{N_\bullet};q,t^{-1})$, respectively, at $P_\lambda[X_{N_\bullet};q,t^{-1}]$.
Combining (\ref{EEigenEq}), (\ref{FEigenEq}), and Lemma \ref{EigenStab}, we have
\begin{align}
\nonumber
    &t^{-n|N_\bullet|}\mathfrak{e}_{p,n}^*(\lambda;q,t^{-1})+\sum_{k=0}^{n-1}c_{p,k,n}^+ \mathfrak{e}_{p,k}^*(\lambda;q,t^{-1})\\
\label{MinusSpec2}
    &=\sum_{k=0}^n\frac{t^{-n|N_\bullet|-(n-k)(p+1)-r\binom{n-k}{2}}}{\displaystyle\prod_{l=1}^{n-k}(1-t^{-rl})}e_k\left[ \sum_{a=1}^{N_p}t^{|N_\bullet|}\left(s^{(p)}_{a}\right)^{-1} \right]
\end{align}
and
\begin{align}
\nonumber
    &t^{n|N_\bullet|}\mathfrak{e}_{p,n}(\lambda;q,t^{-1}) +\sum_{k=0}^{n-1}c_{p,k,n}^-\mathfrak{e}_{p,n}(\lambda;q,t^{-1})\\
\label{PlusSpec2}
    &=\sum_{k=0}^n\frac{t^{n|N_\bullet|+(n-k)(p+1)+r\binom{n-k}{2}}}{\displaystyle\prod_{l=1}^{n-k}(1-t^{rl})}e_k\left[ \sum_{a=1}^{N_p}t^{-|N_\bullet|}s^{(p)}_{a} \right]
\end{align}
We can induct on $n$ to show that, as functions of $\lambda$, $\mathfrak{e}_{p,n}(\lambda ;q,t^{-1})$ is polynomial in $\{ s_{\bullet}^{(p)} \}$ and $\mathfrak{e}_{p,n}^*(\lambda ;q,t^{-1})$ is polynomial in $\{ (s^{(p)}_{\bullet})^{-1} \}$.
Applying (\ref{SpecShift}) $n$ times, we have (when viewed as operators):
\begin{align*}
D_{p,n}^*(X_{N_\bullet};q,t^{-1})\prod_{i\in I}\prod_{l=1}^{N_i}x^{(i)}_{l}=q^{-nr}\prod_{i\in I}\prod_{l=1}^{N_i}x^{(i)}_{l}D_{p,n}^*(X_{N_\bullet};q,t^{-1})\\
D_{p,n}(X_{N_\bullet};q,t^{-1})\prod_{i\in I}\prod_{l=1}^{N_i}x^{(i)}_{l}=q^{nr}\prod_{i\in I}\prod_{l=1}^{N_i}x^{(i)}_{l}D_{p,n}(X_{N_\bullet};q,t^{-1}).
\end{align*}
It then follows from Proposition \ref{ShiftProp} that $\mathfrak{e}_{p,n}(\lambda ;q,t^{-1})$ is homogeneous of degree $n$ and $\mathfrak{e}_{p,n}^*(\lambda;q,t^{-1})$ is homogeneous of degree $-n$.
Thus, $t^{-n|N_\bullet|}\mathfrak{e}_{p,n}^*(\lambda;q,t^{-1})$ is the degree $-n$ piece of (\ref{MinusSpec2}) and $t^{n|N_\bullet|}\mathfrak{e}_{p,n}(\lambda ;q,t^{-1})$ is the degree $n$ piece of (\ref{PlusSpec2}).
This establishes the eigenvalue equations under the appropriate conditions (\ref{EExpand}) and (\ref{FExpand}) on $x^{(i)}_{l}$, $q$, and $t$.
We extend to generic values as in the proof of Theorem \ref{Deg1Thm}.
\end{proof}

\begin{rem}\label{Rmk2Colors}
Even though $r\ge 3$ was assumed throughout, we have verified experimentally that Theorem~\ref{ThmEigenHigher} continues to hold as stated for $r=2$. 
The $r=1$ case is discussed in Remark~\ref{MacCompare} above.
\end{rem}

\begin{ex}
Let $r=2$, $p=1$, $N_\bullet=(0,2)$, and $\lambda=(1)$.
Because $\lambda$ is a $2$-core,
\[
P_\lambda[X_{N_\bullet};q,t]=1
\]
There are only two shift patterns containing $1$:
\begin{align*}
    \Ju_1&=\{x_1^{(1)}\}\\
    \Ju_2&=\{x_2^{(1)}\}.
\end{align*}
Note that
\begin{align*}
    T_{\Ju_1}x_1^{(1)}&=q^2x_1^{(1)} & T_{\Ju_2}x_1^{(1)}&=x_1^{(1)}\\
    T_{\Ju_1}x_2^{(1)}&=x_1^{(1)} & T_{\Ju_2}x_2^{(1)}&=q^2x_2^{(1)}.
\end{align*}
Therefore,
\begin{align*}
    D_{1,1}(X_{N_\bullet};q,t)P_\lambda[X_{\bullet};q,t]&=
    \frac{(-1)(1-qt^{-1})}{1-q^2t^{-2}}
    \left\{
    \frac{qtx_2^{(1)}-q^2x_1^{(1)}}{x_1^{(1)}-x_2^{(1)}}+\frac{qtx_1^{(1)}-q^2x_2^{(1)}}{x_2^{(1)}-x_1^{(1)}}
    \right\}\\
    &=\frac{(-1)(1-qt^{-1})(-qt-q^2)}{1-q^2t^{-2}}\\
    &=qtP_\lambda[X_{N_\bullet};q,t].
\end{align*}
\end{ex}

\appendix

\section{Wreath Noumi-Sano operators}
In this appendix, we apply our methods to study wreath analogues of the trigonometric \textit{Noumi-Sano operators} \cite{NSa}.
We obtain explicit formulas for degree $n=1$ and an integral formula for general $n$.

\subsection{Infinite-variable eigenvalues}
Let $(x;y)_\infty$ denote the infinite $y$-Pochammer symbol:
\[
(x;y)_\infty=\prod_{i=0}^{\infty}(1-xy^i).
\]

\begin{lem}\label{NSFockEigen}
Assume $|q^{\pm 1}|<1$ and $|t^{\pm 1}|<1$ (where `$+$' and `$-$' are separate cases).
For $p\in I$, we have
\begin{equation}
\begin{aligned}
&\left\langle\lambda\left|\exp\left[-\sum_{k>0}\left(\frac{\sum_{i=1}^{r}q^{\pm k(i-1)}h_{p+i,\pm k}}{(1-q^{\pm kr})}\right)\frac{\qqq^{\pm k} z^{\mp k}}{v^{\pm k}[ k]_{\qqq}} \right]\right|\lambda\right\rangle\\
&= \exp\left[\sum_{k>0}\left(\sum_{i=1}^{r}\frac{q^{\pm k(i-1)}}{1-q^{\pm kr}}\left\{\sum_{\substack{b>0\\b-\lambda_b\equiv p+i }}q^{\pm k\lambda_b}t^{\pm kb}-t^{\mp k}\sum_{\substack{b>0\\b-\lambda_b\equiv p+i+1 }}q^{\pm k\lambda_b}t^{\pm kb}\right\}\right)\frac{z^{\mp k}}{k}\right]\\
&= \prod_{i=1}^{r}
\frac{\displaystyle\prod_{\substack{b>0\\b-\lambda_b\equiv p+i+1 }}\left(q^{\pm(\lambda_b+i)}t^{\pm (b- 1)}z^{\mp 1}; q^{\pm r}\right)_\infty}
{\displaystyle\prod_{\substack{b>0\\b-\lambda_b\equiv p+i }}\left(q^{\pm(\lambda_b+i)}t^{\pm b}z^{\mp 1}; q^{\pm r}\right)_\infty}
\end{aligned}
\label{NSEigenLemma}
\end{equation}
where we set $\lambda_b=0$ for all $b>\ell(\lambda)$.
\end{lem}

\subsection{Shuffle elements}
We rewrite
\begin{align*}
&\varsigma\exp\left[-\sum_{k>0}\left(\frac{\sum_{i=1}^{r}q^{\pm k(i-1)}h_{p+i,\pm k}}{(1-q^{\pm kr})}\right)\frac{\qqq^{\pm k} z^{\mp k}}{[ k]_{\qqq}} \right]\\
&= 
\varsigma\exp\left[\sum_{k>0}\left(\frac{\sum_{i=0}^{r-1}q^{\mp k(i+1)}h_{p-i,\pm k}}{(1-q^{\mp kr})}\right)\frac{\qqq^{\pm k} z^{\mp k}}{[ k]_{\qqq}} \right]\\
&=\eta\exp\left[-\sum_{k>0}\left(\frac{\sum_{i=0}^{r-1}t^{\mp k(i+1)}\varsigma^{-1}(h_{p-i,\mp k})}{(1-t^{\mp kr})}\right)\frac{\qqq^{\pm k} z^{\mp k}}{[ k]_{\qqq}} \right]\\
&= \eta\exp\left[(\qqq-\qqq^{-1})^{-1} \sum_{k>0}\left(-q^{\pm k}\varsigma^{-1}(h_{p,\mp k}^\perp)+\varsigma^{-1}(h_{p+1,\mp k}^\perp) \right)\frac{z^{\mp k}}{k} \right].
\end{align*}
Recall the formulas (\ref{Epn}) and (\ref{Hpn}) for $\mathcal{E}_{p,n}^\pm$ and $\mathcal{H}_{p,n}^\pm$.
In \cite{Wen}, it was shown that
\begin{align*}
&\exp\left[(\qqq-\qqq^{-1})^{-1} \sum_{k>0}\left(-q^{- k}\varsigma^{-1}(h_{p, k}^\perp)+\varsigma^{-1}(h_{p+1, k}^\perp) \right)\frac{z^{ k}}{k} \right]\\
&= \sum_{n=0}^\infty\frac{(-1)^{nr}t^{nr}\ddd^{-n}\left( 1-q^{-1}t^{-1} \right)^{nr}}{\qqq^{2n}\prod_{r=1}^n\left( 1-q^{-r}t^{-r} \right)}\Psi_+\left( \mathcal{E}_{p,n}^-\right)\\
&\exp\left[(\qqq-\qqq^{-1})^{-1} \sum_{k>0}\left(-q^{ k}\varsigma^{-1}(h_{p,- k}^\perp)+\varsigma^{-1}(h_{p+1,- k}^\perp) \right)\frac{z^{- k}}{k} \right]\\
&= \sum_{n=0}^\infty\frac{\ddd^n\left( 1-qt \right)^{nr}}{\prod_{r=1}^n\left( 1-q^{-r}t^{-r} \right)}\Psi_-\left( \mathcal{E}_{p,n}^+\right).
\end{align*}
Applying $\eta$, we get
\begin{align*}
\sum_{n=0}^\infty\frac{\qqq^{n(r-1)}t^{-n}\left( 1-q^{-1}t^{-1} \right)^{nr}}{v^{-n}\prod_{r=1}^n\left( 1-q^{-r}t^{-r} \right)}
\Psi_+(\mathcal{H}^-_{p,n})
&= \varsigma\exp\left[-\sum_{k>0}\left(\frac{\sum_{i=0}^{r-1}q^{- ki}h_{p+i,- k}}{(1-q^{- kr})}\right)\frac{\qqq^{- k} z^{ k}}{v^{-k}[ k]_{\qqq}} \right],\\
\sum_{n=0}^\infty\frac{(-1)^{nr}\ddd^{-n(r-1)}t^{n}\left(1- qt \right)^{nr}}{v^nq^{n}\prod_{r=1}^n\left(1- q^{-r}t^{-r} \right)}
\Psi_-\left( \mathcal{H}_{p,n}^+ \right)
&=\varsigma\exp\left[-\sum_{k>0}\left(\frac{\sum_{i=0}^{r-1}q^{ ki}h_{p+i, k}}{(1-q^{ kr})}\right)\frac{\qqq^{ k} z^{- k}}{v^k[ k]_{\qqq}} \right].
\end{align*}

\subsection{Normal ordering}
It will be slightly nicer to reorder our currents differently from Proposition \ref{NormalOrder}:

\begin{prop}\label{NSNormalOrder}
For $p\in I$, we have
\begin{align*}
&\overset{\curvearrowright}{\prod_{a=1}^n}\,\overset{\curvearrowleft}{\prod_{i=1}^{r}}E_{p+i}(z_{p+i,a})\\
&=
\left( (-1)^{\frac{(r-2)(r-3)}{2}+r}\ddd^{-\frac{r}{2}+1}\prod_{i\in I} c_i\right)^n\\
&\times\prod_{1\le a<b\le n}\prod_{i\in I}
\frac{\displaystyle\left(1-z_{i,b}/z_{i,a}\right)\left(1-q^{-1}t^{-1}z_{i,b}/z_{i,a}\right)}{\displaystyle\left(1-t^{-1}z_{i+1,b}/z_{i,a}\right)\left(1-q^{-1}z_{i-1,b}/z_{i,a}\right)}\\
&\times\prod_{a=1}^n \frac{z_{p+1,a}/z_{p,a}}
{\left(1-t^{-1}z_{p+1,a}/z_{p,a}\right)
\prod_{i\in I\setminus\{p\}}
\left(1-q^{-1}z_{i,a}/z_{i+1,a}\right)}\\
&\times\prod_{i\in I}\exp\left(\sum_{a=1}^n\sum_{k>0}\left(p_{k}[X^{(i)}]-t^{-k}p_k[X^{(i-1)}]\right)\frac{z_{i,a}^{k}}{k}\right)\\
&\times\prod_{i\in I}\exp\left(\sum_{a=1}^n\sum_{k>0}\left(-p_{k}[X^{(i)}]^\perp+q^{-k}p_k[X^{(i-1)}]^\perp\right)\frac{z_{i,a}^{-k}}{k}\right)
\prod_{i\in I}\prod_{a=1}^n z_{i,a}^{H_{i,0}}
\end{align*}
where all rational functions are Laurent series expanded assuming 
\begin{equation}
|z_{i,a}|=1,\,
|q|>1,\,
|t|>1.
\end{equation}
For the $F$-currents, we have
\begin{align*}
&\overset{\curvearrowleft}{\prod_{a=1}^n}\,\overset{\curvearrowright}{\prod_{i=1}^{r}}F_{p+i}(z_{p+i,a})\\
&=\left( \frac{(-1)^{\frac{(r-2)(r-3)}{2}+r}\ddd^{\frac{r}{2}-1}}{\displaystyle\prod_{i\in I}c_i}\right)^n\\
&\times\prod_{1\le a<b\le n}\prod_{i\in I}
\frac{\left(1-z_{i,a}/z_{i,b}\right)\left(1-qtz_{i,a}/z_{i,b}\right)}{\left(1-tz_{i-1,a}/z_{i,b}\right)\left(1-qz_{i+1,a}/z_{i,b}\right)}\\
&\times \prod_{a=1}^n\frac{z_{p,a}/z_{p+1,a}}{\left(1-tz_{p,a}/z_{p+1,a}\right)
\prod_{i\in I\setminus\{p+1\}}
\left(1-qz_{i,a}/z_{i-1,a}\right)}\\
&\times\prod_{i\in I}\exp\left(\sum_{a=1}^n\sum_{k>0}\left(-t^{k}p_{k}[X^{(i)}]+p_k[X^{(i-1)}]\right)\frac{z_{i,a}^{k}}{k}\right)\\
&\times\prod_{i\in I}\exp\left(\sum_{a=1}^n\sum_{k>0}\left(q^{k}p_{k}[X^{(i)}]^\perp-p_k[X^{(i-1)}]^\perp\right)\frac{z_{i,a}^{-k}}{k}\right)
\prod_{i\in I}\prod_{a=1}^n z_{i,a}^{-H_{i,0}}
\end{align*}
where all rational functions are Laurent series expanded assuming 
\begin{equation}
|z_{i,a}|=1,\,
|q|< 1,\,
|t|<1.   
\end{equation}
\end{prop}

\subsection{Integral formula}
Let
\begin{align*}
d_n^+&=\frac{\qqq^{n(r-1)}t^{-n}\left( 1-q^{-1}t^{-1} \right)^{nr}}{v^{-n}\prod_{r=1}^n\left( 1-q^{-r}t^{-r} \right)},&
d_n^-&= \frac{(-1)^{nr}\ddd^{-n(r-1)}t^{n}\left(1- qt \right)^{nr}}{v^nq^{n}\prod_{r=1}^n\left(1- q^{-r}t^{-r} \right)}.
\end{align*}
We have
\begin{align*}
d_n^+\pi_{N_\bullet}\biggl( \left( \rho_{\vec{c}}\circ\Psi_+ \right)(\mathcal{H}_{p,n}^-)(f\otimes e^\alpha)\biggr)
&=\frac{(-1)^n\left(1- q^{-1}t^{-1} \right)^{nr}}{\prod_{r=1}^n\left(1- q^{-r}t^{-r} \right)}
\left\{\prod_{i\in I}\prod_{a=1}^n\prod_{l=1}^{N_i}\left(\frac{ z_{i+1,a}^{-1}-t^{-1}x_{l}^{(i)} }{ z_{i,a}^{-1}-x_{l}^{(i)} }\right)\right.\\
&\times\prod_{1\le a<b\le n}
\left[
\frac{\left(1-z_{p,b}/z_{p,a}\right)\left(1-q^{-1}t^{-1}z_{p,b}/z_{p,a}\right)}{\left(1-qz_{p+1,b}/z_{p,a}\right)\left(1-q^{-1}z_{p-1,b}/z_{p,a}\right)}\right.\\
\nonumber
&\times\left.\prod_{i\in I\backslash\{p\}}
\frac{\left(1-z_{i,b}/z_{i,a}\right)\left(1-q^{-1}t^{-1}z_{i,b}/z_{i,a}\right)}{\left(1-q^{-1}z_{i-1,b}/z_{i,a}\right)\left(1-t^{-1}z_{i+1,b}/z_{i,a}\right)}\right]\\
&\times\prod_{a=1}^n \left[
\left(\frac{z_{0,a}}{z_{p+1,a}}\right)\left(\frac{1}{1-q^{-1}z_{p,a}/z_{p+1,a}}\right)\right.\\
\nonumber
&\times\left.
\prod_{i\in I\setminus\{p\}}
\left(\frac{1}{1-q^{-1}z_{i,a}/z_{i+1,a}}\right)
\left( \frac{1}{1-t^{-1}z_{i+1,a}/z_{i,a}} \right)\right]\\
&\times 
\prod_{i\in I}f_i\left[ \sum_{l=1}^{N_i}x_{l}^{(i)}-\sum_{a=1}^nz_{i,a}^{-1}+q^{-1}\sum_{a=1}^n z_{i+1,a}^{-1} \right]\Bigg\}_0\otimes e^\alpha 
\end{align*}
and
\begin{align*}
d_n^-\pi_{N_\bullet}\biggl( \left( \rho_{\vec{c}}\circ\Psi_- \right)(\mathcal{H}_{p,n}^+)(f\otimes e^\alpha)\biggr)
&= \frac{(-1)^n\left(1- qt \right)^{nr}}{\prod_{r=1}^n\left(1- q^{r}t^{r} \right)}
\left\{\prod_{i\in I}\prod_{a=1}^n\prod_{l=1}^{N_i}\left(\frac{  z_{i,a}^{-1}-tx_{l}^{(i)} }{ z_{i+1,a}^{-1}-x_{l}^{(i)} }\right)\right.\\
&\times\prod_{1\le a<b\le n}\left[
\frac{\left(1-z_{p+1,a}/z_{p+1,b}\right)\left(1-qtz_{p+1,a}/z_{p+1,b}\right)}{\left(1-q^{-1}z_{p,a}/z_{p+1,b}\right)\left(1-qz_{p-2,a}/z_{p+1,b}\right)}\right.\\
&\times\left.
\prod_{i\in I\backslash\{p+1\}}
\frac{\left(1-z_{i,a}/z_{i,b}\right)\left(1-qtz_{i,a}/z_{i,b}\right)}{\left(1-qz_{i+1,a}/z_{i,b}\right)\left(1-tz_{i-1,a}/z_{i,b}\right)}\right]\\
&\times \prod_{a=1}^n
\left[
\left(\frac{z_{p+1,a}}{z_{0,a}}\right)
\left(\frac{1}
{1-qz_{p+1,a}/z_{p,a}}\right)\right.\\
&\left.\times\prod_{i\in I\setminus\{p+1\}}
 \left(\frac{1}{1-qz_{i,a}/z_{i-1,a}}\right)\left(\frac{1}{1-tz_{i-1,a}/z_{i,a}}\right) \right]\\
&\left.\times 
\prod_{i\in I}
f_i\left[ \sum_{l=1}^{N_i}  x_{l}^{(i)}+\sum_{a=1}^n qz_{i,a}^{-1}-\sum_{a=1}^n z_{i+1,a}^{-1} \right]\right\}_0\otimes e^\alpha.
\end{align*}

Finally, we make the substitution $w_{i,a}=z_{i,a}^{-1}$ and rewrite these formulas in terms of integrals.
This gets us
\begin{align*}
&d_n^+\pi_{N_\bullet}\biggl( \left( \rho_{\vec{c}}\circ\Psi_+ \right)(\mathcal{H}_{p,n}^-)(f\otimes e^\alpha)\biggr)\\
&=\underset{|w_{i,a}|=1}{\oint\cdots\oint} \bigg(\varrho_{p,n}^+(w_{\bullet,\bullet},X_{N_\bullet})\prod_{i\in I}
f_i\left[ \sum_{l=1}^{N_i} x_{l}^{(i)}-\sum_{a=1}^nw_{i,a}+\sum_{a=1}^nq^{-1}w_{i+1,a} \right]
\prod_{a=1}^n \frac{dw_{i,a}}{2\pi\sqrt{-1}w_{i,a}}
\bigg)\otimes e^\alpha 
\end{align*}
and 
\begin{align*}
&d_n^-\pi_{N_\bullet}\biggl( \left( \rho_{\vec{c}}\circ\Psi_- \right)(\mathcal{H}_{p,n}^+)(f\otimes e^\alpha)\biggr)\\
&=\underset{|w_{i,a}|=1}{\oint\cdots\oint}\bigg(\varrho_{p,n}^-(w_{\bullet,\bullet},X_{N_\bullet})\prod_{i\in I}f_i\left[ \sum_{l=1}^{N_i} x_{l}^{(i)}+\sum_{a=1}^nqw_{i,a}-\sum_{a=1}^nw_{i+1,a} \right]
\prod_{a=1}^n \frac{dw_{i,a}}{2\pi\sqrt{-1}w_{i,a}}
\bigg)\otimes e^\alpha
\end{align*}
where
\begin{align*}
\varrho_{p,n}^+(w_{\bullet,\bullet},X_{N_\bullet})
&= \frac{(-1)^{\frac{n(n+1)}{2}}\left(1- q^{-1}t^{-1} \right)^{nr}}{q^{\frac{n(n-1)}{2}}\prod_{r=1}^n\left(1- q^{-r}t^{-r} \right)}
\left\{\prod_{i\in I}\prod_{a=1}^n\prod_{l=1}^{N_i}\left(\frac{ w_{i+1,a}-t^{-1}x_{l}^{(i)} }{ w_{i,a}-x_{l}^{(i)} }\right)\right.\\
&\times\prod_{1\le a<b\le n}
\left[
\frac{\left(w_{p,b}-w_{p,a}\right)\left(w_{p,b}-q^{-1}t^{-1}w_{p,a}\right)}{\left(w_{p,a}-q^{-1}w_{p+1,b}\right)\left(w_{p-1,b}-q^{-1}w_{p,a}\right)}\right.\\
\nonumber
&\times\left.\prod_{i\in I\backslash\{p\}}
\frac{\left(w_{i,b}-w_{i,a}\right)\left(w_{i,b}-q^{-1}t^{-1}w_{i,a}\right)}{\left(w_{i-1,b}-q^{-1}w_{i,a}\right)\left(w_{i+1,b}-t^{-1}w_{i,a}\right)}\right]\\
&\times\prod_{a=1}^n \left[
\left(\frac{w_{p+1,a}}{w_{0,a}}\right)\left(\frac{w_{p,a}}{w_{p,a}-q^{-1}w_{p+1,a}}\right)\right.\\
\nonumber
&\times\left.
\prod_{i\in I\setminus\{p\}}
\left(\frac{w_{i,a}}{w_{i,a}-q^{-1}w_{i+1,a}}\right)
\left( \frac{w_{i+1,a}}{w_{i+1,a}-t^{-1}w_{i,a}} \right)\right]
\end{align*}
and
\begin{align*}
\varrho_{p,n}^-(w_{\bullet,\bullet},X_{N_\bullet})
&= \frac{(-1)^{\frac{n(n+1)}{2}}q^{\frac{n(n-1)}{2}}\left(1- qt \right)^{nr}}{\prod_{r=1}^n\left(1- q^{r}t^{r} \right)}
\left\{\prod_{i\in I}\prod_{a=1}^n\prod_{l=1}^{N_{i-1}}\left(\frac{  w_{i-1,a}-tx_{l}^{(i-1)} }{ w_{i,a}-x_{l}^{(i-1)} }\right)\right.\\
&\times\prod_{1\le a<b\le n}\left[
\frac{\left(w_{p+1,a}-w_{p+1,b}\right)\left(w_{p+1,a}-qtw_{p+1,b}\right)}{\left(w_{p+1,b}-qw_{p,a}\right)\left(w_{p+2,a}-qw_{p+1,b}\right)}\right.\\
&\times\left.
\prod_{i\in I\backslash\{p+1\}}
\frac{\left(w_{i,a}-w_{i,b}\right)\left(w_{i,a}-qtw_{i,b}\right)}{\left(w_{i+1,a}-qw_{i,b}\right)\left(w_{i-1,a}-tw_{i,b}\right)}\right]\\
&\times \prod_{a=1}^n
\left[
\left(\frac{w_{0,a}}{w_{p+1,a}}\right)
\left(\frac{w_{p+1,a}}{w_{p+1,a}-qw_{p,a}}\right)\right.\\
&\left.\times\prod_{i\in I\setminus\{p+1\}}
 \left(\frac{w_{i,a}}{w_{i,a}-qw_{i-1,a}}\right)\left(\frac{w_{i-1,a}}{w_{i-1,a}-tw_{i,a}}\right) \right].
\end{align*}
\subsection{Degree one}
We compute the integral and record the resulting action on $f$ when $n=1$.

\subsubsection{Difference operators} Let $Sh(X_{N_\bullet})=Sh$ denote the set of \textit{all} shift patterns. Define
\begin{align*}
H_{p,1}^*(X_{N_\bullet};q,t^{-1})
&:= 
-\sum_{\substack{\Ju\in Sh\\\Ju\not=\varnothing}}(1-q^{-1}t^{-1})^{|J|-\delta_{p\in J}}
\frac{x^{(p+1)}_{\Ju^{\triangle}}}{x^{(0)}_{\Ju^{\triangle}}}
\left( \prod_{i\in I}\prod_{\substack{l=1\\ x^{(i)}_{l}\not=x^{(i)}_{\Ju^\triangle}}}^{N_i}\frac{\left(x_{l}^{(i)}-tx^{(i+1)}_{\Ju^\triangle}\right)}{\left(x_{l}^{(i)}-x^{(i)}_{\Ju^\triangle}\right)} \right)\\
&\times
\left( \frac{qtT_{\Ju}^{-1}(x^{(p)}_{\Ju})-x^{(p)}_{\Ju}}{x^{(p)}_{\Ju}-T_{\Ju}^{-1}(x^{(p)}_{\Ju})} \right)^{\delta_{p\in J}}
\left(\prod_{i\in J\backslash\{p\}}\frac{qtT_{\Ju}^{-1}(x^{(i)}_{\Ju})}{\left(x^{(i)}_{\Ju}-T_\Ju^{-1} (x^{(i)}_{\Ju})\right)}\right)
T_{\Ju}^{-1}\\
H_{p,1}(X_{N_\bullet};q,t^{-1})
&:=
-\sum_{\substack{\Ju\in Sh\\\Ju\not=\varnothing}} (1-qt)^{|J|-\delta_{p\in J}}
\frac{x^{(r-1)}_{\Ju^{\triangledown}}}{x^{(p)}_{\Ju^{\triangledown}}}
\left(\prod_{i\in I}\prod_{\substack{l=1\\ x^{(i)}_{l}\not= x^{(i)}_{\Ju^{\triangledown}}}}^{N_i}
\frac{\left(x^{(i)}_{l}-t^{-1}x^{(i-1)}_{\Ju^{\triangledown}}\right)}{\left(x^{(i)}_{l}-x^{(i)}_{\Ju^{\triangledown}}\right)}\right)\\
&\times\left( \frac{q^{-1}t^{-1}T_{\Ju}(x^{(p)}_{\Ju})-x^{(p)}_{\Ju}}{x^{(p)}_{\Ju}-T_{\Ju}(x^{(p)}_{\Ju})} \right)^{\delta_{p\in J}}
\left(\prod_{i\in J\backslash\{p\}}\frac{q^{-1}t^{-1}T_{\Ju}(x^{(i)}_{\Ju})}{\left(x^{(i)}_{\Ju}-T_\Ju (x^{(i)}_{\Ju})\right)}\right)
T_{\Ju}.
\end{align*}
Setting $r=1$ and inverting $t$, we indeed obtain the first Noumi-Sano operator.

\subsubsection{Eigenvalues}
For a series $f(z)$ in $z$, let $[z^n] f(z)$ denote the coefficient of $z^n$.
Methods similar to those in \ref{Eigen} allow us to establish the following.

\begin{thm}
For $|q|>1$, we have
\begin{align}
&H^*_{p,1}(X_{N_\bullet};q,t)P_\lambda[X_{N_\bullet};q,t]\\
&= [z]\left( \prod_{i=1}^{r}
\frac{\displaystyle\prod^{|N_\bullet|}_{\substack{b=1\\b-\lambda_b\equiv p+i+1 }}\left(q^{-(\lambda_b+i)}t^{-|N_\bullet|+ (b- 1)}z; q^{- r}\right)_\infty}
{\displaystyle\prod^{|N_\bullet|}_{\substack{b=1\\b-\lambda_b\equiv p+i }}\left(q^{-(\lambda_b+i)}t^{-|N_\bullet| +b}z; q^{- r}\right)_\infty} \right)P_\lambda[X_{N_\bullet};q,t].
\end{align}
On the other hand, for $|q|<1$, we have
\begin{align}
&H_{p,1}(X_{N_\bullet};q,t)P_\lambda[X_{N_\bullet};q,t]\\
&= [z^{-1}]\left( \prod_{i=1}^{r}
\frac{\displaystyle\prod^{|N_\bullet|}_{\substack{b=1\\b-\lambda_b\equiv p+i+1 }}\left(q^{\lambda_b+i}t^{|N_\bullet|- (b- 1)}z^{-1}; q^{ r}\right)_\infty}
{\displaystyle\prod^{|N_\bullet|}_{\substack{b=1\\b-\lambda_b\equiv p+i }}\left(q^{\lambda_b+i}t^{|N_\bullet| -b}z^{-1}; q^{ r}\right)_\infty} \right)P_\lambda[X_{N_\bullet};q,t].   
\end{align}
\end{thm}
\begin{rem} 
We have presented the eigenvalues in terms of our original spectral variables $q^{\lambda_b}t^{|N_\bullet|-b}$.
However, we can give a more natural combinatorial expression for the eigenvalues if we forgo this and use instead the \textit{transpose} partition $\lambda'$ \cite[(I.1.3)]{Mac}.
Let
\begin{align}
f_\lambda(q,t) = \dfrac{1}{1-q} - \sum_{j\ge 1} q^{j-1} t^{|N_\bullet|-\lambda'_j}.
\end{align}
It can be viewed as a series or as a rational function since 
$(1-q^r) f_\lambda(q,t)$ is a polynomial. Let $\Gamma=\Z/r\Z$ be the cyclic group and let $\chi$ be the generator of $R(\Gamma)$. Define $f_\lambda^{(p)}(q,t)$ by the following expression in 
$\Q(q,t) \otimes R(\Gamma)$.
\begin{align}
f_\lambda(q\chi^{-1},t\chi^{-1}) = \chi^{-1}\sum_{p\in I} f^{(p)}_\lambda(q,t) \chi^p.
\end{align}
Then the eigenvalues are given by 
\begin{align}
H^*_{p,1}(X_{N_\bullet};q,t) P_\la[X_{N_\bullet};q,t] &= f_\lambda^{(p)}(q^{-1},t^{-1}) P_\la[X_{N_\bullet};q,t] \\
H_{p,1}(X_{N_\bullet};q,t) P_\la[X_{N_\bullet};q,t] &= f_\lambda^{(p)}(q,t) P_\la[X_{N_\bullet};q,t].
\end{align}
\end{rem}

\begin{ex} Let $r=2$ and $\alpha=0$ (empty core). We use $N_0=N_1=1$. 
There are three nonempty shift patterns: $\Ju_1 = \{x_1^{(0)}\}$, $\Ju_2=\{x_1^{(1)}\}$, and 
$\Ju_3 = \{x_1^{(0)},x_1^{(1)}\}$. We apply $H_{0,1}[X_{\Nd};q,t^{-1}]$ to $P_\emptyset[X_{\Nd};q,t]=1$ using summands $\Ju_1,\Ju_2,\Ju_3$:
\begin{align*}
  - H_{0,1}(X_{\Nd};q,t^{-1})\cdot P_\emptyset[X_{\Nd};q,t^{-1}] &= 
  (1-qt)^0 \dfrac{q x_1^{(0)}}{x_1^{(0)}} \dfrac{x_1^{(1)}-t^{-1} x_1^{(0)}}{x_1^{(1)}-q x_1^{(0)}}
  \dfrac{q^{-1}t^{-1} q^2 x_1^{(0)} - x_1^{(0)}}{x_1^{(0)}-q^2 x_1^{(0)}} \\
  &+ (1-qt)^1 \dfrac{x_1^{(1)}}{q x_1^{(1)}} \dfrac{x_1^{(0)}-t^{-1} x_1^{(1)}}{x_1^{(0)}-q x_1^{(1)}} 
  \dfrac{q^{-1}t^{-1}q^2 x_1^{(1)}}{x_1^{(1)}-q^2x_1^{(1)}} \\
  &+ (1-qt)^1 \dfrac{q^{-1}t^{-1} q^2 x_1^{(0)}- x_1^{(0)}}{x_1^{(0)}-q^2 x_1^{(0)}}
  \dfrac{q^{-1}t^{-1} q^2 x_1^{(1)}}{x_1^{(1)}-q^2 x_1^{(1)}}\\
  &= \dfrac{q (qt^{-1}-1)}{1-q^2} \dfrac{x_1^{(1)}-t^{-1} x_1^{(0)}}{x_1^{(1)}-qx_1^{(0)}} \\
  &+\dfrac{(1-qt)t^{-1}}{1-q^2}\dfrac{x_1^{(0)}-t^{-1}x_1^{(1)}}{x_1^{(0)}-qx_1^{(1)}}\\
  &+\dfrac{(1-qt)(qt^{-1}-1)qt^{-1}}{(1-q^2)^2} \\
  &= \dfrac{q}{t^2} \dfrac{1-t^2}{1-q^2} P_{\emptyset}[X_{\Nd};q,t^{-1}].
\end{align*}
We have
\begin{align*}
f_{\emptyset}(q,t) &= \dfrac{1}{1-q} - \dfrac{t^2}{1-q} = \dfrac{1-t^2}{1-q} \\
f_{\emptyset}(q\chi^{-1},t\chi^{-1}) &= \dfrac{1-t^2}{1-q^2} (1 + q \chi) \\
f_{\emptyset}^{(0)}(q,t) &= q \dfrac{1-t^2}{1-q^2} \\
f_{\emptyset}^{(0)}(q,t^{-1}) &= q \dfrac{1-t^{-2}}{1-q^2} = - \dfrac{q}{t^2} \dfrac{1-t^2}{1-q^2}.
\end{align*}
\end{ex}

\end{document}